\pdfoutput=1
\documentclass[12pt]{article}
\usepackage{layout}
\usepackage{url}
\usepackage{proof}
\usepackage{mathrsfs}
\usepackage{amsmath,amsthm,amssymb}
\usepackage{bbm}
\usepackage{graphicx}
\usepackage[all]{xy}
\usepackage{color}
\usepackage{indentfirst}

\usepackage{stmaryrd}
\newcommand{\shear}{\mathbin{\mkern-6mu\fatslash}}

\numberwithin{equation}{section}

\newtheorem{thm}{Theorem}[section]
\newtheorem{prop}[thm]{Proposition}
\newtheorem{lemm}[thm]{Lemma}
\newtheorem{cor}[thm]{Corollary}
\newtheorem{defi}[thm]{Definition}

\theoremstyle{definition}
\newtheorem{rmk}[thm]{Remark}

\def\fk#1{\mathfrak{#1}}
\def\cal#1{\mathcal{#1}}

\def\rm#1{\mathrm{#1}}
\def\bb#1{\mathbb{#1}}

\def\wh#1{\widehat{#1}}
\def\wt#1{\widetilde{#1}}
\def\ol#1{\overline{#1}}

\newcommand{\ds}{\displaystyle}
\newcommand{\QQ}{\mathbb{Q}}
\newcommand{\Qlb}{\overline{\mathbb{Q}}_\ell}

\newcommand{\ZZ}{\mathbb{Z}}

\newcommand{\Hom}{\mathrm{Hom}}

\newcommand{\Perv}{\mathrm{Perv}}
\newcommand{\Rep}{\mathrm{Rep}}
\newcommand{\DRep}{\mathrm{DRep}}
\newcommand{\DDRep}{\mathcal{D}\mathrm{Rep}}
\newcommand{\Vect}{\mathrm{Vect}}
\newcommand{\Catinf}[1]{\widehat{\mathrm{Cat}}_{\infty,#1}}
\newcommand{\LocSys}{\mathrm{LS}}
\newcommand{\Sat}{\mathrm{Sat}}

\newcommand{\Hck}{\mathcal{H}\mathrm{ck}}

\newcommand{\HHom}{\mathscr{H}\mathrm{om}}
\newcommand{\Div}{\mathrm{Div}}
\newcommand{\Divtil}{\widetilde{\mathrm{Div}}_0^{\perp}}
\newcommand{\ULA}{\mathrm{ULA}}
\newcommand{\perf}{\mathrm{perf}}
\newcommand{\Perf}{\mathrm{Perf}}

\newcommand{\QCoh}{\mathrm{QCoh}}
\newcommand{\Coh}{\mathrm{Coh}}
\newcommand{\et}{\mathrm{\acute{e}t}}
\newcommand{\Det}{D_{\mathrm{\acute{e}t}}}

\newcommand{\bd}{\mathrm{bd}}

\newcommand{\dia}{\diamondsuit}
\newcommand{\Sym}{\mathrm{Sym}}
\newcommand{\Witt}{\mathrm{Witt}}
\newcommand{\eq}{\mathrm{eq}}

\newcommand{\Ind}{\mathrm{Ind}}

\newcommand{\Gr}{\mathrm{Gr}}

\newcommand{\spl}{\mathrm{spl}}

\newcommand{\lmod}{{\operatorname{-mod}}}

\DeclareMathOperator{\Spec}{Spec}
\DeclareMathOperator{\Spd}{Spd}
\DeclareMathOperator{\Spa}{Spa}

\title{\Large{Derived geometric Satake equivalence on the Beilinson--Drinfeld Grassmannian with one leg in mixed characteristic}}
\author{Katsuyuki Bando\footnote{Katsuyuki Bando, Research Institute for Mathematical Sciences, Kyoto University, Kyoto, Japan, \texttt{kbando@kurims.kyoto-u.ac.jp}}}

\usepackage{cleveref}
\begin{document}
\maketitle
\abstract{
Fargues--Scholze developed a framework for the geometric Langlands program on the Fargues--Fontaine curve.
In particular, they proved the geometric Satake equivalence on the moduli space of closed Cartier divisors on the curve.
We prove the derived version of this equivalence with one leg.
Namely, we show that the derived category of \'{e}tale sheaves on the local Hecke stack is equivalent to the category of L-group-equivariant perfect complexes over the symmetric algebra of the shifted and $(-1)$-Tate-twisted Lie algebra.
}
\section{Introduction}
The geometric Satake equivalence is a symmetric monoidal equivalence between the category of perverse sheaves on the Hecke stack of a reductive group $G$ and the category of finite-dimensional algebraic representations of the dual group of $G$.
Since it is an equivalence between a geometric category and a representation-theoretic category, it is important in geometric representation theory, especially in the geometric Langlands program.
Explicitly, for a discrete valuation field $F$ whose residue field $k$ is algebraically closed and for a reductive group $G$ over $\cal{O}_F$, we can define the Hecke stack $\Hck_{G,\Spec k}$, which is the geometrization of the double quotient of $G(F)$ by $G(\cal{O}_F)$, and there is an equivalence of the form
\[
\Perv(\Hck_{G,\Spec k},\Qlb)\simeq \Rep^{\rm{f.d.}}(\wh{G},\Qlb).
\]
This is proved in \cite{MV} when $F$ is of equal characteristic, and in \cite{Zhumixed} when $F$ is mixed.

The derived versions of these equivalences, called ``the derived geometric Satake equivalences'' are proved by \cite{BF} in equal characteristic and by \cite{Ban} in mixed characteristic:
\[
D^{\ULA}(\Hck_{G,\Spec k},\Qlb)\simeq \rm{Sym}(\wh{\fk{g}}[-2])\lmod^{\perf}(D(\Rep(\widehat{G},\Qlb))).
\]
Here we regard $\rm{Sym}(\wh{\fk{g}}[-2])$ as a formal dg-algebra, and the right-hand side is the category of $\rm{Sym}(\wh{\fk{g}}[-2])$-module objects $M$ in the derived category of representations of $\wh{G}$, such that $M$ is a perfect complex over $\rm{Sym}(\wh{\fk{g}}[-2])$.

There is a factorizable version of the derived geometric Satake equivalence, in which we use a Hecke stack over a smooth projective curve $X$ instead of a Hecke stack over a point.
This equivalence is proved in \cite{CR}.
For the geometric local Langlands program for a local field $F$, the Fargues--Fontaine curve can be used instead of an algebraic curve $X$ as in \cite{FS}.
In this situation, we use a Hecke stack denoted by $\Hck_{G,(\Div^1)^I}$ defined in \cite[{\sc Definition} VI.1.6]{FS}.
In \cite[CHAPTER VI]{FS}, they proved the (non-derived) geometric Satake equivalence over $(\Div^1)^I$:
\[
\Perv(\Hck_{G,(\Div^1)^I},\Qlb)\simeq \Rep^{\rm{f.d.}}((\wh{G}\rtimes W_{F})^I,\Qlb),
\]
where $W_{F}$ is the Weil group of $F$, and the semiproduct $\widehat{G}\rtimes W_{F}$ arises from the usual algebraic $W_{F}$-action $\widehat{G}$ twisted by the cyclotomic action (see \S \ref{ssc:dualgp}).
In particular, when $|I|=1$, they have
\[
\Perv(\Hck_{G,\Div^1},\Qlb)\simeq \Rep^{\rm{f.d.}}(\wh{G}\rtimes W_{F},\Qlb).
\]
In this paper, we prove the derived version of this:
\begin{thm}\label{thm:main}
Let $G$ be a reductive group over $F$.
There exists an equivalence
\begin{align}\label{eqn:main}
\cal{D}^{\ULA}(\Hck_{G,\Div^1},\Qlb)&\simeq \rm{Sym}(\wh{\fk{g}}(-1)[-2]))\lmod^{\perf}(\cal{D}\Rep(\widehat{G}\rtimes W_{F},\Qlb))
\end{align}
of stable monoidal $\infty$-categories, where $\wh{G}\rtimes W_{F}$ acts on $\wh{\fk{g}}$ via the adjoint action, the symbol $(-1)$ denotes the $(-1)$-Tate twist, and we regard $\rm{Sym}(\wh{\fk{g}}(-1)[-2]))$ as a formal dg-algebra.
Moreover, it is compatible with the (non-derived) geometric Satake equivalence in \cite[\S VI]{FS}.
\end{thm}
We first show the derived geometric Satake equivalence over $\Spd C$:
\begin{thm}\label{thm:mainSpdC}
Let $G$ be a reductive group over $F$.
There exists an equivalence
\begin{align}\label{eqn:mainSpdC}
\cal{D}^{\ULA}(\Hck_{G,\Spd C},\Qlb)\simeq \rm{Sym}(\wh{\fk{g}}[-2]))\lmod^{\perf}(\cal{D}\Rep(\widehat{G},\Qlb))
\end{align}
of stable monoidal $\infty$-categories.
\end{thm}
Essentially, this theorem is already proved in \cite{Ban}, via \cite[{\sc Corollary} VI.6.7]{FS}:
In \cite{Ban}, we introduced a space, called the ``Fargues--Fontaine surface'', obtained by combining the equal characteristic and mixed characteristic Fargues--Fontaine curves.
Using this surface, we constructed an equivalence between the derived Satake category on the equal characteristic Hecke stack $\cal{D}^{\ULA}(\Hck^{\eq},\Qlb)$ and that on the Witt-vector Hecke stack $\cal{D}^{\ULA}(\Hck^{\Witt},\Qlb)$.
By \cite[{\sc Corollary} VI.6.7]{FS}, the latter category is also equivalent to $\cal{D}^{\ULA}(\Hck_{G,\Spd C},\Qlb)$.
Therefore, Theorem \ref{thm:mainSpdC} can be shown by composing these equivalences and the derived Satake equivalence in the equal characteristic proved in \cite{BF}.

However, in this paper, we give an alternative proof of Theorem \ref{thm:mainSpdC} that is more analogous to the argument in \cite{BF}, while we still use the above equivalences of derived Satake categories (in particular, we still rely on the idea of the Fargues--Fontaine surface).
This method allows us to prove that the equivalence (\ref{eqn:mainSpdC}) is $W_{F}$-equivariant.
Regarding this $W_{F}$-equivariance, a similar result is sketched in the last paragraph of \cite[\S 6.7]{BSV}, where they explain the Frobenius-equivariance of Bezrukavnikov--Finkelberg's derived Satake equivalence over $\ol{\bb{F}}_q$.
They establish it parallel to \cite{BF} by verifying the Frobenius-equivariance of each functor in \cite{BF}.
However, our strategy differs in that we refine the argument before \S 6.5 in \cite{BF} and avoid the argument in \S 6.5 and 6.6 in \cite{BF}.
More precisely, in \cite{BF}, they construct a cohomology functor (denoted by $H^*$ in \cite{BF})
\[
R\Gamma\colon \cal{D}^{\ULA}(\Hck_{G,\Spd C},\Qlb)\to \cal{O}_{\bb{T}(\wh{\fk{t}}^*/W)}^{\shear}\lmod
\]
and a Kostant functor
\[
\kappa \colon \Sym(\wh{\fk{g}}[-2])\lmod^{\perf}(\Rep(\wh{G},\Qlb))\to \cal{O}_{\bb{T}(\wh{\fk{t}}^*/W)}^{\shear}\lmod.
\]
Then they proved that these two are fully faithful on the full subcategory of semisimple elements and have the same essential image there.
This defines the fully faithful functor
\[
\wt{S}_{qc}\colon \Sym(\wh{\fk{g}}[-2])\lmod^{fr}(\Rep(\wh{G},\Qlb))\to \cal{D}^{\ULA}(\Hck_{G,\Spd C},\Qlb), 
\]
where $\Sym(\wh{\fk{g}}[-2])\lmod^{fr}(\Rep(\wh{G},\Qlb))$ is the fully subcategory of semisimple objects in $\Sym(\wh{\fk{g}}[-2])\lmod^{\perf}(\Rep(\wh{G},\Qlb))$.
They provide several arguments in \cite[\S 6.5,6.6]{BF} to extend this to the desired equivalence.
However, we prove that the above two functors are fully faithful on the whole category, by using the fact that the sources are generated by semisimple elements under extensions.
We can also show that they are monoidal and $W_F$-equivariant, and have the same essential image.
Therefore, we can directly obtain the desired equivalence (\ref{eqn:mainSpdC}) that is monoidal and $W_F$-equivariant.
We also proved the relationship with the (non-derived) geometric Satake equivalence, which was not discussed in \cite{Ban}.

Once we prove the $W_{F}$-equivariance of (\ref{eqn:mainSpdC}), Theorem \ref{thm:main} can be proved by ``taking $W_{F}$-invariants'' of both sides of (\ref{eqn:mainSpdC}).
More precisely, letting $E/F$ be a finite Galois extension such that $G$ is unramified over $E$, we show the $W_{F}/I_E$-equivariant equivalence
\[
\cal{D}^{\ULA}(\Hck_{G,\Spd \breve{E}},\Qlb)\simeq \rm{Sym}(\wh{\fk{g}}[-2]))\lmod^{\perf}(\cal{D}\Rep(\widehat{G}\times I_E,\Qlb))
\]
and ``take $W_{F}/I_E$-invariants'' of both sides to get (\ref{eqn:main}).

As an application of Theorem \ref{thm:main}, we, together with Teruhisa Koshikawa, plan to enhance the spectral action on the category of lisse $\Qlb$-sheaves on the stack $\rm{Bun}_G$ of $G$-bundles on the Fargues–Fontaine curve.
\subsubsection*{Plan of paper}
We quickly explain the organization of this paper.
In \S \ref{sec:not&prelim}, we introduce some notation and preliminaries.
In \S \ref{sec:cohofHecke}, we calculate the cohomology of the Hecke stack over $\Spd C$ and the $W_F$-action on it, reducing to the equal characteristic case.
We need this cohomology to define the functor $R\Gamma$.
In \S \ref{sec:derSatoverSpdC}, we prove the derived Satake equivalence over $\Spd C$.
In \S \ref{sec:derSatoverSpdEbreve}, we prove the derived Satake equivalence over $\Spd \breve{E}$ with $W_F/I_E$-equivariance.
This immediately implies the derived Satake equivalence over $\Div^1$, in \S \ref{sec:derSatoverDiv1}.
\subsection*{Acknowledgments}
The author is grateful to Teruhisa Koshikawa and Naoki Imai for their useful discussions and comments.
This work is supported by JSPS KAKENHI Grant Number 25KJ0189.
\section{Notation and Preliminaries}\label{sec:not&prelim}
\subsection{Fundamental notation}
Let $p$ be a prime number and $F$ a finite extension of $\QQ_p$.
We write $\cal{O}_{F}$ for the ring of integers of $F$, and $k_{F}$ for its residue field of order $q$.
Let $k=\ol{k}_{F}$ be an algebraic closure of $k_{F}$.
We write $C:=\wh{\ol{F}}$ for the completion of an algebraic closure of $F$ and $\breve{F}$ for the completion of the maximal unramified extension of $F$ in $C$.
We denote by $W_{F}$ the Weil group of $F$, and by $I_{F}$ the inertia group of $F$.

Let $\ell$ be a prime number not equal to $p$.
Let $\Lambda$ be either
a ring killed by some power of $\ell$, 
a finite extension $L$ of $\bb{Q}_{\ell}$,
its ring of integers $\cal{O}_L$, or
$\Qlb$.
We usually consider the case $\Lambda= \Qlb$.
For a dg-algebra $R$, we write $R\lmod$ for the stable $\infty$-category of left dg $R$-modules.
Let $R\lmod^{\perf}\subset R\lmod$ denote the stable full subcategory of perfect $R$-complexes.
Put $\Vect_{\Qlb}:=\Qlb\lmod$ and $\Vect_{\Qlb}^{\perf}:=\Qlb\lmod^{\perf}$.
Let $\Vect^{\heartsuit}$ denote the abelian category of $\Qlb$-vector spaces and $\Vect^{\heartsuit, \rm{f.d.}}$ the abelian category of finite dimensional $\Qlb$-vector spaces.

For an algebraic stack $X$ over $\Qlb$, let $\QCoh(X,\Qlb)$, $\Perf(X,\Qlb)$, and $\Coh(X,\Qlb)$ denote the stable $\infty$-categories of quasi-coherent, perfect, and coherent complexes on $X$, respectively.

Let $G$ be a connected reductive group over $F$.
Let $T$ be a (not necessarily split) maximal torus in $G$.
For any torus $T'$ over a scheme $S$,  we write $\bb{X}^*(T')=\Hom_S(T',\bb{G}_{m,S})$ (resp.~$\bb{X}_*(T')=\Hom_S(\bb{G}_{m,S},T')$) for the character (resp.~cocharacter) lattice.
Fixing a Borel subgroup of $G_{\ol{F}}$ containing $T_{\ol{F}}$, 
let $\bb{X}_*^+(T_{\ol{F}})$ denote the subset of dominant cocharacters.
Let $W=N_{G_{\ol{F}}}(T_{\ol{F}})/T_{\ol{F}}$ denote the Weyl group of $G$.

Throughout this paper, we fix a finite Galois extension $E$ over $F$ such that $G$ is unramified over $E$.
\subsection{Categories of Representations}
For a group scheme $H$ over $\Spec \Lambda$ and a locally profinite group $\Gamma$ with a $\Gamma$-action on $H$ factoring through a finite quotient of $\Gamma$, let $\Rep^{\rm{f.p.}}(H\rtimes \Gamma,\Lambda)$ denote the exact category of representations of $H\rtimes \Gamma$ on a finite projective $\Lambda$-module, which 
are algebraic on $H$ and continuous on $\Gamma$.
When $\Lambda$ is a field, we also write $\Rep^{\rm{f.d.}}$ for $\Rep^{\rm{f.p.}}$.
We write 
\[
\Rep(H\rtimes \Gamma,\Lambda):=\Ind(\Rep^{\rm{f.p.}}(H\rtimes \Gamma,\Lambda))
\]
for its ind-completion.

We define 
\[
\DDRep^{\perf}(I_E,\Lambda)
\]
as follows:
When $\Lambda$ is a torison ring, put
\[
\DDRep^{\perf}(I_E,\Lambda):=\cal{D}^b(\Rep^{\rm{f.p.}}(I_E,\Lambda)).
\]
Let $L$ be  a finite extension of $\QQ_{\ell}$.
We define
\begin{align*}
\DDRep^{\perf}(I_E,\cal{O}_L)&:=\lim_{\substack{\longleftarrow \\ n}}\DDRep^{\perf}(I_E,\cal{O}_L/\ell^n),\\ 
\DDRep^{\perf}(I_E,L)&=\DDRep^{\perf}(I_E,\cal{O}_L)[1/\ell],
\end{align*}
where $(-)[1/\ell]$ denotes the $\infty$-category obtained by inverting $\ell$ in the Hom-spaces.
We also define $\DDRep^{\perf}(I_E,\Qlb)$ as the colimit of $\DDRep^{\perf}(I_E,L)$'s.

Let $\Rep^{\rm{sm}}(I_E,\Lambda)$ denote the category of smooth $I_E$-representations on a $\Lambda$-module.
We need the following basic results.
\begin{lemm}\label{lemm:DReperfcharacterizationtorsion}
Assume that $\Lambda$ is a torison ring.
The canonical functor
\[
\cal{D}^b(\Rep^{\rm{f.p.}}(I_E,\Lambda))\to \cal{D}(\Rep^{\rm{sm}}(I_E,\Lambda)), 
\]
induced by the universality of $\cal{D}^b(-)$, see \cite[Corollary 7.4.12]{BCKW}, is fully faithful.

Moreover, an object $A\in \cal{D}(\Rep^{\rm{sm}}(I_E,\Lambda))$ is in the essential image of $\cal{D}^b(\Rep^{\rm{f.p.}}(I_E,\Lambda))$ if and only if its image under the forgetful functor
\[
\cal{D}(\Rep^{\rm{sm}}(I_E,\Lambda))\to \Lambda\lmod
\]
is a perfect $\Lambda$-complex.
\end{lemm}
\begin{proof}
For the first statement, it suffices to show that for any object in $A\in \Rep^{\rm{f.p.}}(I_E,\Lambda)$ and for any surjection $B\to A$ in $\Rep^{\rm{sm}}(I_E,\Lambda)$, there exist an object $A'\in \Rep^{\rm{f.p.}}(I_E,\Lambda)$ and a morphism $X'\to Y$ such that the composition $A'\to B\to A$ is a surjection whose kernel is in $\Rep^{\rm{f.p.}}(I_E,\Lambda)$.
Choose a system of generators of $A$ as an $I_E$-representation, and lift it to $B$.
By the definition of $\Rep^{\rm{sm}}(I_E,\Lambda)$, this allows us to get a homomorphism $A':=\Lambda[I_E/U]^{\oplus n}\to B$ for some open subgroup $U\subset I_E$.
The composition $A'\to B\to A$ is surjective by construction, and its kernel is in $\Rep^{\rm{f.p.}}(I_E,\Lambda)$ since this category is closed in kernels.

For the last statement, assume that the image of $A\in \cal{D}(\Rep^{\rm{sm}}(I_E,\Lambda))$ under the forgetful functor is a perfect $\Lambda$-complex.
Since all cohomologies of $A$ are finitely generated over $\Lambda$, by the above resolution method, we get a complex of $I_E$-representation that is quasi-isomorphic to $A$ and of the form
\[
\cdots \to \Lambda[I_E/U_1]^{\oplus n_1}\to \Lambda[I_E/U_0]^{\oplus n_0}\to 0\to 0\to \cdots,
\]
where $U_i\subset I_E$ is an open subgroup.
Since $A$ is a perfect $\Lambda$-complex, the kernel of the differential of this complex is finite projective over $\Lambda$ in sufficiently low degrees where the cohomology of $A$ vanishes.
Therefore, this complex is quasi-isomorphic to the finite complex of objects in $\Rep^{\rm{f.p.}}(I_E,\Lambda)$.
This implies the result.
\end{proof}

For a finite extension $L$ over $\QQ_{\ell}$, the functor
\begin{align}\label{eqn:derivedlimit}
\lim_{\substack{\longleftarrow \\ n}}(\cal{O}_L/\ell^n\lmod)\to \cal{O}_L\lmod,\quad (A_n)_n\mapsto R\lim_{\substack{\longleftarrow \\ n}}A_n
\end{align}
taking the derived limit is a fully faithful functor whose essential image is the subcategory of $\ell$-adically derived complete objects (i.e. objects satisfying $M\cong \ds R\lim_{\substack{\longleftarrow \\ n}} (M\otimes^{\bb{L}}_{\cal{O}_L}\cal{O}_L/\ell^n)$).
Furthermore, by \cite[Lemma 09AW]{Sta}, the full subcategory $\ds\lim_{\substack{\longleftarrow \\ n}}(\cal{O}_L/\ell^n\lmod^{\perf})$ of the source maps onto the full subcategory of perfect $\cal{O}_L$-complexes.
\begin{cor}\label{cor:DReperfcharacterizationOL}
Let $L$ be a finite extension over $\QQ_{\ell}$.
The canonical functor
\[
\DDRep^{\perf}(I_E,\cal{O}_L):=\lim_{\substack{\longleftarrow \\ n}}\cal{D}^b(\Rep^{\rm{f.p.}}(I_E,\cal{O}_L/\ell^n)) \to \lim_{\substack{\longleftarrow \\ n}}\cal{D}(\Rep^{\rm{sm}}(I_E,\cal{O}_L/\ell^n))
\]
is fully faithful.

Moreover, an object $A\in \ds \lim_{\substack{\longleftarrow \\ n}}\cal{D}(\Rep^{\rm{sm}}(I_E,\cal{O}_L/\ell^n))$ is in the essential image of $\DDRep^{\perf}(I_E,\cal{O}_L)$ if and only if its image under the forgetful functor
\[
\lim_{\substack{\longleftarrow \\ n}}\cal{D}(\Rep^{\rm{sm}}(I_E,\cal{O}_L/\ell^n))\to \lim_{\substack{\longleftarrow \\ n}}(\cal{O}_L/\ell^n\lmod)\xrightarrow{(\ref{eqn:derivedlimit})} \cal{O}_L\lmod
\]
is a perfect $\cal{O}_L$-complex.
\end{cor}
\begin{proof}
This follows from Lemma \ref{lemm:DReperfcharacterizationtorsion} and the paragraph just before this corollary.
\end{proof}

Since $L\lmod^{\perf}=\cal{O}_L\lmod^{\perf}$ and $\Qlb\lmod^{\perf}=\ds\lim_{\substack{\longrightarrow \\ L}}L\lmod^{\perf}$ hold, there is a canonical forgetful functor
\[
\DDRep^{\perf}(I_E,\Lambda)\to \Lambda\lmod^{\perf}
\]
for general $\Lambda$.

Put
\[
\DDRep(I_E,\cal{O}_L):=\Ind(\DDRep^{\perf}(I_E,\cal{O}_L)).
\]

The conjugation action of $W_F$ on $I_E$ induces the action of $W_F$ (considered as a discrete group) on $\DDRep^{\perf}(I_E,\Lambda)$ and $\DDRep(I_E,\Lambda)$.
Since the restriction of this action to $I_E$ is naturally isomorphic to the trivial action, we get the $W_F/I_E$-action on $\DDRep^{\perf}(I_E,\Lambda)$ and $\DDRep(I_E,\Lambda)$.
We define
\begin{align*}
\DDRep^{\perf}(W_F,\Lambda)&:=\DDRep^{\perf}(I_E,\Lambda)^{W_F/I_E},\\
\DDRep(W_F,\Lambda)&:=\DDRep(I_E,\Lambda)^{W_F/I_E},
\end{align*}
where $(-)^{W_F/I_E}$ is the operation that takes $W_F/I_E$-invariants defined, for example, in \cite[(A.2)]{AG}.
For an $\infty$-category $\cal{C}$ with $W_F/I_E$-actions, the category $\cal{C}^{W_{F}/I_E}$ is the totalization of a cosimplicial category whose $i$-th term is $\prod_{(W_{F}/I_E)^i}\cal{C}$.
When $\Lambda=\cal{O}_L,L$, or $\Qlb$, there is a canonical $t$-structure on $\DDRep^{\perf}(W_F,\Lambda)$ whose heart is $\Rep^{\rm{f.p.}}(I_E,\Lambda)^{W_F/I_E}=\Rep^{\rm{f.p.}}(W_F,\Lambda)$.

For an exact category $\cal{E}$, let $\cal{D}^b(\cal{E})$ denote its bounded derived category, see \cite[\S 7.4]{BCKW}.
For a Grothendieck abelian category $\cal{A}$, let $\cal{D}(\cal{A})$ be its derived category defined in \cite[\S 1.3.5]{LurHA}.

Let $H$ be an algebraic group over $\Lambda$.
Put 
\begin{align*}
\DDRep^{\perf}(H,\Lambda)&:=\cal{D}^b(\Rep^{\rm{f.p.}}(H,\Lambda)),\\
\DDRep(H,\Lambda)&:=\Ind(\cal{D}^b(\Rep^{\rm{f.p.}}(H,\Lambda))).
\end{align*}
If $H$ is a reductive group, there is a canonical equivalence
\begin{align}\label{eqn:DIndequiv}
\DDRep(H,\Qlb)\simeq \cal{D}(\Ind(\Rep^{\rm{f.p.}}(H,\Qlb))),
\end{align}
This follows from \cite[Theorem 4.9]{Kra}.
Note that they only prove that the right-hand side is equivalent to the homotopy category $K(\rm{Inj})$ of complexes of injective objects there, but the latter is equivalent to the derived category since $\Rep(H,\Qlb)$ has finite global dimension if $H$ is reductive.

We also define
\begin{align*}
\DDRep^{\perf}(H\times I_E,\Lambda)&:=\DDRep^{\perf}(H,\Lambda)\otimes_{\Lambda\lmod^{\perf}} \DDRep^{\perf}(I_E,\Lambda),\\
\DDRep(H\times I_E,\Lambda)&:=\DDRep(H,\Lambda)\otimes_{\Lambda\lmod} \DDRep(I_E,\Lambda).
\end{align*}
When $H$ has a $W_F/I_E$-action, we have a $W_F/I_E$-action on $\DDRep^{\perf}(H\times I_E,\Lambda)$ and $\DDRep(H\times I_E,\Lambda)$, and put
\begin{align*}
\DDRep^{\perf}(H\rtimes W_F,\Lambda)&:=\DDRep^{\perf}(H\times I_E,\Lambda)^{W_F/I_E},\\
\DDRep(H\rtimes W_F,\Lambda)&:=\DDRep(H\times I_E,\Lambda)^{W_F/I_E}.
\end{align*}
If $H$ is reductive, any object $V\in \Rep^{\rm{f.p.}}(H\times I_E,\Qlb)$ is of the form
\[
\bigoplus_{i=1}^n V_i\boxtimes W_i,
\]
where $n$ is a nonnegative integer, $V_i$ is an irreducible $H$-representation, and $W_i$ is a $I_E$-representation.
Thus, there is a canonical functor
\[
\Rep^{\rm{f.p.}}(H\times I_E,\Qlb)\to \DDRep^{\perf}(H\times I_E,\Qlb),\quad \bigoplus_{i=1}^n V_i\boxtimes W_i\mapsto \bigoplus_{i=1}^n V_i\boxtimes W_i.
\]
Since $\Rep^{\rm{f.p.}}(H\rtimes W_F,\Qlb)=\Rep^{\rm{f.p.}}(H\times I_E,\Qlb)^{W_F/I_E}$ holds, we also get a functor
\[
\Rep^{\rm{f.p.}}(H\rtimes W_F,\Qlb)\to \DDRep^{\perf}(H\rtimes W_F,\Qlb).
\]

For a stable symmetric monoidal $\infty$-category $\cal{C}$, we write $\rm{Alg}(\cal{C})$ for the category of associative algebra objects in $\cal{C}$ and $\rm{CAlg}(\cal{C})$ for the category of commutative algebra objects in $\cal{C}$.
For a stable monoidal $\infty$-category 
$\cal{C}$, and for an associative algebra object $A\in \rm{Alg}(\cal{C})$, we denote by $A\lmod(\cal{C})$ the $\infty$-category of left $A$-module objects in $\cal{C}$.
For $H'=H, H\times I_E$ or $H\rtimes W_F$, and for  $A\in \rm{Alg}(\DDRep(H',\Qlb))$, we write
\[
A\lmod^{\perf}(\DDRep(H',\Qlb)))
\]
for the full subcategory of $A\lmod(\DDRep(H',\Qlb))$ consisting of objects whose images in 
\[
A\lmod(\Vect_{\Qlb})=A\lmod
\]
under the forgetful functor are perfect over $A$.
Let
\[
A\lmod^{fr}(\DDRep(H',\Qlb))
\]
denote the essential image of 
\[
fr\colon \DDRep^{\perf}(H',\Qlb) \to A\lmod(\DDRep(H',\Qlb)), V\mapsto A\otimes V.
\]
\subsection{Dual group and its Weil action}\label{ssc:dualgp}
We write $\wh{G}_{\Lambda}$ for the dual groups of $G$ over $\Lambda$ and put $\wh{G}:=\wh{G}_{\Qlb}$.
Let $\fk{\wh{g}}$ denote the Lie algebras of $\wh{G}$.
We define $\wh{T}$, $\wh{\fk{t}}$, etc. similarly.
Let
\[
\rm{act}^{\rm{alg}}_{\wh{G}}\colon W_{F} \to \rm{Aut}(\wh{G}_{\Lambda})
\]
denote the usual action by pinned automorphisms.
We twist this action as follows, see the appendix by Richarz--Zhu in \cite{RZ}:
Let $2\rho\in \bb{X}_*(\wh{T}_{\Lambda})$ denote the sum of the positive coroots of $\wh{G}_{\Lambda}$ and $\rho_{\rm{ad}}\colon \bb{G}_{m,\Lambda}\to \wh{T}_{\Lambda}/Z(\wh{G}_{\Lambda})\subset (\wh{G}_{\Lambda})_{\rm{ad}}$ denote a unique homomorphism such that $2\rho_{\rm{ad}}$ is equal to the composition $\bb{G}_{m,\Lambda}\overset{2\rho}{\to} \wh{T}_{\Lambda}\to \wh{T}_{\Lambda}/Z(\wh{G}_{\Lambda})$.
We define
\[
\chi\colon W_{F}\overset{\rm{cycl}}{\to} \rm{Aut}(\mu_{\ell^{\infty}}(\ol{F}))\cong \bb{Z}_{\ell}^{\times} \to \Lambda^{\times} \overset{\rho_{\rm{ad}}}\to (\wh{G}_{\Lambda})_{\rm{ad}}(\Lambda),
\]
where $\rm{cycl}$ is the cyclotomic action.
This induces $\rm{Ad}_{\chi}\colon W_{F} \to \rm{Aut}(\wh{G}_{\Lambda})$ by inner automorphisms.
Put 
\begin{align}\label{eqn:actgeom}
\rm{act}^{\rm{geom}}_{\wh{G}}:= \rm{act}_{\wh{G}}^{\rm{alg}}\circ \rm{Ad}_{\chi}\colon W_{F}\to \rm{Aut}(\wh{G}_{\Lambda}).
\end{align}
Note that since the action via $\rm{Ad}_{\chi}$ is trivial on $\wh{T}_{\Lambda}$, we have 
\[(\rm{act}^{\rm{geom}}_{\wh{G}}(w))|_{\wh{T}_{\Lambda}}=(\rm{act}^{\rm{alg}}_{\wh{G}}(w))|_{\wh{T}_{\Lambda}}=\rm{act}^{\rm{geom}}_{\wh{T}}(w)=\rm{act}^{\rm{alg}}_{\wh{T}}(w)
\]
for any $w\in W_{F}$.

Let $\rm{act}^{\rm{geom}}_{\wh{\fk{g}}}(-)$ denote the $W_{F}$-action on $\wh{\fk{g}}$ given by the derivative of $\rm{act}^{\rm{geom}}_{\wh{G}}(-)$.
For $w\in W_{F}$, put
\[
\rm{act}^{\rm{geom}}_{\wh{\fk{g}}^*}(w):=(\rm{act}^{\rm{geom}}_{\wh{\fk{g}}}(w^{-1}))^*\colon \wh{\fk{g}}^*\to \wh{\fk{g}}^*.
\]
\subsection{Notation on diamonds and v-stacks}\label{ssc:notvstack}
The notation on diamonds and v-stacks is largely based on \cite{Sch}.
For a perfect field $k'$ of characteristic $p$, we write $\Perf_{k'}$ for the category of perfectoid spaces over $k'$.
When we consider v-stacks in this paper, we work on $\Perf_{k_{F}}$.

For a small v-stack $X$, let $\rm{Frob}_{X}\colon X\to X$ be the absolute Frobenius of $X$, which is defined such that for any $S\in \Perf_{k_{F}}$, the induced map $X(S)\to X(S)$ is given by the precomposition with the absolute Frobenius of $S$.

For a small v-stack $X$, we write $\Det(X,\Lambda)$ for the derived category of \'{e}tale sheaves on $X$, and let $\cal{D}_{\et}(X,\Lambda)$ denote its $\infty$-categorical enhancement.
These are defined as follows: when $\Lambda$ is a torsion ring, the definitions are given in \cite[Definition 1.7]{Sch} and \cite[Lemma 17.1]{Sch}; when $\Lambda$ is the ring of integers, these are obtained by passing to the limit as in \cite[Definition 26.1]{Sch}; when $\Lambda$ is a finite extension of $\QQ_{\ell}$, these are defined by inverting $\ell$; and for $\Lambda=\Qlb$, these are given as their colimit.
The six functors can be naively defined for $\cal{D}_{\et}$'s.

Similarly, we define the full subcategory $\cal{D}_{\rm{lc}}(X,\Lambda)\subset \cal{D}_{\et}(X,\Lambda)$ of locally constant complexes, and its homotopy category $D_{\rm{lc}}(X,\Lambda)$ as in the paragraph before {\sc Proposition} IV.7.3 in \cite{FS} when $\Lambda$ is a torsion ring, and as above for general $\Lambda$.
We also define the category $\LocSys(X,\Lambda)$ of local systems on $X$, as the full subcategory of $\cal{D}_{\rm{lc}}(X,\Lambda)$ consisting of complexes that are concentrated in degree 0.
\begin{rmk}
When $\Lambda$ contains $\QQ_{\ell}$, our definition of $\cal{D}_{\et}$ and $\cal{D}_{\rm{lc}}$ might not be the ``correct'' one.
For example, $\cal{D}_{\rm{lc}}([\Spd k/W_F],\Qlb)$ is not equivalent to $\DDRep^{\perf}(W_F,\Qlb))$.
The local system corresponds to some object in $\Rep^{\rm{f.d.}}(W_F,\Qlb)$ does not belong even to $\cal{D}_{\et}([\Spd k/W_F],\Qlb)$.
However, this is not a serious issue, since we restrict our use of $\cal{D}_{\et}$ and $\cal{D}_{\rm{lc}}$ to limited situations.
\end{rmk}
We need the following lemmas later:
\begin{lemm}\label{lemm:DlcSpdkIFRepIFequiv}
There is a canonical equivalence
\begin{align}\label{eqn:DlcSpdkIFRepIF}
\cal{D}_{\rm{lc}}([\Spd k/I_E],\Lambda)\simeq \DDRep^{\perf}(I_E,\Lambda).    
\end{align}
\end{lemm}
\begin{proof}
We can assume that $\Lambda$ is a torsion ring.
First, $\DDRep^{\perf}(I_E,\Lambda)$ is equivalent to the full subcategory of the derived category $\cal{D}(\Rep^{\rm{sm}}(I_E,\Lambda))$ consisting of objects whose image under the forgetful functor is a perfect complex.
This follows from the same argument as \cite[{\sc Theorem} V.1.1]{FS}, noting that the pro-$p$ assumption there is unnecessary when restricted to bounded complexes.
Now the lemma follows from Lemma \ref{lemm:DReperfcharacterizationtorsion}.
\end{proof}
\begin{lemm}\label{lemm:BIFSpdFbreveff}
Let $X$ be a small v-stack.
Let $q\colon \Spd \breve{E}=[\Spd C/I_E]\to [\Spd k/I_E]$ be the map induced by the projection $\Spd C\to \Spd k$.
Then the pullback functor
\[
(\rm{id}_X\times q)^*\colon \cal{D}_{\et}(X\times [\Spd k/I_E],\Lambda)\to \cal{D}_{\et}(X\times \Spd \breve{E},\Lambda)
\]
is fully faithful.
\end{lemm}
\begin{proof}
We can assume that $\Lambda$ is a torison ring.
Then the proof is identical to the proof of \cite[Proposition IV.7.1]{FS}.
\end{proof}
\begin{lemm}\label{lemm:SpdkIFSpdCIFequiv}
The canonical pullback functor
\[
\cal{D}_{\rm{lc}}([\Spd k/I_E],\Lambda)\to  \cal{D}_{\rm{lc}}([\Spd C/I_E],\Lambda)
\]
is an equivalence.
\end{lemm}
\begin{proof}
We can assume that $\Lambda$ is a torison ring.
It is fully faithful by Lemma \ref{lemm:BIFSpdFbreveff}, and an equivalence by the same argument as the proof of the last part of \cite[Proposition IV.7.1]{FS}.
\end{proof}
\begin{lemm}\label{lemm:SpdkIFSpdOCIFequiv}
The canonical pullback functor
\[
\cal{D}_{\rm{lc}}([\Spd k/I_E],\Lambda)\to  \cal{D}_{\rm{lc}}([\Spd \cal{O}_C/I_E],\Lambda)
\]
is an equivalence.
\end{lemm}
\begin{proof}
We can assume that $\Lambda$ is a torsion ring.
We first show the following lemma:
\begin{lemm}
Let $Y$ be a small v-stack such that $Y\times \Spd k(\!(u)\!)$ is a spatial diamond.
Then the pullback functor
\[
\cal{D}_{\rm{consta}}(Y,\Lambda)\to \cal{D}_{\et}(Y\times \Spd \cal{O}_C,\Lambda), 
\]
where $\cal{D}_{\rm{consta}}(Y,\Lambda)\subset \cal{D}_{\et}(Y,\Lambda)$ is the full subcategory of constant sheaves, is fully faithful.
\end{lemm}
\begin{proof}
Let $p\colon \Spd \cal{O}_C\times Y\to Y$ denote the projection and $i\colon Y=\Spd k\times Y\to \Spd \cal{O}_C\times Y$ the closed immersion.
By \cite[Theorem IV.5.3]{FS} applied with $X=\Spd C$ and $S=Y\times \Spd k(\!(u)\!)$, we can show that 
\[
Rp_*A\cong i^*A
\]
for $A\in D_{\et}(\Spd \cal{O}_C\times Y)$ whose restriction to $D_{\et}(\Spd C\times Y)$ is an exterior tensor product, see the argument in \cite[Remark V.4.3]{FS}.
Therefore, the cohomologies of a constant complex on $\Spd \cal{O}_C\times Y$ are isomorphic to those on $Y$.
This proves the lemma.
\end{proof}
Applying this lemma with $Y=I_E^n$, it follows that the pullback functor
\[
\cal{D}_{\rm{consta}}(I_E^n,\Lambda)\to \cal{D}_{\et}(I_E^n\times \Spd \cal{O}_C,\Lambda)
\]
is fully faithful.
Now the claim follows from the same argument as the proof of \cite[Proposition IV.7.1]{FS}.
\end{proof}
Therefore, we get
\begin{align}\label{eqn:DlcSpdEbreveDRepIEeqiv}
 \cal{D}_{\rm{lc}}(\Spd \breve{E},\Lambda)\simeq \cal{D}_{\rm{lc}}([\Spd \cal{O}_C/I_E],\Lambda)\simeq 
\cal{D}_{\rm{lc}}([\Spd k/I_E],\Lambda)\simeq \DDRep^{\perf}(I_E,\Lambda).
\end{align}
\begin{cor}
Let $L$ be a finite extension of $\QQ_{\ell}$.
When $\Lambda=\cal{O}_L$, $L$, or $\Qlb$, there is a canonical $t$-structure on $\DDRep^{\perf}(I_E,\Lambda)$ such that the forgetful functor
\[
\DDRep^{\perf}(I_E,\Lambda)\to \Lambda\lmod^{\perf}
\]
is t-exact.
Its heart is the category $\Rep^{\rm{f.g.}}(I_E,\Lambda)$ of continuous $I_E$-representations on a finitely-generated $\Lambda$-module.
\end{cor}
\begin{proof}
Assume that $\Lambda=\cal{O}_L$.
Since $\DDRep^{\perf}(I_E,\cal{O}_L/\ell^n)\simeq \cal{D}_{\rm{lc}}(\Spd \breve{E},\cal{O}_L/\ell^n)$ holds, it is equivalent to the totalization of a cosimplicial category whose $i$-th term is 
\[
\cal{C}'_{i,n}:=\cal{D}_{\rm{lc}}(\Spd C\times I_E^i,\cal{O}_L/\ell^n).
\]
Since $\Spd C\times I_E^i$ is strictly totally disconnected, $\cal{C}_{i,n}$ is equivalent to 
\[
\cal{C}_{i,n}:=\cal{D}_{\rm{top, lc}}(I_E^i,\cal{O}_L/\ell^n).
\]
where $\cal{D}_{\rm{top, lc}}(I_E^i,\cal{O}_L/\ell^n)$ is the derived category of the sheaves of $\cal{O}_L/\ell^n$-modules on the topological space $I_E^i$.
Therefore, $\DDRep^{\perf}(I_E,\cal{O}_L)\simeq \cal{D}_{\rm{lc}}(\Spd \breve{E},\cal{O}_L)$ is equivalent to the totalization of a cosimplicial category whose $i$-th term is 
\[
\cal{C}_i:=\ds\lim_{\substack{\longleftarrow \\ n}}\cal{D}_{\rm{top, lc}}(I_E^i,\cal{O}_L/\ell^n).
\]
The standard $t$-structure on $\cal{C}_0\overset{(\ref{eqn:derivedlimit})}{\simeq} \cal{O}_L\lmod^{\perf}$ defines the $t$-structure on $\DDRep^{\perf}(I_E,\Lambda)$, and its heart is the limit of the following diagram consisting of the two categories:
\begin{align}\label{eqn:heartofDRepperfOL}
\xymatrix{
\cal{C}_0^{\heartsuit}\ar@<0.5ex>[r]\ar@<-0.5ex>[r]\ar@{<-}[r]&
\cal{C}_1^{\heartsuit}.
}
\end{align}
The left-hand side
\[
\cal{C}_0^{\heartsuit}=(\ds\lim_{\substack{\longleftarrow \\ n}}\cal{D}_{\rm{lc}}(*,\cal{O}_L/\ell^n))^{\heartsuit}\simeq (\lim_{\substack{\longleftarrow \\ n}}(\cal{O}_L/\ell^n\lmod^{\perf}))^{\heartsuit}\overset{(\ref{eqn:derivedlimit})}{\simeq} (\cal{O}_L\lmod^{\perf})^{\heartsuit}
\]
is equivalent to the category of finitely-generated $\cal{O}_L$-modules, and for an object in $\cal{C}_0^{\heartsuit}$ corresponding to a finitely-generated $\cal{O}_L$-module $M$, the endomorphism of its image in $\cal{C}_1^{\heartsuit}$ is isomorphic to 
\[
\ds\lim_{\substack{\longleftarrow \\ n}}\rm{Map}_{\rm{cont}}(I_E,\rm{End}_{\cal{O}_L/\ell^n}(M/\ell^nM))=\rm{Map}_{\rm{cont}}(I_E,\rm{End}_{\cal{O}_L}(M)),
\]
where $\rm{Map}_{\rm{cont}}(-,-)$ is the set of continuous maps.
Unwinding the definitions, it follows that the limit of (\ref{eqn:heartofDRepperfOL}) is equivalent to $\Rep^{\rm{f.g.}}(I_E,\Lambda)$.

Assume that $\Lambda=L$.
The $t$-structure for $L$ is induced by that for $\cal{O}_L$, and its heart is $\Rep^{\rm{f.g.}}(I_E,\cal{O}_L)[1/\ell]$.
It is equal to $\Rep^{\rm{f.p.}}(I_E,\cal{O}_L)[1/\ell]$, and since the action of the profinite group $I_E$ on an $L$-vector space preserves a $\cal{O}_L$-lattice in it,  we have $\Rep^{\rm{f.p.}}(I_E,L)=\Rep^{\rm{f.p.}}(I_E,\cal{O}_L)[1/\ell]$.

Assume that $\Lambda=\Qlb$.
The $t$-structure for $\Qlb$ is induced by those for $L$'s, and its heart is $\ds\lim_{\substack{\longrightarrow\\ L}}\Rep^{\rm{f.p.}}(I_E,L)$.
Since any representation of $I_E$ over $\Qlb$ is trivial on an open subgroup of the wild inertia group, and since the quotient of $I_E$ by this subgroup is topologically finitely generated, we have
\[
\Rep^{\rm{f.p.}}(I_E,\Qlb)=\ds\lim_{\substack{\longrightarrow\\ L}}\Rep^{\rm{f.p.}}(I_E,L).
\]
\end{proof}
\subsection{Notation and preliminaries on geometric Satake}\label{sec:notationHecke}
\subsubsection{Hecke stacks}
For $S\in \Perf_{k_{F}}$, let $X_S$ denote the Fargues--Fontaine curve over $S$, see \cite[{\sc Definition} II.1.15]{FS}.
Let $\Div^1_{k_{F}}:=\Div^1_{X,k_{F}}$ denote the moduli of closed Cartier divisors of $X_S$ of degree 1, defined in \cite[{\sc Definition} II.1.19]{FS}.
Put $\Div^1:=\Div^1_{k_{F}}\times_{\Spd k_{F}}\Spd k$.
By definition, we have
\begin{equation*}
\Div^1_{k_{F}}=[\Spd F/\rm{Frob}_{\Spd F}^{\ZZ}], \qquad
\Div^1=[\Spd \breve{F}/\varphi^{\ZZ}],
\end{equation*}
where $\varphi=\rm{Frob}_{\Spd F}\times_{\Spd k_{F}}\Spd k$.
We also consider variants of $\Div^1_{k_{F}}$ and $\Div^1$:
\begin{align*}
\Div^1_{\cal{Y},k_{F}}&:=\Spd \cal{O}_{F},\qquad
\Div^1_{\cal{Y}}:=\Spd \cal{O}_{\breve{F}},\\
\Div^1_{Y,k_{F}}&:=\Spd F,\qquad
\Div^1_{Y}:=\Spd \breve{F}.
\end{align*}
In \cite[\S VI]{FS}, they defined local Hecke stacks as follows: For an affinoid perfectoid space $S=\Spd (R, R^+)\in \Perf_{k_{F}}$ and a map $S\to \Div^1_{k_{F}}$ corresponding to an affinoid divisor $D$ of $X_S$ of degree $1$, let $B^+_{\Div^1_{k_{F}}}(S)$ denote the global section of the completion of $\cal{O}_{X_S}$ along $\cal{I}$, where $\cal{I}$ is the ideal sheaf corresponding to $D$, and put $B_{\Div^1_{k_{F}}}(S)=B^+_{\Div^1_{k_{F}}}(S)\left[\frac{1}{\cal{I}}\right]$.
Moreover, we define functors
\begin{equation*}
L^+_{\Div^1_{k_{F}}}G\colon \Perf_{k_{F}}\to \rm{Sets},\qquad 
L_{\Div^1_{k_{F}}}G\colon \Perf_{k_{F}}\to \rm{Sets},
\end{equation*}
sending such $S=\Spd (R,R^+)$ with a map $S\to \Div^1_{k_{F}}$ as above to $G(B^+_{\Div^1_{k_{F}}}(S))$ and $G(B_{\Div^1_{k_{F}}}(S))$, respectively.
For a morphism $S\to \Div^1_{k_{F}}$ from a small v-stack $S$, we write
\begin{equation*}
L^+_SG:=L^+_{S/\Div^1_{k_{F}}}G:=L^+_{\Div^1_{k_{F}}}G\times_{\Div^1_{k_{F}}} S, \quad
L_SG:=L_{S/\Div^1_{k_{F}}}G:=L_{\Div^1}G\times_{\Div^1_{k_{F}}} S.
\end{equation*}
Then we define the quotient \'{e}tale stacks
\begin{equation*}
\Gr_{G,S}:=\Gr_{G,S/\Div^1_{k_{F}}}:=[L_{S}G/L^+_{S}G],\quad
\Hck_{G,S}:=\Hck_{G,S/\Div^1_{k_{F}}}:=[L^+_{S}G\backslash L_{S}G/L^+_{S}G].
\end{equation*}
We call $\Gr_{G,S}$ the Beilinson--Drinfeld affine Grassmannian over $S$ and $\Hck_{G,S}$ the local Hecke stack over $S$.
For a reducitve group $G'$ over $\cal{O}_F$ and for a small v-stack $S\to \Div^1_{\cal{Y},k_{F}}$, we can also define $L^+_{S/\Div^1_{\cal{Y},k_{F}}}G'$, $L_{S/\Div^1_{\cal{Y},k_{F}}}G'$, $\Gr_{G',S/\Div^1_{\cal{Y},k_{F}}}$ and $\Hck_{G',S/\Div^1_{\cal{Y},k_{F}}}$ similarly.
\subsubsection{Schubert varieties}
Assume that $G$ is split.
For $\mu\in \bb{X}_*(T)$, we have a point
\begin{align}\label{eqn:[mu]}
[\mu]\colon S\to L_S\bb{G}_m\xrightarrow{L\mu} L_ST\subset L_SG\to \Gr_{G,S},
\end{align}
where the first map corresponds to a local generator $\xi$ of the ideal sheaf of the Fargues--Fontaine curve determined by the map $S\to \Div^1_{(-),k_{F}}$ ($(-)=\cal{Y}$ or $X$), and the last map is the natural projection.
The $L^+_SG$-orbit of this map does not depend on the choice of $\xi$, and we write $\Gr_{G,S,\mu}$ for this.
Assume that $\mu$ is dominant. 
Put
\begin{gather*}
\Gr_{G,S,\leq \mu}=\bigcup_{\substack{\lambda\in \bb{X}^+_*(T) 
\lambda\leq \mu}} \Gr_{G,S,\lambda}, \\
\Hck_{G,S,\mu}=[L^+_SG\backslash \Gr_{G,S,\mu}], \quad
\Hck_{G,S,\leq \mu}=[L^+_SG\backslash \Gr_{G,S,\leq \mu}].
\end{gather*}
Then $\Gr_{G,S,\leq \mu}\to S$ is proper, representable in spatial diamonds, and of finite transcendental degree.
Moreover, the map $\Gr_{G,S,\leq \mu}\to \Gr_{G,S}$ is a closed subfunctor, $\Gr_{G,S,\mu}\to \Gr_{G, S,\leq \mu}$ is an open subfunctor, and $\Gr_{G,S,\mu}$ $(\mu\in \bb{X}^+_*(T))$ defines a stratification of $\Gr_{G,S}$, see \cite[\S VI.2]{FS}.
The subfunctor $\Gr_{G,S,\leq \mu}$ (or $\Hck_{G,S,\leq \mu}$) is called a Schubert variety, and $\Gr_{G,S,\mu}$ (or $\Hck_{G,S,\mu}$) a Schubert cell.

Assume that $G$ is not necessarily split, but that $S\to \Div^1$ factors through the natural map $\Spd \breve{E}\to \Div^1$.
There is a split reductive group $G^{\spl}$ over $\bb{Z}$ and an isomorphism $G^{\spl}_{\breve{E}}\cong G_{\breve{E}}$.
This isomorphism extends to an isomorphism over $R^{\sharp}$ and lifts to $B^+_{dR}(R)$ for any affinoid perfectoid space $\Spa(R,R^+)$ over $S$.
This induces an isomorphism $L_{S}G\cong L_{S}G^{\spl}$ and $L^+_{S}G\cong L^+_{S}G^{\spl}$.
Therefore, it holds that
\begin{align}\label{eqn:HckGHckGspl}
\Hck_{G,S}\cong \Hck_{G^{\spl},S}.
\end{align}
For each $\mu\in \bb{X}^+_{*}(T_{\ol{F}})=\bb{X}^+_{*}(T_{\ol{F}}^{\spl})$, we write $\Hck_{G,\Spd C,\leq \mu}$ for the preimage of $\Hck_{G^{\spl},\Spd C,\leq \mu}$, which does not depend on the choice of split models.

\subsubsection{ULA sheaves and Satake category}\label{sssc:ULAsheavesSatcat}
Let $S$ be a small v-stack over $\Div^1$.
We define the $\infty$-category of ULA complexes on $\Hck_{G,S}$, but only for the cases where $S$ satisfies the condition described below or $S=\Div^1$.
We do not define it for other cases, as it is not necessary for this paper.

First, assume that $S\to \Div^1$ factors through the natural map $\Spd \breve{E}\to \Div^1$. 
When $\Lambda$ is a torsion ring, universal local acyclicity on $\Hck_{G,S}$ is defined in \cite[{\sc Definition} VI.6.1]{FS}.
Namely, we say that an object $A\in \Det(\Hck_{G,S},\Lambda)$ is universally locally acyclic (ULA) over $S$ if its pullback to $\Gr_{G,S}$ is ULA over $S$ in the sense of \cite[Definition IV.2.1]{FS}.
We write
\begin{equation*}
\cal{D}^{\ULA}(\Hck_{G,S},\Lambda)\subset \cal{D}_{\et}(\Hck_{G,S},\Lambda)
\end{equation*}
for the full subcategory of ULA objects.
Similarly, we write
\[
\Sat(\Hck_{G,S},\Lambda)\subset \cal{D}_{\et}(\Hck_{G,S},\Lambda)
\]
for the full subcategory of universally locally acyclic and flat perverse objects, see \cite[\S VI.7.1]{FS}.
For general $\Lambda$, the categories $\cal{D}^{\ULA}(\Hck_{G,S},\Lambda)$ and $\Sat(\Hck_{G,S},\Lambda)$ are defined as follows:
when $\Lambda$ is the ring of integers, these are obtained by passing to the limit as in \cite[Definition 26.1]{Sch}; when $\Lambda$ is a finite extension of $\QQ_{\ell}$, these are defined by inverting $\ell$; and for $\Lambda=\Qlb$, these are given as their colimit.

For a closed subspace $X\subset \Hck_{G,S}$, we can similarly define the category $\cal{D}^{\ULA}(X,\Lambda)$.
In particular, since 
\begin{align}\label{eqn:defofi0S}
i_{0,S}\colon [S/L^+_SG]=[L^+_SG\backslash L^+_SG/L^+_SG]\to [L^+_SG\backslash L_SG/L^+_SG]=\Hck_{G,S}    
\end{align}
is a closed immersion, we get the category $\cal{D}^{\ULA}([S/L^+_SG],\Lambda)$.
\begin{lemm}\label{lemm:ULAisbounded}
\[
\cal{D}^{\ULA}(\Hck_{G,S},\Lambda)=\lim_{\substack{\longrightarrow \\ \mu}}\cal{D}^{\ULA}(\Hck_{G,S,\leq \mu},\Lambda)
\]
\end{lemm}
\begin{proof}
When $\Lambda$ is a torsion ring, the lemma holds by definition.
Let $L$ be a finite extension of $\QQ_{\ell}$ and assume that $\Lambda=\cal{O}_L$.
Since a perfect complex $M$ over $\cal{O}_L/\ell^n$ is zero if and only if $M\otimes^{\bb{L}}_{\cal{O}_L/\ell^n}(\cal{O}_L/\ell)=0$, it follows that $A\in \cal{D}^{\ULA}(\Hck_{G,S},\cal{O}_L/\ell^n)$ is supported on $\Hck_{G,S,\leq \mu}$ if and only if $A\otimes^{\bb{L}}_{\cal{O}_L/\ell^n}(\cal{O}_L/\ell)$ is supported on $\Hck_{G,S,\leq \mu}$.
This proves the lemma for $\Lambda=\cal{O}_L$.
Now, the lemma for $\Lambda=L$ or $\Qlb$ follows from the definitions.
\end{proof}
When $G'$ is a reductive group over $\cal{O}_F$ and $S$ is a small v-stack over $\Div^1_{\cal{Y}}$ such that $S\to \Div^1_{\cal{Y}}$ factors through the map $\Spd \cal{O}_{\breve{E}}\to \Div^1_{\cal{Y}}$ induced by the natural map $\Spd \cal{O}_C\to \Div^1_{\cal{Y}}$.
For example, $S=\Spd \cal{O}_C, \Spd k, [\Spd \cal{O}_C/I_E]$ or $[\Spd k/I_E]$ satisfies this condition.
Then we define $\cal{D}^{\ULA}(\Hck_{G',S},\Lambda)$, $\Sat(\Hck_{G',S},\Lambda)$, and so on in a similar way.

Next, assume that $S=\Div^1$.
Note that there is a $W_E/I_F$-action on $\Hck_{G,\Spd \breve{E}}=\Hck_{G,\Div^1}\times_{\Div^1}\Spd \breve{E}$ since $\Div^1$ is a quotient of $\Spd \breve{E}$ by a $W_F/I_E$-action, see \cite[\S IV.7]{FS}.
This defines a $W_F/I_E$-action on the $\infty$-category $\cal{D}^{\ULA}(\Hck_{G,\Spd \breve{E}},\Lambda)$.
We define 
\[
\cal{D}^{\ULA}(\Hck_{G,\Div^1},\Lambda):=\cal{D}^{\ULA}(\Hck_{G,\Spd \breve{E}},\Lambda)^{W_F/I_E}.
\]
For a closed subspace $X\subset \Hck_{G,\Div^1}$, we can similarly define the category $\cal{D}^{\ULA}(X,\Lambda)$.
\begin{rmk}\label{rmk:defofDULA}
\begin{enumerate}
\item It follows immediately from the definition that if $S\to \Div^1$ factors through $\Spd \breve{E}$, then we have inclusions
\[
\Sat(\Hck_{G,S},\Lambda)\subset \cal{D}^{\ULA}(\Hck_{G,S},\Lambda)\subset \cal{D}_{\et}(\Hck_{G,S},\Lambda)
\]
for general $\Lambda$.
On the other hand, $\cal{D}^{\ULA}(\Hck_{G,\Div^1},\Lambda)$ is not a full subcategory of $\cal{D}_{\et}(\Hck_{G,\Div^1},\Lambda)$ under our definition.
However, it is not a serious problem since we do not need this inclusion.
\item When $\Lambda$ is a torsion ring, our definition of $\cal{D}^{\ULA}(\Hck_{G,\Div^1},\Lambda)$ coincides with the definition in \cite[{\sc Definition} VI.6.1]{FS}.
This follows from the result on descent in \cite[Proposition 17.3]{Sch}
and the equivalence 
\[
\cal{D}_{\et}(\Hck_{G,\Spd \breve{E}}\times (W_{F}/I_E)^i,\Lambda)\simeq \prod_{(W_{F}/I_E)^i}\cal{D}_{\et}(\Hck_{G,\Spd \breve{E}},\Lambda).
\]
\item 
It is presumably possible to define the analogue of $\cal{D}^{\ULA}(\Hck_{G,S},\Lambda)$ using the motivic language of \cite{SchBM}.
The definition would be like
\[
\cal{D}^{\ULA}_{MT}(\Hck_{G,S})\otimes_{\cal{D}_{MT}(*)^{\rm{dual}}}\Lambda\lmod^{\perf},
\]
in the notation of \cite{SchGLCM}.
This definition is expected to be equivalent to our definition (including the case where $S=\Div^1$), but we do not discuss this in this paper.
\end{enumerate} 
\end{rmk}

Assume that $S\to \Div^1$ factors through $\Spd \breve{E}$.
The category $\cal{D}^{\ULA}(\Hck_{G,S},\Lambda)$ is a monoidal category via the convolution product $-\star -$ defined in \cite[\S VI.8]{FS}.
Namely, considering the diagram
\begin{align}\label{eqn:conv}
\Hck_{G,S}\times_S \Hck_{G,S}\xleftarrow{p} \Hck^{(2)}_{G,S}\xrightarrow{m} \Hck_{G,S}, 
\end{align}
where $\Hck^{(2)}_{G,S}:=[L^+_SG\backslash L_SG\times^{L_S^+G} L_SG/L_S^+G]$, the convolution product is defined by $A\star B:=Rm_*p^*(A\boxtimes B)$.

The monoidal unit $\mathbbm{1}_{D^{\ULA}(\Hck_{G,S},\Lambda)}$ is $i_{0,S*}\Lambda$ where $i_{0,S}$ is the closed immersion defined in (\ref{eqn:defofi0S}).
Let $\pi_{\Hck_{G,S}}\colon \Gr_{G,S}\to \Hck_{G,S}$ and $\pi_{G,S}\colon \Gr_{G,S}\to S$ denote the natural maps, and put
\begin{align}\label{eqn:defofFGS}
F_{G,S}:=\cal{H}^i(R\pi_{G,S*}\pi_{\Hck_{G,S}}^*(-))\colon \Sat(\Hck_{G,S},\Lambda)\to \LocSys(S,\Lambda)    
\end{align}
as in \cite[{\sc Definition/Proposition} VI.7.10]{FS}.
By the same argument as the Fargues--Scholze's geometric Satake equivalence in \cite[\S VI]{FS}, we can prove the following:
\begin{prop}\label{prop:nonderSatSpd}
\begin{enumerate}
\item There are a symmetric monoidal structure on $\Sat(\Hck_{G,\Spd \breve{E}},\Lambda)$ and a symmetric monoidal structure on $F_{G,\Spd \breve{E}}$ that induce a symmetric monoidal equivalence
\begin{align}\label{eqn:nonderivedSpdbreveFLambda}
S_{\Spd \breve{E}}\colon \Rep^{\rm{f.p.}}(\wh{G}_{\Lambda}\times I_{E},\Lambda)\simeq \Sat(\Hck_{G,\Spd \breve{E}},\Lambda)
\end{align}
via the Tannakian formalism.
\item There are a symmetric monoidal structure on $\Sat(\Hck_{G,\Spd C},\Lambda)$ and a symmetric monoidal structure on $F_{G,\Spd C}$ that induce a symmetric monoidal equivalence
\begin{align}\label{eqn:nonderSatSpdCLambda}
S_{\Spd C}\colon \Rep^{\rm{f.p.}}(\wh{G}_{\Lambda},\Lambda)\simeq \Sat(\Hck_{G,\Spd C},\Lambda).
\end{align}
\end{enumerate}
\end{prop}
\begin{proof}
We can assume that $\Lambda$ is a torsion ring.
Then we can show the claims using the same argument as \cite[\S VI]{FS}.
Note that the argument in \cite{FS} is not entirely applicable to the construction of the monoidal structure on $F_{G,\Spd C}$; this is handled in \cite{Bantwomonoidal}.
\end{proof}
\subsubsection{Witt vector and equal characteristic Hecke stacks}\label{sssc:schemes}
We often reduce problems on the above Hecke stack defined as a v-stack to the Witt vector Hecke stack or the equal characteristic Hecke stack, which is defined scheme-theoretically.
In this section, we collect preliminaries on the derived category of \'{e}tale sheaves on the latter.

Assume that $G$ is split.
We still write $G$ for its model over $\cal{O}_F$.
Let $\Gr^{\Witt}_{G,\Spec k}$ (resp.$\Gr^{\eq}_{G,\Spec k}$) be the Witt-vector (resp. the equal characteristic) affine Grassmannian defined in \cite[\S 1.1]{Zhumixed} (resp. for example in \cite[\S 2]{MV}). Similarly, let $L^{+,\Witt}_{\Spec k}G$ (resp. $L^{+,\eq}_{\Spec k}G$) be the Witt vector (resp. the equal characteristic) loop group, and put $\Hck_{G,\Spec k}^{(-)}=[L^{+,(-)}_{\Spec k}G\backslash \Gr^{(-)}_{G,\Spec k}]$ for $(-)=\Witt$ or $\eq$.
We similarly define Schubert cells, Schubert varieties, and the point $[\mu]$.

For $(-)=\Witt$ or $\eq$, we first define the category
\[
\cal{D}_{\rm{cons}}(\Hck^{(-)}_{G,\Spec k,\leq \mu},\Lambda)
\]
of perfect-constructible $\Lambda$-complexes on $\Hck^{(-)}_{G,\Spec k,\leq \mu}$ for $\mu\in \bb{X}_*^+(T)$ as follows:
Since the $L^{+,(-)}_{\Spec k}G$-action on $\Gr^{(-)}_{G,\Spec k,\leq \mu}$ factors through a quotient $L^{+,(-)}_{\Spec k}G/L^{(-)}_{\Spec k}G^{(m)}$ for some $m$, where $L^{(-)}_{\Spec k}G^{(m)}$ is the $m$-th congruence subgroup defined in \cite[\S 5.2]{Zhueq} or \cite[(1.1.1)]{Zhumixed}.
Since $L^{+,(-)}_{\Spec k}G/L^{(-)}_{\Spec k}G^{(m)}$ is a scheme of finite type over $k$ or its perfection, the quotient stack
\[
X^{(m)}_{\leq \mu}=[(L^{+,(-)}_{\Spec k}G/L^{(-)}_{\Spec k}G^{(m)})\backslash \Gr^{(-)}_{G,\Spec k,\leq \mu}]
\]
is an Artin stack of finite type over $\Spec k$ or its perfection.
Then, by \cite{LZ}, when $\Lambda$ is a torsion ring, there is a $\infty$-category 
\[
\cal{D}_{\et}(X^{(m)}_{\leq \mu},\Lambda)
\]
of \'{e}tale sheaves on $X^{(m)}_{\leq \mu}$ (since the choice of a deperfection does not change the category of \'{e}tale sheaves by the arguments in \cite[\S A]{Zhumixed}).
Let 
\[
\cal{D}_{\rm{cons}}(X^{(m)}_{\leq \mu},\Lambda).
\]
denote the full subcategory of perfect-constructible sheaves, where a complex on $X^{(m)}_{\leq \mu}$ is called perfect-constructible if its pullback to $\Gr_{G,\Spec k,\leq \mu}$ is perfect-constructible.
We put
\begin{align*}
\cal{D}_{\et}(\Hck^{(-)}_{G,\Spec k,\leq \mu},\Lambda)&:=\cal{D}_{\et}(X^{(m)}_{\leq \mu},\Lambda)\\
\cal{D}_{\rm{cons}}(\Hck^{(-)}_{G,\Spec k,\leq \mu},\Lambda)&:=\cal{D}_{\rm{cons}}(X^{(m)}_{\leq \mu},\Lambda).    
\end{align*}
These do not depend on the choice of $m$, since the congruence subgroup $L^{(-)}_{\Spec k}G^{(m)}$ is (the perfection of) a pro-unipotent group for any $m$, and since taking the quotient stack by the trivial action of a unipotent group does not change the derived category.

On the Hecke stack, perfect-constructibility agrees with universal local acyclicity:
\begin{lemm}\label{lemm:consULA}
Assume that $\Lambda$ is a torsion ring.
An object 
\[
A\in \cal{D}((X^{(m)}_{\leq \mu})_{\text{\'{e}t}},\Lambda)
\]
belongs to $\cal{D}_{\rm{cons}}((X^{(m)}_{\leq \mu})_{\text{\'{e}t}},\Lambda)$ if and only if it is universally locally acyclic(ULA) as a complex on $\Hck_{G,\Spec k}$ in the sense of \cite[Definition 4.6]{Ban}.
\end{lemm}
\begin{proof}
Let 
\[
\pi_{\Gr,k}\colon \Gr_{G,\Spec k}^{(-)}\to \Hck_{G,\Spec k}^{(-)}
\]
be the natural projection.
By definition, $A$ belongs to $\cal{D}_{\rm{cons}}((X^{(m)}_{\leq \mu})_{\text{\'{e}t}},\Lambda)$ if and only if $\pi_{\Gr,k}^*A$ is perfect-constructible after pulling back to $\Gr_{G,\Spec k}^{(-)}$.
By the equivariance of $A$, the object $\pi_{\Gr,k}^*A$ is locally constant on each Schubert cell.
Therefore, $A$ belongs to $\cal{D}_{\rm{cons}}((X^{(m)}_{\leq \mu})_{\text{\'{e}t}},\Lambda)$ if and only if the fiber $[\mu]^*A$ is a perfect complex.
Then the claim follows from \cite[Proposition 4.7]{Ban}.
\end{proof}
For a finite extension $L$ of $\QQ_\ell$, we define
\begin{align}
\cal{D}_{\rm{cons}}(\Hck^{(-)}_{G,\Spec k,\leq \mu},\cal{O}_L)&:= \lim_{\substack{\longleftarrow\\ n}} \cal{D}_{\rm{cons}}(\Hck^{(-)}_{G,\Spec k,\leq \mu},\cal{O}_L/\ell^n),\label{eqn:Dconslim} \\
\cal{D}_{\rm{cons}}(\Hck^{(-)}_{G,\Spec k,\leq \mu},L)&:= \cal{D}_{\rm{cons}}(\Hck^{(-)}_{G,\Spec k,\leq \mu},\cal{O}_L)[1/\ell]\label{eqn:Dconslinv}
\end{align}
and
\begin{align}\label{eqn:Dconscolim}
\cal{D}_{\rm{cons}}(\Hck^{(-)}_{G,\Spec k,\leq \mu},\Qlb):=\lim_{\substack{\longrightarrow\\ L/\QQ_{\ell}}}\cal{D}_{\rm{cons}}(\Hck^{(-)}_{G,\Spec k,\leq \mu},L)
\end{align}
where $L$ runs over finite extensions of $\QQ_{\ell}$ in $\Qlb$.
On the other hand, by \cite{BSproet} and \cite{Cho}, there is a category 
\[
D_{\rm{cons}}((X^{(m)}_{\leq \mu})_{\text{pro\'{e}t}},\Lambda)
\]
of perfect-constructible sheaves on the pro-\'{e}tale site over $X^{(m)}_{\leq \mu}$.
\begin{lemm}
Assume that $\Lambda=\cal{O}_L/\ell^n,\cal{O}_L,L$ or $\Qlb$, where $L$ is a finite extensions of $\QQ_{\ell}$ and $n$ is a nonnegative integer.
The homotopy category of $\cal{D}_{\rm{cons}}(\Hck^{(-)}_{G,\Spec k,\leq \mu},\Lambda)$ is equivalent to $D_{\rm{cons}}((X^{(m)}_{\leq \mu})_{\text{pro\'{e}t}},\Lambda)$.
\end{lemm}
\begin{proof}
It suffices to show a similar statement for any Artin stack $X$.
When $\Lambda$ is a torsion ring, then by descent (\cite[Proposition 5.3.5]{LZ} and \cite[Proposition 4.27]{Cho}), this reduces to the case of schemes and is proved in the paragraph after \cite[Definition 6.5.1]{BSproet}.
When $\Lambda=\cal{O}_L$, this reduces to the case of schemes again, and follows from \cite[Definition 6.5.1]{BSproet}.
When $\Lambda=L$ or $\Qlb$, this follows from \cite[Theorem 1.2]{Cho}.
\end{proof}
In view of Lemma \ref{lemm:consULA}, we also write $\cal{D}^{\ULA}(\Hck_{G,\Spec k,\leq \mu}^{(-)},\Lambda)$ for $\cal{D}_{\rm{cons}}(\Hck_{G,\Spec k,\leq \mu}^{(-)},\Lambda)$, and we define the category of ULA complexes on $\Hck_{G,\Spec k}^{(-)}$ by
\[
\cal{D}^{\ULA}(\Hck_{G,\Spec k}^{(-)},\Lambda)=\lim_{\substack{\longrightarrow \\ \mu\in \bb{X}_*^+(T)}} \cal{D}_{\rm{cons}}((\Hck_{G,\Spec k,\leq \mu}^{(-)})_{\et},\Lambda)
\]
even if $\Lambda$ is not a torsion ring.

The perverse $t$-structure on $\cal{D}_{\rm{cons}}(\Gr_{G,\Spec k, \leq \mu},\Qlb)$ induces that on $\cal{D}_{\rm{cons}}(\Hck_{G,\Spec k, \leq \mu},\Qlb)$, that is, $A\in \cal{D}_{\rm{cons}}(\Hck_{G,\Spec k, \leq \mu},\Qlb)$ is in ${}^p\cal{D}^{\geq 0}$ (resp.${}^p\cal{D}^{\leq 0}$) if and only if its pullback to $\Gr_{G,\Spec k, \leq \mu}$ is in ${}^p\cal{D}^{\geq 0}$ (resp.${}^p\cal{D}^{\leq 0}$) with respect to the usual perverse $t$-structure $({}^p\cal{D}^{\geq 0},{}^p\cal{D}^{\leq 0})$.
This induces the perverse $t$-structure on $\cal{D}^{\ULA}(\Hck_{G,\Spec k}^{(-)},\Qlb)$.
Since any constructible sheaf on any scheme of finite type over $k$ is bounded with respect to the perverse $t$-structure, we have the following:
\begin{lemm}\label{lemm:DULAschemeboundedperverse}
Any object in $\cal{D}^{\ULA}(\Hck_{G,\Spec k}^{(-)},\Qlb)$ is bounded with respect to the perverse $t$-structure.    
\end{lemm}
\section{Cohomology of Hecke stack}\label{sec:cohofHecke}
In this section, we will prove some results on the cohomology ring of $\Hck_{G,\Spd C}$.
\subsection{Equivalence between categories of ULA sheaves}
Recall that we are fixing a maximal torus $T\subset G$ (not necessarily split).
Let $T^{\spl}\subset G^{\spl}$ be a split model over $\Spec \bb{Z}$ of $T\subset G$, i.e. a split reductive group $G^{\spl}$ over $\Spec \bb{Z}$ and its maximal split torus $T^{\spl}$ equipped with an identification of the inclusion $T^{\spl}_{\ol{F}}\subset G^{\spl}_{\ol{F}}$ with $T_{\ol{F}}\subset G_{\ol{F}}$.
First, we have the following proposition:
\begin{prop}\label{prop:DULAequiv}
There exist monoidal equivalences
\begin{align}
\cal{D}^{\ULA}(\Hck_{G,\Spd C},\Lambda)&\simeq \cal{D}^{\ULA}(\Hck_{G^{\spl},\Spd C},\Lambda)\notag\\
&\simeq \cal{D}^{\ULA}(\Hck_{G^{\spl},\Spd \cal{O}_C},\Lambda)\notag\\
&\simeq \cal{D}^{\ULA}(\Hck_{G^{\spl},\Spd k},\Lambda)\notag\\
&\simeq \cal{D}^{\ULA}(\Hck^{\Witt}_{G^{\spl},\Spec k},\Lambda)\notag\\
&\simeq \cal{D}^{\ULA}(\Hck^{eq}_{G^{\spl},\Spec k},\Lambda).\label{eqn:DULAequiv}
\end{align}
of monoidal $\infty$-categories.
\end{prop}
\begin{proof}
By (\ref{eqn:Dconslim}-\ref{eqn:Dconscolim}), Lemma \ref{lemm:ULAisbounded} and the definitions of $\cal{D}^{\ULA}$, we can assume that $\Lambda$ is a torsion ring.
First, the isomorphism (\ref{eqn:HckGHckGspl}) implies
\[
\cal{D}^{\ULA}(\Hck_{G,\Spd C},\Lambda)\simeq \cal{D}^{\ULA}(\Hck_{G^{\spl},\Spd C},\Lambda).
\]
Moreover, by \cite[{\sc Corollary VI.6.7}]{FS}, pullback functors define equivalences
\begin{align}\label{eqn:equivCOCk}
\cal{D}^{\ULA}(\Hck_{G^{\spl},\Spd C},\Lambda)\overset{\sim}{\leftarrow}\cal{D}^{\ULA}(\Hck_{G^{\spl},\Spd \cal{O}_C},\Lambda)\overset{\sim}{\rightarrow}\cal{D}^{\ULA}(\Hck_{G^{\spl},\Spd k},\Lambda).
\end{align}
Next, we claim that
\begin{align*}
\cal{D}^{\ULA}(\Hck_{G^{\spl},\Spd k},\Lambda)\simeq \cal{D}^{\ULA}(\Hck_{G^{\spl},\Spec k},\Lambda).
\end{align*}
This follows from the result in \cite[\S 27]{Sch} on the comparison to schemes.
More precisely, it goes as follows:
It holds that
\begin{align}
\Hck_{G^{\spl},\Spec k,\leq \mu}^{\Witt}&=[L^{+,\Witt}_{\Spec k}G^{\spl}\backslash \Gr^{\Witt}_{G^{\spl},\Spec k,\leq \mu}]\\
\Hck_{G^{\spl},\Spd k,\leq \mu}&=[(L^{+,\Witt}_{\Spec k}G^{\spl})^{\diamondsuit}\backslash (\Gr^{\Witt}_{G^{\spl},\Spec k,\leq \mu})^{\dia}]
\end{align}
in the terminology of \cite[\S 27]{Sch}.
As written in \S \ref{sssc:schemes}, the $L^{+,\Witt}_{\Spec k}G^{\spl}$-action on $\Gr^{\Witt}_{G^{\spl},\Spec k,\leq \mu}$ factors through a quotient $L^{+,\Witt}_{\Spec k}G^{\spl}/L^{+,\Witt}_{\Spec k}G^{\spl, (m)}$ by the $m$-th congruence group for some $m$, and 
\[
\cal{D}_{\et}(\Hck_{G^{\spl},\Spec k,\leq \mu}^{\Witt},\Lambda)=\cal{D}_{\et}([(L^{+,\Witt}_{\Spec k}G^{\spl}/L^{+,\Witt}_{\Spec k}G^{\spl, (m)})\backslash \Gr^{\Witt}_{G^{\spl},\Spec k,\leq \mu} ],\Lambda).
\]
Similarly, by \cite[Proposition VI.4.1]{FS}, 
\[
\cal{D}_{\et}(\Hck_{G^{\spl},\Spd k,\leq \mu},\Lambda)\simeq \cal{D}_{\et}([(L^{+,\Witt}_{\Spec k}G^{\spl}/L^{+,\Witt}_{\Spec k}G^{\spl, (m)})^{\diamondsuit}\backslash (\Gr^{\Witt}_{G^{\spl},\Spec k,\leq \mu})^{\diamondsuit} ],\Lambda).
\]
It follows from \cite[Proposition 27.2]{Sch} and descent results (\cite[Proposition 5.3.5]{LZ} and \cite[Proposition 17.3]{Sch}) that for any stack $X$ that is the perfection of an Artin stack of finite type over $\Spec k$, there is a canonical fully faithful functor
\[
c_X^*\colon \cal{D}_{\et}(X,\Lambda)\to \cal{D}_{\et}(X^{\diamondsuit},\Lambda), 
\]
and this functor commutes with $*$-pullbacks and $!$-pushfowards by \cite[Proposition 27.1, Proposition 27.4]{Sch}.
Hence, there is a fully faithful functor
\begin{align}\label{eqn:HckSpeckmucomparisontosch}
\cal{D}_{\et}(\Hck^{\Witt}_{G^{\spl},\Spec k,\leq \mu},\Lambda)\to \cal{D}_{\et}(\Hck_{G^{\spl},\Spd k,\leq \mu},\Lambda)    
\end{align}
for any $\mu$.
By the characterization of universal local acyclicity in \cite[{\sc Proposition} VI.6.5]{FS} and \cite[Proposition 4.7]{Ban}, one can show that this restricts to a functor
\[
\cal{D}^{\ULA}(\Hck^{\Witt}_{G^{\spl},\Spec k,\leq \mu},\Lambda)\to \cal{D}^{\ULA}(\Hck_{G^{\spl},\Spd k,\leq \mu},\Lambda).
\]
Therefore, we obtain a fully faithful functor
\[
\cal{D}^{\ULA}(\Hck_{G^{\spl},\Spec k},\Lambda)\to \cal{D}^{\ULA}(\Hck_{G^{\spl},\Spd k},\Lambda).
\]
This functor is monoidal since the convolution product is defined using $*$-pullbacks and $!$-pushforwards.
By the semiorthogonal decomposition corresponding to the stratification into Schubert cells, the essential surjectivity of this functor is reduced to that of
\[
D^{\ULA}(\Hck_{G^{\spl},\Spec k,\mu},\Lambda)\to D^{\ULA}(\Hck_{G^{\spl},\Spd k,\mu},\Lambda)
\]
for each $\mu$.
Since $\Hck_{G^{\spl},\Spec k,\mu}=[\Spec k/H_{\mu}]$ and  $\Hck_{G^{\spl},\Spd k,\mu}=[\Spd k/H_{\mu}^{\dia}]$ holds for some group scheme $H_{\mu}$, the essential surjectivity reduces to that of $D_{\rm{lc}}(\Spec k,\Lambda)\to D_{\rm{lc}}(\Spd k,\Lambda)$ by the descent argument in the proof of \cite[{\sc Proposition} IV.7.1]{FS}, where $D_{\rm{lc}}(-)$ means the category of locally constant sheaves with perfect fibers.
Since $D_{\rm{lc}}(\Spec k,\Lambda)$ and $D_{\rm{lc}}(\Spd k,\Lambda)$ are equivalent to the category of perfect $\Lambda$-complexes, the claim follows.

Finally, we have
\begin{align}\label{eqn:equivWitteq}
\cal{D}^{\ULA}(\Hck_{G^{\spl},\Spec k}^{\Witt},\Lambda)\simeq \cal{D}^{\ULA}(\Hck_{G^{\spl},\Spec k}^{\eq},\Lambda)
\end{align}
by \cite[Theorem 1]{Ban}.
This proves the proposition.
\end{proof}
We used \cite[Theorem 1]{Ban} in the above proof.
Let us recall the argument of the proof of \cite[Theorem 1]{Ban}.
We introduce a scheme $\Divtil$ and an affine Grassmannian $\Gr_{\Divtil}$ over $\Divtil$, such that its fiber over a point $x_1\colon \Spec k\to \Divtil$ is the Witt-vector affine Grassmannian, and the fiber over another point $x_2$ is (the perfection of) the equal characteristic affine Grassmannian.
Then we can show (\ref{eqn:equivWitteq}) by using triples 
\begin{align}\label{eqn:O1O2}
(x_1,\Spec \cal{O}_1,\Spec \ol{K(\Divtil)}])\text{ and }(x_2,\Spec \cal{O}_2,\Spec \ol{K(\Divtil)}])
\end{align}
for some Henselian local ring $\cal{O}_1,\cal{O}_2$ instead of using a triple $(\Spd k, \Spd \cal{O}_C,\Spd C)$ in \cite[{\sc Corollary VI.6.7}]{FS}.
We will not use the details of this proof, but we will later use the fact that all the functors appearing in this proof are pullbacks along maps.

We will need the following lemma on the preservation of perversity later:
\begin{lemm}\label{lemm:SatPervequiv}
For $(-)=\eq$ or $\Witt$, let $\Perv(\Hck^{(-)}_{G^{\spl},\Spec k},\Qlb)\subset \cal{D}^{\ULA}(\Hck^{(-)}_{G^{\spl},\Spec k},\Qlb)$ denote the full subcategory of the $L^{+,(-)}_{\Spec k}G^{\spl}$-equivariant perverse $\Qlb$-sheaves on $\Gr^{(-)}_{G^{\spl},\Spec k}$.
\begin{enumerate}
\item The equivalences in Proposition \ref{prop:DULAequiv} restrict to monoidal equivalences between 
\begin{align}
\Sat(\Hck_{G,\Spd C},\Qlb)&\simeq \Sat(\Hck_{G^{\spl},\Spd C},\Qlb)\notag\\
&\simeq \Sat(\Hck_{G^{\spl},\Spd \cal{O}_C},\Qlb)\notag\\
&\simeq \Sat(\Hck_{G^{\spl},\Spd k},\Qlb)\notag\\
&\simeq \Perv(\Hck^{\Witt}_{G^{\spl},\Spec k},\Qlb)\notag\\
&\simeq \Perv(\Hck^{\eq}_{G^{\spl},\Spec k},\Qlb).\label{eqn:SatPervequiv}
\end{align}
\item \label{item:threenonderSatcompati} There are canonical monoidal natural isomorphisms that make squares
\[
\xymatrix{
\Sat(\Hck_{G,\Spd C},\Qlb)\ar[r]^-{(\ref{eqn:nonderSatSpdCLambda})}\ar[d]_{\simeq}^{(\ref{eqn:SatPervequiv})}&\Rep^{\rm{f.d.}}(\wh{G},\Qlb)\ar@{=}[d]\\
\Perv(\Hck^{\Witt}_{G^{\spl},\Spec k},\Qlb)\ar[r]^-{S_{\Spec k}^{\Witt}}\ar[d]_{\simeq}^{(\ref{eqn:SatPervequiv})}&\Rep^{\rm{f.d.}}(\wh{G},\Qlb)\ar@{=}[d]\\
\Perv(\Hck^{\eq}_{G^{\spl},\Spec k},\Qlb)\ar[r]^-{S_{\Spec k}^{\eq}}&\Rep^{\rm{f.d.}}(\wh{G},\Qlb)\\
}
\]
commute, where $S_{\Spec k}^{\Witt}$ is the geometric Satake equivalence in \cite{Zhumixed} and $S_{\Spec k}^{\eq}$ is the one in \cite{MV}.
\end{enumerate}
\end{lemm}
\begin{rmk}
Unlike \cite[\S VI]{FS}, the constructions of the geometric Satake equivalences in \cite{MV} and \cite{Zhumixed} are not completely canonical.
They establish the geometric Satake equivalence from the fiber functor
\[
\Perv(\Hck^{(-)}_{G^{\spl},\Spec k},\Qlb)\to \Vect_{\Qlb}^{\heartsuit, \rm{f.d.}}, A\mapsto \bigoplus_{i}H^i(\Gr^{(-)}_{G^{\spl},\Spec k},A)
\]
via Tannakian construction, but the isomorphism between the Tannakian group and $\wh{G}$ has an ambiguity of inner automorphisms of $\wh{G}$.
Therefore, the term ``canonical'' in the statement of Lemma \ref{lemm:SatPervequiv} (\ref{item:threenonderSatcompati}) means, more precisely, ``canonical up to this ambiguity.''
This ambiguity does not affect our results.
Moreover, it is possible to avoid this ambiguity and prove the statement strictly if we use the method of \cite{FS} in the proofs of \cite{MV} and \cite{Zhumixed}.
\end{rmk}
\begin{proof}
\begin{enumerate}
\item The last equivalence is proved in \cite[\S 6]{Ban}.

For the first four equivalence, the result is implicitly mentioned in \cite[Remark I.2.14]{FS}.
We give a proof as follows:
Since the first equivalence is implied by (\ref{eqn:HckGHckGspl}), we can assume that $G=G^{\spl}_F$, and abusively write $G$ for $G^{\spl}$.
For the second and third equvalence, we can assume that $\Lambda$ is a torsion ring.
By characterization of ULA sheaves and Satake sheaves in \cite[Proposition VI.6.4, Proposition VI.7.7]{FS}, the second (resp. third) equivalence reduces to the claim that all fibers of an object in $\cal{D}_{\rm{lc}}(\Spd \cal{O}_C,\Lambda)$ are finite projective $\Lambda$-modules in degree 0 if and only if its pullback to $\cal{D}_{\rm{lc}}(\Spd k,\Lambda)$ (resp. $\cal{D}_{\rm{lc}}(\Spd C,\Lambda)$) satisfies the same condition.
This follows from the equivalences $\cal{D}_{\rm{lc}}(S,\Lambda)\simeq \Lambda\lmod^{\perf}$ for $S=\Spd \cal{O}_C, \Spd k$, and $\Spd C$.

For the fourth equivalence, let
\[
\cal{D}_{\et}(\Hck_{G,\Spd k},\Lambda)^{\bd}\subset \cal{D}_{\et}(\Hck_{G,\Spd k},\Lambda)
\]
denote the full subcategory of all objects with quasi-compact support.
The perverse $t$-structure on $\cal{D}_{\et}(\Hck_{G,\Spd k},\Lambda)^{\bd}$ is defined in \cite[\sc{Definition/Proposition} VI.7.1]{FS} when $\Lambda$ is a torsion ring, and by taking limits, inverting $\ell$ and taking a colimit, we also get the perverse $t$-structure on $\cal{D}_{\et}(\Hck_{G,\Spd k},\Qlb)^{\bd}$.
Since the fully faithful functor (\ref{eqn:HckSpeckmucomparisontosch}) is essentially surjective by the same argument as the proof of Proposition \ref{prop:DULAequiv}, we get an equivalence
\begin{align}\label{eqn:comparisontoschDetbd}
D_{\et}(\Hck^{\Witt}_{G,\Spec k},\Qlb)^{\bd}\simeq D_{\et}(\Hck_{G,\Spd k},\Qlb)^{\bd}.    
\end{align}
By the definition of the perverse $t$-structures, we can show that this equivalence preserves the perverse $t$-structures. 

Let $j_{\mu,\Spd k}\colon \Hck_{G,\Spd k,\mu}\hookrightarrow \Hck_{G,\Spd k}$ denote the inclusion and put $d_{\mu}=\langle 2\rho,\mu\rangle$.
Let $IC_{\mu,\Spd k}$ denote the image of the natural homomorphism
\[
{}^p\cal{H}^0(j_{\mu,\Spd k!}\Qlb[d_{\mu}])\to {}^p\cal{H}^0(Rj_{\mu,\Spd k*}\Qlb[d_{\mu}]).
\]
In fact, $IC_{\mu,\Spd k}\cong {}^p\cal{H}^0(j_{\mu,\Spd k!}\Qlb[d_{\mu}])\cong  {}^p\cal{H}^0(Rj_{\mu,\Spd k*}\Qlb[d_{\mu}])$ holds by \cite[Proposition VI.7.5]{FS}.
Since $j_{\mu,\Spd k !}$ and ${}^p\cal{H}^0(-)$ is compatible with the comparison equivalence (\ref{eqn:comparisontoschDetbd}), this equivalence maps $IC_{\mu,\Spd k}$ to the intersection cohomology sheaf $IC_{\mu,\Spec k}$ on $\Gr^{\Witt}_{G,\Spec k,\leq \mu}$.
By the same argument as \cite[\S VI.11]{FS} on ``$\Sat_G(\QQ_{\ell})$'', the category $\Sat(\Hck_{G,S},\Qlb)$ is given by $\ds\bigoplus_{\mu\in \bb{X}_*^+(T)}\Vect_{\Qlb}^{\heartsuit, \rm{f.d.}}\otimes IC_{\mu,\Spd k}$.
Since $\Perv(\Hck^{\Witt}_{G,\Spec k},\Qlb)$ is generated by $IC_{\mu,\Spec k}$'s by \cite[\S 2.1.1]{Zhumixed}, we get the fourth equivalence.

Note that monoidalities follow from those of (\ref{eqn:DULAequiv}).
\item 
We can assume that $G=G^{\spl}_F$ and abusively write $G$ for $G^{\spl}$.
Put 
\begin{align*}
\cal{P}_1&=\Sat(\Hck_{G,\Spd C},\Qlb),\\
\cal{P}_2&=\Sat(\Hck_{G,\Spd k},\Qlb)\\
\cal{P}_3&=\Perv(\Hck^{\Witt}_{G,\Spec k},\Qlb),\\
\cal{P}_4&=\Perv(\Hck^{\eq}_{G,\Spec k},\Qlb).
\end{align*}
For $1\leq i\leq 4$, let
\[
F_i\colon \cal{P}_i\to \Vect^{\heartsuit,\rm{f.p.}}
\]
be the functor sending a perverse sheaf $A$ to the direct sum of the cohomology groups of $A$ on the affine Grassmannian over all degrees.
The functor $F_1$ is the one defined in (\ref{eqn:defofFGS}).
We equip $\cal{P}_1,\cal{P}_3,\cal{P}_4$ with symmetric monoidal structures as in Proposition \ref{prop:nonderSatSpd}, \cite[\S 2.4]{Zhumixed} and \cite[\S 5.2]{Zhueq}, respectively.
We also equip $F_1,F_3,F_4$ with monoidal structures that are symmetric as in Proposition \ref{prop:nonderSatSpd}, \cite[\S 2.3]{Zhumixed} and \cite[\S 5.2]{Zhueq}, respectively.
In \cite[\S 5.2]{Zhueq}, it is mentioned that there are several ways to endow a monoidal structure with $F_4$ but that these coincide.
Moreover, we equip $F_2$ with a monoidal structure, in the same way as \cite[\S 2.3]{Zhumixed}.

Since $F_i$ is defined via $*$-pullbacks and pushforward from the proper subspace in the affine Grassmannian, we can show that there is a natural isomorphism
\begin{align}\label{eqn:diag:fiberfunctorintw}
\vcenter{
\xymatrix{
\cal{P}_i\ar[r]^-{F_i}\ar[d]^{\simeq}_{(\ref{eqn:SatPervequiv})}&\Vect^{\heartsuit,\rm{f.d.}}\ar@{=}[d]\\
\cal{P}_{i+1}\ar[r]^-{F_{i+1}}&\Vect^{\heartsuit,\rm{f.d.}}
}    
}
\end{align}
for $i=1,2,3$.
We claim that this isomorphism is monoidal for $i=1,2,3$.
When $i=1$, it is proved in \cite{Bantwomonoidal}.
When $i=2$, the monoidal structures on $F_2$ and $F_3$ are induced from the monoidal structure on a functor of taking the equivariant cohomology $H^*_{L^+G}(\Gr_G,-)$ as an $H^*_{L^+G}(*,\Qlb)\lmod$.
These constructions are compatible with the comparison functor in the proof of (\ref{eqn:DULAequiv}).
Hence, the natural isomorphism (\ref{eqn:diag:fiberfunctorintw}) is monoidal for $i=2$.

Similarly for the case $i=3$, since the ways of defining monoidal structures on $F_3$ and $F_4$ are parallel, and since we can define similar functors and monoidal structures on the spaces $\Spec \cal{O}_1$, $\Spec \ol{K(\Divtil)}$ and $\Spec \cal{O}_2$ in (\ref{eqn:O1O2}), we can show that the natural isomorphism (\ref{eqn:diag:fiberfunctorintw}) is monoidal for $i=3$.

Thus, we obtain a diagram
\begin{align}\label{eqn:diag:fiberfunctorcompati}
\xymatrix{
\cal{P}_1\ar[r]^-{F_1}\ar[d]^{\simeq}_{(\ref{eqn:SatPervequiv})}&\Vect^{\heartsuit,\rm{f.d.}}\ar@{=}[d]\\
\cal{P}_3\ar[r]^-{F_3}\ar[d]^{\simeq}_{(\ref{eqn:SatPervequiv})}&\Vect^{\heartsuit,\rm{f.d.}}\ar@{=}[d]\\
\cal{P}_{4}\ar[r]^-{F_{4}}&\Vect^{\heartsuit,\rm{f.d.}}
}
\end{align}
commutative up to monoidal natural isomorphisms.
Since $F_1,F_3,F_4$ are symmetric and faithful, it follows that the left vertical isomorphisms are symmetric.
Therefore, by Tannakian formalism, there are automorphisms $\varphi_1,\varphi_2\colon \wh{G}\to \wh{G}$ such that the diagram
\[
\xymatrix{
\Sat(\Hck_{G,\Spd C},\Qlb)\ar[r]\ar[d]_{\simeq}^{(\ref{eqn:SatPervequiv})}&\Rep^{\rm{f.d.}}(\wh{G},\Qlb)\ar[d]_{\simeq}^{(-)^{\varphi_1}}\\
\Perv(\Hck^{\Witt}_{G,\Spec k},\Qlb)\ar[r]\ar[d]_{\simeq}^{(\ref{eqn:SatPervequiv})}&\Rep^{\rm{f.d.}}(\wh{G},\Qlb)\ar[d]_{\simeq}^{(-)^{\varphi_2}}\\
\Perv(\Hck^{\eq}_{G,\Spec k},\Qlb)\ar[r]&\Rep^{\rm{f.d.}}(\wh{G},\Qlb)\\
}
\]
is commutative up to monoidal natural isomorphisms, where $(-)^{\varphi_i}$ is a functor twisting a $\wh{G}$-action by $\varphi_i$.
However, by (1), the intersection sheaves correspond under the left vertical arrows, and it is known that the intersection sheaf $IC_{\mu}$ corresponds to the irreducible representation $V_\mu$ under the horizontal arrows.
Therefore, the functors $(-)^{\varphi_1}$ and $(-)^{\varphi_2}$ send an irreducible representation with highest weight $\mu$ to itself.
This implies $\varphi_1,\varphi_2$ are inner automorphisms.
\end{enumerate}
\end{proof}
\subsection{Cohomology ring of Hecke stack}
From Proposition \ref{prop:DULAequiv}, we can compute various equivariant cohomologies:
\begin{cor}\label{cor:H*G}
There exists a canonical graded isomorphism
\[
H^*_{L^+_{\Spd C}G}(\Spd C,\Qlb)
\cong \cal{O}_{\wh{\fk{t}}^*/W},
\]
where $\cal{O}_{\wh{\fk{t}}^*/W}$ is graded so that the coordinates of $\wh{\fk{t}}^*$ are of degree 2.
\end{cor}
\begin{proof}
We can assume that $G=G^{\spl}_F$, and we abusively write $G$ for $G^{\spl}$.
Since $\mathbbm{1}:=i_{0*}\Qlb$ (defined in \S \ref{sssc:ULAsheavesSatcat}) is ULA, it follows from Proposition \ref{prop:DULAequiv} that
\[
H^*_{L^+_{\Spd C}G}(\Spd C)\cong H^*(\Hck_{G,\Spd C},\mathbbm{1})\cong H^*(\Hck^{\Witt}_{G,\Spec k},\mathbbm{1})\cong H^*_{L^{+,\Witt}_{\Spec k}G}(\Spec k).
\]
Since the first congruence group $L^{+,\Witt}_{\Spec k}G^{(1)}$ is (the perfection of) a pro-unipotent group and $L^{+,\Witt}_{\Spec k}G/L^{+,\Witt}_{\Spec k}G^{(1)}$ is the perfection of $G_{\Spec k}$, we have
\[
H^*_{L^{+,\Witt}_{\Spec k}G}(\Spec k)\cong H^*_{G}(\Spec k).
\]
The isomorphism $H^*_{G}(\Spec k)\cong \cal{O}_{\wh{\fk{t}}^*/W}$ is well-known.
\end{proof}
\begin{cor}\label{cor:H*Gm}
There exists a canonical graded isomorphism
\[
H^*_{L^+_{\Spd C}\bb{G}_m}(\Spd C,\Qlb)
\cong \Qlb[x],
\]
where $x$ is an indeterminate of degree 2 (, which is the Chern class of the canonical line bundle).
\end{cor}
\begin{cor}\label{cor:H*T}
There exists a canonical graded isomorphism
\[
H^*_{L^+_{\Spd C}T}(\Spd C,\Qlb)
\cong \Qlb[\wh{\fk{t}}], 
\]
where the elements of $\wh{\fk{t}}$ are of degree 2.
\end{cor}
The correspondence is as follows:
For a character $\alpha\colon T_{\ol{F}}\to \bb{G}_{m,\ol{F}}$, the image of $x$ under the induced map 
\[
\alpha^*\colon H^*_{L^+_{\Spd C}\bb{G}_{m}}(\Spd C,\Qlb)\to H^*_{L^+_{\Spd C}T}(\Spd C,\Qlb)
\]
corresponds to the image of $\alpha$ under the map $\bb{X}^*(T_{\ol{E}})\cong \bb{X}_*(\wh{T})\overset{\rm{Lie}}{\to}\Hom_{\Qlb}(\Qlb,\wh{\fk{t}}) =\wh{\fk{t}}$.
\begin{defi}
Assume that $G$ is semisimple and simply-connected.
We define the derived global section algebra of $\Hck_{G,\Spd C}$ by
\[
R\Gamma(\Hck_{G,\Spd C},\Qlb):=\underset{\substack{\longleftarrow \\ \mu\in \bb{X}_*^+(T_{\ol{F}})}}{R\lim}R\Gamma(\Hck_{G,\Spd C},i_{\leq \mu *}\Qlb)\in \Qlb\lmod
\]
where 
\[
R\Gamma(\Hck_{G,\Spd C},i_{\leq \mu *}\Qlb):=R\Hom_{\cal{D}_{\et}(\Hck_{G,\Spd C},\Qlb)}(i_{\leq \mu *}\Qlb, i_{\leq \mu *}\Qlb)\in \Qlb\lmod
\]
Moreover, we define the $i$-th cohomology group of $\Hck_{G,\Spd C}$ by 
\[
H^i(\Hck_{G,\Spd C},\Qlb):=H^i(R\Gamma(\Hck_{G,\Spd C},\Qlb)).
\]
We often write $H^i_{L^+_{\Spd C}G}(\Gr_{G,\Spd C},\Qlb)$ for $H^i(\Hck_{G,\Spd C},\Qlb)$.
\end{defi}
In fact, $H^i(\Hck_{G,\Spd C},\Qlb)$ can be written as the limit of $H^i(\Hck_{G,\Spd C,\leq \mu},\Qlb)$.
\begin{lemm}\label{lemm:H*twodef}
Assume that $G$ is semisimple and simply-connected.
The canonical homomorphism
\[
H^i(\Hck_{G,\Spd C},\Qlb)\to \lim_{\substack{\longleftarrow \\ \mu\in \bb{X}_*^+(T_{\ol{F}})}} H^i(\Hck_{G,\Spd C,\leq \mu},\Qlb)
\]
is an equivalence.
\end{lemm}
\begin{proof}
Since $G$ is semisimple and simply-connected, the index set $\bb{X}_*^+(T_{\ol{F}})$ is a countable directed set.
By Mittag--Leffler, it suffices to show that the homomorphism
\[
H^i(\Hck_{G,\Spd C,\leq \mu'},\Qlb)\to H^i(\Hck_{G,\Spd C,\leq \mu},\Qlb)
\]
is surjective for any $\mu'\geq \mu\in \bb{X}_*^+(T_{\ol{F}})$.
By Proposition \ref{prop:DULAequiv}, it reduces to the surjectivity of 
\begin{align}\label{eqn:H*Hckmu'tomu}
H^i(\Hck^{\eq}_{G^{\spl},\Spec k,\leq \mu'},\Qlb)\to H^i(\Hck^{\eq}_{G^{\spl},\Spec k,\leq \mu},\Qlb).
\end{align}
We know the explicit description of $H^i(\Gr^{\eq}_{G^{\spl},\Spec k},\Qlb)$ by \cite{GinPerv}, and one can show that $H^i(\Gr^{\eq}_{G^{\spl},\Spec k},\Qlb)$ is pure of Frobenius weight $i$.
By \cite[Theorem 1.5]{GinLoop}, it follows that $H^i(\Gr^{\eq}_{G^{\spl},\Spec k,\leq \mu},\Qlb)$ is also pure of weight $i$ for any $\mu \in \bb{X}_*^+(T_{\ol{F}})$.
By \cite[Theorem 14.1 (8)]{GKM}, it follows that there is a non-canonical isomorphism
\[
H^*(\Hck^{\eq}_{G^{\spl},\Spec k,\leq \mu},\Qlb)\cong H^*(\Gr^{\eq}_{G^{\spl},\Spec k,\leq \mu},\Qlb)\otimes H^*_{L^{+,\eq}_{\Spec k}G}(\Spec k,\Qlb).
\]
Therefore, by the graded Nakayama lemma, the surjectivity of (\ref{eqn:H*Hckmu'tomu}) reduces to the surjectivity of 
\[
H^i(\Gr^{\eq}_{G^{\spl},\Spec k,\leq \mu'},\Qlb)\to H^i(\Gr^{\eq}_{G^{\spl},\Spec k,\leq \mu},\Qlb),
\]
which follows from \cite[Theorem 1.5]{GinLoop}.
\end{proof}
\begin{prop}\label{prop:H*Hck}
Assume that $G$ is semisimple and simply-connected.
Then there is a canonical graded isomorphism
\begin{align}
H^*_{L^+_{\Spd C}G}(\Gr_{G,\Spd C},\Qlb)
\cong \cal{O}_{\bb{T}(\wh{\fk{t}}^*/W)},\label{eqn:H*Hck}
\end{align}
where  $\bb{T}(-)$ denotes the tangent bundle, $W$ is the Weyl group of $G$, and the grading of $\cal{O}_{\bb{T}(\wh{\fk{t}}^*/W)}$ is induced from that of $\cal{O}_{\bb{T}(\wh{\fk{t}}^*)}(\hookleftarrow \cal{O}_{\bb{T}(\wh{\fk{t}}^*/W)})$, which is graded so that the coordinates of $\wh{\fk{t}}^*$ in the base direction have degree 2 and those in the fiber direction have degree 0.
\end{prop}
\begin{proof}
First, we claim that $i_{\leq \mu *}\Lambda\in \cal{D}^{\ULA}(\Hck_{G,\Spd C},\Lambda)$.
We can assume that $\Lambda$ is a torsion ring, and in this case the claim follows since $i_{\leq \mu *}\Lambda$ satisfies the condition of \cite[{\sc Proposition} VI.6.5]{FS}.

Therefore, by Proposition \ref{prop:DULAequiv} and by the definition of $H^*_{L^+_{\Spd C}G}(\Gr_{G,\Spd C},\Qlb)$, it follows that
\[
H^*_{L^+_{\Spd C}G}(\Gr_{G,\Spd C},\Qlb)\cong H^*_{L^+_{\Spec k}G}(\Gr^{\eq}_{G,\Spec k},\Qlb).
\]
The claim follows since it holds that
\[
H^*_{L^{+,\eq}_{\Spec k}G}(\Gr^{\eq}_{G,\Spec k},\Qlb)\cong \cal{O}_{\bb{T}(\wh{\fk{t}}^*/W)}
\]
by \cite[Theorem 1]{BF}.
\end{proof}
In (\ref{eqn:H*Hck}), we do not consider the Weil action on both sides.
This will be discussed in the next section.
The rest of this subsection is devoted to some preparation.
The group $L^+_{\Spd C}T(\subset L^+_{\Spd C}G)$ acts on $\Gr_{G,\Spd C}$ and the image of $[\lambda]$ (defined in (\ref{eqn:[mu]})) is fixed under this action.
(Actually, all fixed points are of this form.)
Therefore we have the isomorphism
\begin{align}\label{eqn:prodisom}
\lim_{\substack{\longleftarrow \\ \mu\in \bb{X}^+_*(T_{\ol{F}})}}H^*_{L^+_{\Spd C}T}\left(\coprod_{\substack{\lambda\in \bb{X}_*(T_{\ol{F}})\\ [\lambda]\in \Gr_{G,\leq \mu}}}[\lambda],\Qlb\right)\
&\cong \prod_{\lambda\in \bb{X}_*(T_{\ol{F}})}H^*_{L^+_{\Spd C}T}([\lambda],\Qlb)\notag\\
&\cong \prod_{\lambda\in \bb{X}_*(T_{\ol{F}})} \Qlb[\wh{\fk{t}}].
\end{align}
Therefore, when $G$ is semisimple and simply-connected, there is a canonical homomorphism
\begin{align}
\cal{O}_{\bb{T}(\wh{\fk{t}}^*/W)}&\overset{(\ref{eqn:H*Hck})}{\cong} H^*_{L^+_{\Spd C}G}(\Gr_{G,\Spd C},\Qlb)\notag\\
&\cong \lim_{\substack{\longleftarrow \\ \mu\in \bb{X}^+_*(T_{\ol{F}})}}H^*_{L^+_{\Spd C}G}(\Gr_{G,\Spd C,\leq \mu},\Qlb)\notag\\
&\to \lim_{\substack{\longleftarrow \\ \mu\in \bb{X}^+_*(T_{\ol{F}})}}H^*_{L^+_{\Spd C}T}\left(\coprod_{\substack{\lambda\in \bb{X}_*(T_{\ol{F}})\\ [\lambda]\in \Gr_{G,\leq \mu}}}[\lambda],\Qlb\right)\notag \\
&\overset{(\ref{eqn:prodisom})}{\cong} \prod_{\lambda\in \bb{X}_*(T_{\ol{F}})} \Qlb[\wh{\fk{t}}].\label{eqn:H*Hckemb}
\end{align}
Then we have the following:
\begin{prop}\label{prop:H*Hckemb}
Assume that $G$ is semisimple and simply-connected.
The homomorphism (\ref{eqn:H*Hckemb}) agrees with the following composition:
\begin{align}\label{eqn:OTtWemb}
&\cal{O}_{\bb{T}(\wh{\fk{t}}^*/W)}\to\cal{O}_{\bb{T}(\wh{\fk{t}}^*)}
\cong \cal{O}_{\wh{\fk{t}}^*}[\wh{\fk{t}}]\notag\\
&\to \prod_{\lambda\in \bb{X}_*(T_{\ol{F}})} \cal{O}_{\wh{\fk{t}}^*}\cong \prod_{\lambda\in \bb{X}_*(T_{\ol{F}})} \Qlb[\wh{\fk{t}}].
\end{align}
Here the isomorphism in the first line is the one over an algebra $\cal{O}_{\wh{\fk{t}}^*}$ with respect to the $\cal{O}_{\wh{\fk{t}}^*}$-algebra structure on $\cal{O}_{\bb{T}(\wh{\fk{t}}^*)}$ coming from the natural projection $\bb{T}(\wh{\fk{t}}^*)\to \wh{\fk{t}}^*$ to the base space, and the arrow in the second line is given by
\[
\cal{O}_{\wh{\fk{t}}^*}[\wh{\fk{t}}]\ni a\cdot t_1\ldots t_n\mapsto (a \langle \lambda,t_1\rangle \cdots \langle \lambda,t_n\rangle)_{\lambda}\in \prod_{\lambda\in \bb{X}_*(T_{\ol{F}})} \cal{O}_{\wh{\fk{t}}^*}
\]
for $a\in \cal{O}_{\wh{\fk{t}}^*}$ and $t_1,\ldots,t_n\in \wh{\fk{t}}$.
Here the image of $\lambda$ under $\bb{X}_*(T_{\ol{F}})\cong \bb{X}^*(\wh{T})\overset{\rm{Lie}}{\to} \Hom_{\Qlb}(\wh{\fk{t}},\Qlb)=\wh{\fk{t}}^*$ is also written by $\lambda$.
In particular, the homomorphism (\ref{eqn:H*Hckemb}) is injective.
\end{prop}
\begin{proof}
Since all the homomorphisms involved are induced by pullback along a map of stacks, the homomorphism 
\[
H^*_{L^+_{\Spd C}G}(\Gr^{\eq}_{G,\Spd C},\Qlb)\to \prod_{\lambda\in \bb{X}_*(T_{\ol{F}})}H^*_{L^+_{\Spd C}T}([\lambda],\Qlb)
\]
and 
\begin{align}\label{eqn:H*toprodeq}
H^*_{L^{+,\eq}_{\Spec k}G^{\spl}}(\Gr_{G^{\spl},\Spec k},\Qlb)\to \prod_{\lambda\in \bb{X}_*(T^{\spl})}H^*_{L^+_{\Spec k}T^{\spl}}([\lambda],\Qlb)
\end{align}
are compatible under the isomorphisms in the proofs of Proposition \ref{prop:H*Hck} and Corollary \ref{cor:H*T}.
Therefore, the theorem reduces to a similar statement for (\ref{eqn:H*toprodeq}).
One can show this statement by looking carefully at the proof of \cite[Theorem 1]{BF}.
The sketch is as follows:
By abuse of notation, write $G$ for $G^{\spl}$, $T$ for $T^{\spl}$, and so on.
We also omit the subscript $\Spec k$ and the superscript $\eq$.
In \cite[Theorem 1]{BF}, the isomorphism $H^*_{L^+G}(\Gr_G,\Qlb)\cong \cal{O}_{\bb{T}(\wh{\fk{t}}^*/W)}$ is obtained by taking $\hbar=0$ in an isomorphism 
\begin{align}\label{eqn:H*GmHck}
H^*_{L^+G\rtimes \bb{G}_m}(\Gr_G,\Qlb)\cong \cal{O}_{N_{(\wh{\fk{t}}^*/W)^2}\Delta}
\end{align}
over $H^*_{\bb{G}_m}(\Spec k,\Qlb)\cong \bb{A}^1=\Spec \Qlb[\hbar]$.
Here $\bb{G}_m$ acts on $L^+G$ and $LG$ by the loop rotation, and for a closed subscheme $Z$ of a scheme $X$, the space $N_XZ$ is Fulton's deformation to the normal cone.
The isomorphism (\ref{eqn:H*GmHck}) is induced by 
\begin{align*}
&\cal{O}_{(\wh{\fk{t}}^*/W)^2\times \bb{A}^1}\cong H^*_{G\times \bb{G}_m\times G}(\Spec k,\Qlb) \\
&\cong H^*_{L^+G\rtimes \bb{G}_m\ltimes L^+G}(\Spec k,\Qlb)\to H^*_{L^+G\rtimes \bb{G}_m}(\Gr_G,\Qlb),
\end{align*}
where the right homomorphism is the pullback via $\Gr_G\to [\Spec k/L^+G]$.
We consider the composition of this homomorphism with the pullback homomorphism $H^*_{L^+G\rtimes \bb{G}_m}(\Gr_G,\Qlb)\to H^*_{L^+T\rtimes \bb{G}_m}([\lambda],\Qlb)$.
By the argument of the proof of \cite[Lemma 1]{BF}, the pullback
\[
H^*_{G\times \bb{G}_m\times G}(\Spec k,\Qlb)\to H^*_{L^+T\rtimes \bb{G}_m}([\lambda],\Qlb)
\]
is isomorphic to the homomorphism $f_{\lambda}\colon \cal{O}_{(\wh{\fk{t}}^*/W)^2\times \bb{A}^1}\to \cal{O}_{\Gamma_{\lambda}}$, where 
\[
\Gamma_{\lambda}:=\{(x_1,x_2,a)\in (\wh{\fk{t}}^*)^2\times \bb{A}^1\mid x_2=x_1+a\cdot \lambda\}
\]
and $f_{\lambda}$ corresponds to the map $\Gamma_{\lambda} \subset (\wh{\fk{t}}^*)^2\times \bb{A}^1\to (\wh{\fk{t}}^*/W)^2\times \bb{A}^1$.
This induces a map $\Gamma_{\lambda}\to N_{(\wh{\fk{t}}^*/W)^2}\Delta$.
Restricting to the locus $\hbar=0$, it becomes a map
\begin{align*}
\wh{\fk{t}}^*&\to \wh{\fk{t}}^*\times \wh{\fk{t}}^*\cong \bb{T}(\wh{\fk{t}}^*)\to \bb{T}(\wh{\fk{t}}^*/W), \\
x&\mapsto (\lambda,x)
\end{align*}
over $\wh{\fk{t}}^*/W$.
This implies the claim.
\end{proof}
By the same argument, we can show the following lemma:
\begin{lemm}\label{lemm:H*GtoH*T}
The composition 
\begin{align*}
\cal{O}_{\wh{\fk{t}}^*/W}\cong H^*_{L^+_{\Spd C}G}(\Spd C)\to H^*_{L^+_{\Spd C}T}(\Spd C)\cong \cal{O}_{\wh{\fk{t}}^*}
\end{align*}
 agrees with the homomorphism induced by the natural map $\wh{\fk{t}}^*\to \wh{\fk{t}}^*/W$.
In particular, it is injective.
\end{lemm}
\begin{lemm}\label{lemm:H*GtoH*Hck}
Assume that $G$ is semisimple and simply-connected.
Let 
\[
\pi_1,\pi_2\colon \Hck_{G,\Spd C}\to [\Spd C/L^+_{\Spd C}G]
\]
be the projections corresponding to the quotient by the right action and the left action, respectively.
For $i=1,2$, the composition 
\begin{align*}
\cal{O}_{\wh{\fk{t}}^*/W}&\cong H^*_{L^+_{\Spd C}G}(\Spd C,\Qlb)
\overset{\pi_i^*}{\to} H^*(\Hck_{G,\Spd C},\Qlb)\cong \cal{O}_{\bb{T}(\wh{\fk{t}}^*/W)}
\end{align*}
agrees with the homomorphism induced by the natural projection $\bb{T}(\wh{\fk{t}}^*/W)\to \wh{\fk{t}}^*/W$.
In particular, it is injective.
\end{lemm}
Moreover, recall the convolution diagram (\ref{eqn:conv}) and put
\[
\Hck_{G,\Spd C,\leq (\mu_1, \mu_2)}^{(2)}:=p^{-1}(\Hck_{G,\Spd C,\leq \mu_1}\times \Hck_{G,\Spd C,\leq \mu_2}).
\]
Let $i_{\leq (\mu_1,\mu_2)}\colon \Hck_{G,\Spd C,\leq (\mu_1, \mu_2)}^{(2)}\to \Hck_{G,\Spd C}^{(2)}$ be the closed immersion, and put $m_{\leq (\mu_1,\mu_2)}:=m\circ i_{\leq (\mu_1,\mu_2)}$.
By the definition of $-\star -$, we have a homomorphism
\begin{align*}
R\Gamma(\Hck_{G,\Spd C, \leq (\mu_1+\mu_2)},\Qlb)&\to R\Gamma(\Hck_{G,\Spd C, \leq (\mu_1+\mu_2)},Rm_{\leq (\mu_1,\mu_2) *}\Qlb)\\
&=R\Gamma(\Hck_{G,\Spd C}, i_{\leq \mu_1*}\Qlb\star i_{\leq \mu_2*}\Qlb),
\end{align*}
and a homomorphism
\begin{align}
&R\Gamma(\Hck_{G,\Spd C, \leq \mu_1},\Qlb)\otimes^{\bb{L}}_{R\Gamma([\Spd C/L^+_{\Spd C}G])}R\Gamma(\Hck_{G,\Spd C, \leq \mu_2},\Qlb)\notag\\
&\to R\Gamma(\Hck_{G,\Spd C, \leq \mu_1}\times_{[\Spd C/L^+_{\Spd C}G]}\Hck_{G,\Spd C, \leq \mu_2},\Qlb) \notag\\
&=R\Gamma(\Hck_{G,\Spd C, \leq (\mu_1,\mu_2)}^{(2)},\Qlb)\notag\\
&\simeq R\Gamma(\Hck_{G,\Spd C, \leq (\mu_1+\mu_2)},Rm_{\leq (\mu_1,\mu_2) *}\Qlb)\notag\\
&=R\Gamma(\Hck_{G,\Spd C}, i_{\leq \mu_1*}\Qlb\star i_{\leq \mu_2*}\Qlb)\label{eqn:RGammaHckconv}
\end{align}
where the first map is given by applying the following lemma with $X=Y=\Hck_{G,\Spd C}$, $Z=[\Spd C/L^+_{\Spd C}G]$ and with $f,g\colon \Hck_{G,\Spd C}\to [\Spd C/L^+_{\Spd C}G]$ the map corresponds to the right and left actions, respectively.
\begin{lemm}\label{lemm:RGammaKunnethformal}
For any cartesian diagram
\[
\xymatrix{
X\times_ZY\ar[r]^-{\rm{pr}_1}\ar[d]_{\rm{pr}_2}&X\ar[d]^{f}\\
Y\ar[r]_-{g}&Z
}
\]
of small v-stacks and for any $A\in \cal{D}_{\et}(X,\Lambda)$ and $B\in \cal{D}_{\et}(Y,\Lambda)$, there is a canonical homomorphism
\[
R\Gamma(X,A)\otimes^{\bb{L}}_{R\Gamma(Z,\Lambda)}R\Gamma(Y,B)\to R\Gamma(X\times_ZY,A\boxtimes B).
\]
\end{lemm}
\begin{proof}
This can be proved by a formal argument using six functors.
Namely, we can assume that $\Lambda$ is a torsion ring.
There is a homomorphism
\begin{align*}
&(f\circ \rm{pr}_1)^*(Rf_*A\otimes^{\bb{L}}Rg_*B)\\
&\cong (f\circ \rm{pr}_1)^*Rf_*A\otimes^{\bb{L}}(f\circ \rm{pr}_1)^*Rg_*B\\
&\cong (f\circ \rm{pr}_1)^*Rf_*A\otimes^{\bb{L}}(g\circ \rm{pr}_2)^*Rg_*B\\
&\cong f^*Rf_*A\boxtimes g^*Rg_*B\\
&\to A\boxtimes B
\end{align*}
and by passing to the adjoint, we have
\[
Rf_*A\otimes^{\bb{L}}Rg_*B\to R(f\circ \rm{pr}_1)_*(A\boxtimes B).
\]
Therefore, if we show the lemma for $X=Y=Z$, then the desired homomorphism is given by
\begin{align*}
R\Gamma(X,A)\otimes^{\bb{L}}_{R\Gamma(Z,\Lambda)}R\Gamma(Y,B)&\cong R\Gamma(Z,Rf_*A)\otimes^{\bb{L}}_{R\Gamma(Z,\Lambda)}R\Gamma(Z,Rg_*B)\\
&\to R\Gamma(Z,Rf_*A\otimes^{\bb{L}}Rg_*B)\\
&\to R\Gamma(X\times_ZY,A\boxtimes B).
\end{align*}
Thus, we can assume that $X=Y=Z$.
In this case, the claim follows from the fact that the tensor product $A\otimes^{\bb{L}}B$ as sheaves valued in the derived category $\Lambda\lmod$ agrees with a sheafification of a presheaf of the form
$U\mapsto R\Gamma(U,A)\otimes^{\bb{L}}_{R\Gamma(U,\Lambda)}R\Gamma(U,B)$.
\end{proof}
This defines homomorphisms
{\footnotesize
\begin{align*}
\xymatrix{
&H^*(\Hck_{G,\Spd C,\leq \mu_1},\Qlb)\otimes_{R_G}H^*(\Hck_{G,\Spd C,\leq \mu_1},\Qlb)\ar[d]\\
H^*(\Hck_{G,\Spd C,\leq (\mu_1+\mu_2)},\Qlb)\ar[r]&H^*(\Hck_{G,\Spd C},i_{\leq \mu_1*}\Qlb\star i_{\leq \mu_2*}\Qlb),
}
\end{align*}
}
where $R_G:=H^*([\Spd C/L^+_{\Spd C}G])$.
Let
\begin{align}\label{eqn:defofmT}
m_{\bb{T}}\colon \bb{T}(\wh{\fk{t}}^*/W)\times_{\wh{\fk{t}}^*/W}\bb{T}(\wh{\fk{t}}^*/W)\to \bb{T}(\wh{\fk{t}}^*/W)
\end{align}
be a morphism of vector bundles over $\wh{\fk{t}}^*/W$, which is fiberwise given by addition.
This defines a homomorphism
\begin{align*}
m_{\bb{T}}^*\colon \cal{O}_{\bb{T}(\wh{\fk{t}}^*/W)}\to \cal{O}_{\bb{T}(\wh{\fk{t}}^*/W)}\otimes_{\cal{O}_{\wh{\fk{t}}^*/W}}\cal{O}_{\bb{T}(\wh{\fk{t}}^*/W)}.
\end{align*}
Then we can show the following:
\begin{lemm}\label{lemm:H*conv}
The diagram
\[
\xymatrix{
\cal{O}_{\bb{T}(\wh{\fk{t}}^*/W)}\ar[dd]\ar[r]^-{m_{\bb{T}}^*}&\cal{O}_{\bb{T}(\wh{\fk{t}}^*/W)}\otimes_{\cal{O}_{\wh{\fk{t}}^*/W}}\cal{O}_{\bb{T}(\wh{\fk{t}}^*/W)}\ar[d]\\
&H^*(\Hck_{G,\Spd C,\leq \mu_1},\Qlb)\otimes_{R_G}H^*(\Hck_{G,\Spd C,\leq \mu_1},\Qlb)\ar[d]\\
H^*(\Hck_{G,\Spd C,\leq (\mu_1+\mu_2)},\Qlb)\ar[r]&H^*(\Hck_{G,\Spd C},i_{\leq \mu_1*}\Qlb\star i_{\leq \mu_2*}\Qlb)
}
\]
commutes, where the left vertical map is given by Proposition \ref{prop:H*Hck} (composed with the pullback $H^*(\Hck_{G,\Spd C},\Qlb)\to H^*(\Hck_{G,\Spd C,\leq (\mu_1+\mu_2)},\Qlb)$), and similarly, the upper right map is given by Proposition \ref{prop:H*Hck}, Corollary \ref{cor:H*G}, and Lemma \ref{lemm:H*GtoH*Hck}.
\end{lemm}
\begin{proof}
All functors used to define the involved homomorphisms are pullbacks under maps of small v-stacks or pushforwards by proper maps, such as $\Hck_{G,\Spd C,\leq \mu}\to [\Spd C/L^+_{\Spd C}G]$ or $m_{\leq (\mu_1,\mu_2)}$.
Similar maps exist for $\Hck_{G,\Spd \cal{O}_C}$, $\Hck_{G,\Spd k}$, $\Hck_{G,\Spec k}^{\Witt}$, or $\Hck_{G,\Spec k}^{\eq}$.
Therefore, by the same argument as the proof of Proposition \ref{prop:H*Hck} using Proposition \ref{prop:DULAequiv}, the proper base change and \cite[Proposition 27.5]{Sch}, the claim reduces to the equal characteristic case.
Then, one can show this by looking carefully at the proof of \cite[Theorem 1]{BF}.
The sketch is as follows:

By abuse of notation, write $G$ for $G^{\spl}$, $T$ for $T^{\spl}$, and so on.
We also omit the subscript $\Spec k$ and the superscript $\eq$.
First, define the convolution product $-\star -$ in $\Det([*/(L^+G\rtimes \bb{G}_m\ltimes L^+G)])=\Det([*/(L^+G\rtimes \bb{G}_m)]\times_{[*/\bb{G}_m]} [*/(L^+G\rtimes \bb{G}_m)])$ using the convolution diagram
\begin{align*}
&([*/(L^+G\rtimes \bb{G}_m)]\times_{[*/\bb{G}_m]} [*/(L^+G\rtimes \bb{G}_m)])\times ([*/(L^+G\rtimes \bb{G}_m)]\times_{[*/\bb{G}_m]} [*/(L^+G\rtimes \bb{G}_m)])\\
&\overset{\rm{pr}_{12}\times \rm{pr}_{23}}\leftarrow [*/(L^+G\rtimes \bb{G}_m)]\times_{[*/\bb{G}_m]} [*/(L^+G\rtimes \bb{G}_m)]\times_{[*/\bb{G}_m]} [*/(L^+G\rtimes \bb{G}_m)]\\
&\overset{\rm{pr}_{13}}\to [*/(L^+G\rtimes \bb{G}_m)]\times_{[*/\bb{G}_m]} [*/(L^+G\rtimes \bb{G}_m)]
\end{align*}
Then we can show the diagram
\[
\xymatrix{
\cal{O}_{(\wh{\fk{t}}^*/W)^2\times \bb{A}^1}\ar[dd]^{\cong}\ar[r]^-{m_1^*}&\cal{O}_{(\wh{\fk{t}}^*/W)^2\times \bb{A}^1}\otimes_{\cal{O}_{\wh{\fk{t}}^*/W\times \bb{A}^1}}\cal{O}_{(\wh{\fk{t}}^*/W)^2\times \bb{A}^1}\ar[d]\\
&H^*_{L^+G\rtimes \bb{G}_m\ltimes L^+G}(*)\otimes_{H^*_{L^+G\rtimes \bb{G}_m}(*)}H^*_{L^+G\rtimes \bb{G}_m\ltimes L^+G}(*)\ar[d]\\
H^*_{L^+G\rtimes \bb{G}_m\ltimes L^+G}(*,\Qlb)\ar[r]&H^*_{L^+G\rtimes \bb{G}_m\ltimes L^+G}(*,\Qlb\star \Qlb)
}
\]
commutes, where $m_1$ is defined by 
\begin{align*}
((\wh{\fk{t}}^*/W)^2\times \bb{A}^1)\times_{\wh{\fk{t}}^*/W\times \bb{A}^1}((\wh{\fk{t}}^*/W)^2\times \bb{A}^1)&\overset{m_1}\to(\wh{\fk{t}}^*/W)^2\times \bb{A}^1, \\
((x_1,x_2,a),(x_2,x_3,a))&\mapsto (x_1,x_3,a).
\end{align*}
The map $m_1$ induces a map
\begin{align*}
m_2\colon (N_{(\wh{\fk{t}}^*/W)^2\times \bb{A}^1}\Delta) \times_{(\wh{\fk{t}}^*/W)\times \bb{A}^1} (N_{(\wh{\fk{t}}^*/W)^2\times \bb{A}^1}\Delta)\to N_{(\wh{\fk{t}}^*/W)^2\times \bb{A}^1}\Delta.
\end{align*}
From this, we can show the commutativity of the diagram of the form
\[
\xymatrix{
\cal{O}_{N_{(\wh{\fk{t}}^*/W)^2}\Delta}\ar[dd]^{\cong}\ar[r]^-{m_2^*}&\cal{O}_{N_{(\wh{\fk{t}}^*/W)^2}\Delta}\otimes_{\cal{O}_{\wh{\fk{t}}^*/W\times \bb{A}^1}}\cal{O}_{N_{(\wh{\fk{t}}^*/W)^2}\Delta}\ar[d]\\
&H^*_{\bb{G}_m}(\Hck_{G,\leq \mu_1},\Qlb)\otimes_{H^*([\Spec k/L^+G])}H^*_{\bb{G}_m}(\Hck_{G,\leq \mu_2},\Qlb)\ar[d]\\
H^*_{\bb{G}_m}(\Hck_{G,\leq (\mu_1+\mu_2)},\Qlb)\ar[r]&H^*_{\bb{G}_m}(\Hck_{G},i_{\leq \mu_1*}\Qlb\star i_{\leq \mu_2*}\Qlb).
}
\]
Restricted to the locus $\hbar =0$, the map $m_2$ agrees with the map $m_{\bb{T}}$.
This shows the lemma.
\end{proof}
\subsection{Weil action on cohomology ring of Hecke stack}\label{ssc:WeilonCohHck}
In this section, we determine how the Weil actions correspond under (\ref{eqn:H*Hck}).

For a semisimple and simply-connected group $G$, the object $R\Gamma(\Hck_{G,\Spd C},\Qlb)$ has the structure of an object in $\DDRep(W_F,\Qlb)$ as follows:
In $\Qlb\lmod$, we have
\begin{align}
R\Gamma(\Hck_{G,\Spd C},\Qlb)&:=R\lim_{\substack{\longleftarrow\\\mu}}R\Gamma(\Hck_{G,\Spd C,\leq \mu},\Qlb)\notag\\
&\simeq \lim_{\substack{\longrightarrow\\ n}}\tau^{\leq n}R\lim_{\substack{\longleftarrow\\\mu}}R\Gamma(\Hck_{G,\Spd C,\leq \mu},\Qlb)\notag\\
&\simeq \lim_{\substack{\longrightarrow\notag\\ n}}R\lim_{\substack{\longleftarrow\\\mu}}\tau^{\leq n}R\Gamma(\Hck_{G,\Spd C,\leq \mu},\Qlb)\notag\\
&\simeq \lim_{\substack{\longrightarrow\\ n}}R\lim_{\substack{\longleftarrow\\\mu}}((\tau^{\leq n}R\Gamma(\Hck_{G,\Spd C,\leq \mu},\ZZ_{\ell}))\otimes_{\ZZ_{\ell}}\Qlb)\label{eqn:RHckcalculation}
\end{align}
Since $\Div^1\cong [\Spd C/\underline{W_{F}}]$ (see \cite[\S IV.7]{FS}), the $\Hck_{G,\Spd C}$ has an action of $W_F$.
The $I_E$-action restricts to $\Hck_{G,\Spd C,\leq \mu}$, and we have
\[
R\Gamma(\Hck_{G,\Spd C,\leq \mu},\Lambda)\in \cal{D}(\Rep^{\rm{sm}}(I_E,\Lambda))
\]
for a torsion ring $\Lambda$.
Hence, the object
\[
R\Gamma(\Hck_{G,\Spd C,\leq \mu},\ZZ_{\ell})=R\lim_{\substack{\longleftarrow\\ i}}R\Gamma(\Hck_{G,\Spd C,\leq \mu},\ZZ/\ell^i)
\]
is regarded as an object in $\ds\lim_{\substack{\longleftarrow\\ i}}\cal{D}(\Rep^{\rm{sm}}(I_E,\ZZ_{\ell}/\ell^i)))$.
Since all cohomologies of $\tau^{\leq n}R\Gamma(\Hck_{G,\Spd C,\leq \mu},\ZZ_{\ell})$ are finitely generated over $\ZZ_{\ell}$, it is perfect over $\ZZ_{\ell}$.
Therefore, we have
\[
\tau^{\leq n}R\Gamma(\Hck_{G,\Spd C,\leq \mu},\ZZ_{\ell})\in \DDRep^{\perf}(I_E,\ZZ_{\ell})
\]
by Corollary \ref{cor:DReperfcharacterizationOL}.
Thus, the object
\[
\tau^{\leq n}R\Gamma(\Hck_{G,\Spd C,\leq \mu},\Qlb)=(\tau^{\leq n}R\Gamma(\Hck_{G,\Spd C,\leq \mu},\ZZ_{\ell})) \otimes_{\ZZ_{\ell}}\Qlb
\]
can be regarded as an object in $\DDRep^{\perf}(I_E,\Qlb)$.
By Lemma \ref{lemm:H*twodef} and Proposition \ref{prop:H*Hck}, the limit
\[
H^i(\Hck_{G,\Spd C},\Qlb)=\lim_{\substack{\longleftarrow \\\mu}} H^i(\Hck_{G,\Spd C,\leq \mu},\Qlb)
\]
is finite-dimensional, and by the proof of Lemma \ref{lemm:H*twodef}, the transition maps of this limit are surjective.
Thus,
\[
\lim_{\substack{\longleftarrow \\\mu}} H^i(\Hck_{G,\Spd C,\leq \mu},\Qlb)\simeq H^i(\Hck_{G,\Spd C,\leq \mu'},\Qlb)
\]
for some $\mu'$, and
\[
R\lim_{\substack{\longleftarrow \\\mu}} \tau^{\leq n} R\Gamma(\Hck_{G,\Spd C,\leq \mu},\Qlb)\simeq \tau^{\leq n} R\Gamma(\Hck_{G,\Spd C,\leq \mu'},\Qlb)
\]
for some $\mu'$.
Therefore, we can regard the latter as an object in $\DDRep^{\perf}(I_E,\Qlb)$, and by (\ref{eqn:RHckcalculation}), $R\Gamma(\Hck_{G,\Spd C},\Qlb)$ as an object in $\DDRep(I_E,\Qlb)$.

Moreover, by restricting the index set of $\mu$ to a cofinal set of Weil-invariant dominant cocharacters, we see that every object above belongs to the $W_F/I_E$-invariants of each category.
Hence, we have $R\Gamma(\Hck_{G,\Spd C},\Qlb)\in \DDRep(W_F,\Qlb)$.
By these constructions, its $i$-th cohomology
\[
H^i(\Hck_{G,\Spd C},\Qlb):=\tau^{\geq i}\tau^{\leq i}R\Gamma(\Hck_{G,\Spd C},\Qlb)
\]
has the structure of an object in $\DDRep^{\perf}(W_F,\Qlb)^{\heartsuit}=\Rep^{\rm{f.d.}}(W_F,\Qlb)$.
Similarly, 
\[
H^i([\Spd C/L^+_{\Spd C}G],\Qlb)=H^i(\Hck_{G,\Spd C},i_{0*}\Qlb)
\]
also has the structure of an object in $\Rep^{\rm{f.d.}}(W_F,\Qlb)$.

We first upgrade Corollary \ref{cor:H*Gm} to include Weil actions.
\begin{lemm}\label{lemm:H*GmWequiv}
There exists a canonical graded isomorphism of $W_{F}$-representations
\[
H^*_{L^+_{\Spd C}\bb{G}_m}(\Spd C,\Qlb)\cong \rm{Sym(\Qlb(-1)[-2])}, 
\]
where $(i)$ denotes the $i$-th Tate twist.
\end{lemm}
\begin{proof}
\begin{enumerate}
\item First, we claim that the inertia group $I_{F}\subset W_{F}$ acts trivially.
Since $L^+_{\Spd \cal{O}_C}\bb{G}_m$ is defined over $\Spd \cal{O}_{\breve{F}}$ and $\Spd \cal{O}_C\to \Spd \cal{O}_{\breve{F}}$ factors through the stack $[(\Spd \cal{O}_C)/I_{F}]$, the action of $I_{F}$ on $[\Spd C/L^+_{\Spd C}\bb{G}_m]$ extends to the action on $[\Spd \cal{O}_C/L^+_{\Spd \cal{O}_C}\bb{G}_m]$.
This is restricted to the action of $[\Spd k/L^+_{k}\bb{G}_m]$, which is trivial.
These actions are compatible with the isomorphism
\[
H^*_{L^+_{\Spd k}\bb{G}_m}(\Spd k)\overset{\sim}{\leftarrow}
H^*_{L^+_{\Spd \cal{O}_C}\bb{G}_m}(\Spd \cal{O}_C)\overset{\sim}{\rightarrow}H^*_{L^+_{\Spd C}\bb{G}_m}(\Spd C)
\]
in the proof of Corollary \ref{cor:H*Gm}, and the action on $H^*_{L^+_{\Spd k}\bb{G}_m}(\Spd k)$ is trivial.
This implies the claim.
\item We show that a Frobenius lift $\sigma\in W_{F}$ acts as $x\mapsto q^{-1}x$, where $x$ is as in Corollary \ref{cor:H*Gm}.
Using the fact that any locally constant complex on $\Spd \cal{O}_{\breve{F}}$ is constant and applying the same argument as the proof of Proposition \ref{prop:DULAequiv} and Corollary \ref{cor:H*Gm}, we can show that the pullback homomorphism
\[
H^*_{L^+_{\Spd \cal{O}_{\breve{F}}}\bb{G}}(\Spd \cal{O}_{\breve{F}})\to H^*_{L^+_{\Spd k}\bb{G}}(\Spd k) 
\]
and the composition 
\[
H^*_{L^+_{\Spd \cal{O}_{\breve{F}}}\bb{G}}(\Spd \cal{O}_{\breve{F}})\to H^*_{L^+_{\Spd \breve{F}}\bb{G}}(\Spd \breve{F})\to H^*_{L^+_{\Spd C}\bb{G}}(\Spd C)
\]
are isomorphisms.
The action of $\sigma\in W_{F}$ on $[\Spd C/L^+_{\Spd C}\bb{G}]$ descends to the action on $[\Spd \breve{F}/L^+_{\Spd \breve{F}}\bb{G}]$ (since $\Spd C\to \Div^1$ factors through $\Spd \breve{F}$).
Recalling the equalities $\Spd \breve{F}=\Spd F\times_{\Spd k_{F}} \Spd k$ and $\Div^1_{k_{F}}=\Spd F/\rm{Frob}_{\Spd F}$, it follows that the action of $\sigma$ on $[\Spd \breve{F}/L^+_{\Spd \breve{F}}\bb{G}]$ is $\rm{id}\times \rm{Frob}_{\Spd F}\times \rm{id}$ under the identification
\begin{align}\label{eqn:SpdEbreveLGm}
[\Spd \breve{F}/L^+_{\Spd \breve{F}}\bb{G}_m]\cong [\Div^1_{k_{F}}/L^+_{\Div^1_{k_{F}}}\bb{G}_m]\times_{\Div^1_{k_{F}}}\Spd F\times_{\Spd k_{F}}\Spd k.
\end{align}
However, we have the following lemma, which easily follows from the definition:
\begin{lemm}
For a scheme $Z$ over $F$, the absolute Frobenii of $L^+_{\Div^1_{k_{F}}}Z$ and $L_{\Div^1_{k_{F}}}Z$ are the identities.
\end{lemm}
\begin{proof}
For a perfectoid space $S$ over $k_{F}$, the set of $S$-valued points of $L^+_{\Div^1_{k_{F}}}Z$ (resp. $L_{\Div^1_{k_{F}}}Z$) is by definition the set of pairs of the Cartier divisor $D$ of the Fargues--Fontaine curve $X_S$, and an element of $Z(B^+_{\Div^1}(S))$ (resp. $Z(B_{\Div^1}(S))$).
For a morphism $S\to T$ of perfectoid spaces, the induced map on the set of valued points is given by the pullback of divisors and elements.
However, the pullback by the absolute Frobenius $S\to S$ does not change these data since $X_S$ is already the quotient by the Frobenius.
\end{proof}
Therefore, the action of $\sigma$ on (\ref{eqn:SpdEbreveLGm}) is given by the relative Frobenius over $\Spd k$.
This action extends to the relative Frobenius on $[\Spd \cal{O}_{\breve{F}}/L^+_{\Spd \cal{O}_{\breve{F}}}\bb{G}_m]$ over $k$, and restricts to the relative Frobenius of $[\Spd k/L^+_{\Spd k}\bb{G}_m]$ over $k$.
It suffices to show that the relative Frobenius on $[\Spd k/L^+_{\Spd k}\bb{G}_m]$ induces an action on the cohomology of the form $x\mapsto q^{-1}x$.
This is well-known once it is reduced to the schematic case.
\end{enumerate}
The lemma follows from (1) and (2).
\end{proof}
This allows us to enhance Corollary \ref{cor:H*T} to a Weil-equivariant form:
\begin{cor}\label{cor:H*TWequiv}
There exists a canonical graded isomorphism of $W_{F}$-representations
\[
H^*_{L^+_{\Spd C}T}(\Spd C,\Qlb)\cong \rm{Sym}(\wh{\fk{t}}(-1)[-2]),
\]
where $\wh{\fk{t}}$ has a $W_{F}$-action $\rm{act}^{\rm{geom}}_{\wh{\fk{t}}}$ (defined in \S \ref{ssc:dualgp}).
\end{cor}
\begin{proof}
For $\alpha\in \bb{X}^*(T_{\ol{F}})$, put $\alpha^w =w^{-1} \circ \alpha \circ w$, i.e. $\alpha^w$ is a character that makes the following square commute:
\[
\xymatrix{
T_{\ol{F}}\ar[r]^-{\alpha}\ar@{<-}[d]_{w}&\bb{G}_{m, \ol{F}} \ar@{<-}[d]^w\\
T_{\ol{F}}\ar[r]^-{\alpha^w}&\bb{G}_{m, \ol{F}}.
}
\]
This induces the square
\[
\xymatrix{
H^*_{L^+_{\Spd C}\bb{G}_{m}}(\Spd C)\ar[r]^-{\alpha^{*}}\ar[d]_{w^*}&H^*_{L^+_{\Spd C}T}(\Spd C) \ar[d]^{w^*}\\
H^*_{L^+_{\Spd C}\bb{G}_{m}}(\Spd C)\ar[r]^-{(\alpha^{w})^*}&H^*_{L^+_{\Spd C}T}(\Spd C).
}
\]
Therefore, by the paragraph after Corollary \ref{cor:H*T}, when we write $d\alpha(1)$ for the image of $\alpha$ under the map $\bb{X}^*(T^{\spl})\cong \bb{X}_*(\wh{T})\overset{\rm{Lie}}{\to}\Hom(\Qlb,\wh{\fk{t}}) =\wh{\fk{t}}$, we have
\[
w^*(d\alpha(1))=w^*(\alpha^*(x))=\alpha^{w,*}(w^*(x))\overset{\text{Lemma \ref{lemm:H*GmWequiv}}}{=}\alpha^{w,*}(q^{-\deg w}x)=q^{-\deg w} d\alpha^w(1),
\]
where $\deg:W_{F}\to \bb{Z}$ is the projection.
This proves the claim since $d\alpha(1)\mapsto d\alpha^w(1)$ gives the usual $W_{F}$-action on $\wh{\fk{t}}$.
\end{proof}
For $w\in W_{F}$, the map $\rm{act}^{\rm{geom}}_{\wh{\fk{t}}^*}(w)\colon \wh{\fk{t}}^*\to \wh{\fk{t}}^*$ (defined in \S \ref{ssc:dualgp})
induces a map $f^w \colon \wh{\fk{t}}^*/W \to \wh{\fk{t}}^*/W$.
Let $q^{\deg w}\colon \wh{\fk{t}}^*/W \to \wh{\fk{t}}^*/W$ be the map induced by the multiplication by $q^{\deg w}$ on $\wh{\fk{t}}^*$.
The map $f^w\circ q^{\deg w}$ induces a homomorphism
\[
a_{\cal{O}_{\wh{\fk{t}}^*/W}}(w)\colon \cal{O}_{\wh{\fk{t}}^*/W}\to \cal{O}_{\wh{\fk{t}}^*/W}.
\]
Let $q^{\deg w}_{\rm{fib}}\colon \bb{T}(\wh{\fk{t}}^*/W)\to \bb{T}(\wh{\fk{t}}^*/W)$ be the map that applies $q^{\deg w}$ on each fiber (being the identity on the base).
The derivative $d(f^w\circ q^{\deg w})\colon \bb{T}(\wh{\fk{t}}^*/W)\to \bb{T}(\wh{\fk{t}}^*/W)$ composed with $(q^{\deg w}_{\rm{fib}})^{-1}$ induces a homomorphism
\begin{align}\label{eqn:defofaOTtW}
a_{\cal{O}_{\bb{T}(\wh{\fk{t}}^*/W)}}(w)\colon \cal{O}_{\bb{T}(\wh{\fk{t}}^*/W)}\to \cal{O}_{\bb{T}(\wh{\fk{t}}^*/W)}
\end{align}
Now we can state the main theorem of this subsection.
\begin{thm}\label{thm:H*HckWeil}
Assume that $G$ is semisimple and simply-connected.
When we endow $\cal{O}_{\bb{T}(\wh{\fk{t}}^*/W)}$ with the $W_{F}$-action defined by $a_{\cal{O}_{\bb{T}(\wh{\fk{t}}^*/W)}}$, the isomorphism
\[
H^*_{L^+_{\Spd C}G}(\Gr_G,\Qlb)\cong \cal{O}_{\bb{T}(\wh{\fk{t}}^*/W)}
\]
in Proposition \ref{prop:H*Hck} is Weil-equivariant.
\end{thm}
\begin{proof}
Similarly to the definition of $\alpha^w$, put $\lambda^w=w^{-1}\circ \lambda \circ w$ for $\lambda\in \bb{X}_*(T_{\ol{F}})$.
The cohomology
\[
H^*_{L^+_{\Spd C}T}\left(\coprod_{\lambda\in \bb{X}_*(T_{\ol{F}})}[\lambda],\Qlb\right)\cong \prod_{\lambda\in \bb{X}_*(T_{\ol{F}})} H^*_{L^+_{\Spd C}T}([\lambda],\Qlb)
\]
has an action of $W_{F}$, and the action of $w\in W_{F}$ splits into
\[
H^*_{L^+_{\Spd C}T}([\lambda],\Qlb)\to H^*_{L^+_{\Spd C}T}([\lambda^{w^{-1}}],\Qlb).
\]
By Corollary \ref{cor:H*T}, this induces a homomorphism 
\[
\Qlb[\wh{\fk{t}}]\to \Qlb[\wh{\fk{t}}],
\]
and by Corollary \ref{cor:H*TWequiv}, this is induced by the composition of the usual action of $w$ on $\wh{\fk{t}}$ and the multiplication of elements of $\wh{\fk{t}}$ by $q^{-\deg w}$.
Putting these together, we have an explicit description of a $W_{F}$-action on $\prod_{\lambda\in \bb{X}_*(T_{\ol{F}})} \Qlb[\wh{\fk{t}}]$ such that the isomorphism (\ref{eqn:prodisom}) is $W_{F}$-equivariant.
We denote this action by $a_{\prod \Qlb[\wh{\fk{t}}]}$.

By a simple calculation, we can show that the homomorphism (\ref{eqn:OTtWemb}) is $W_{F}$-equivariant with respect to $a_{\cal{O}_{\bb{T}(\wh{\fk{t}}^*/W)}}$ and $a_{\prod \Qlb[\wh{\fk{t}}]}$.
Then the theorem follows by Proposition \ref{prop:H*Hckemb}.
\end{proof}
By a similar argument, we have the following:
\begin{thm}\label{thm:H*GWeil}
When we endow $\cal{O}_{\wh{\fk{t}}^*/W}$ with the $W_{F}$-action defined by $a_{\cal{O}_{\wh{\fk{t}}^*/W}}$, the isomorphism
\[
H^*_{L^+_{\Spd C}G}(\Spd C,\Qlb)\cong \cal{O}_{\wh{\fk{t}}^*/W}
\]
in Corollary \ref{cor:H*G} is Weil-equivariant.
\end{thm}
\begin{proof}
This follows from Corollary \ref{cor:H*TWequiv} and Lemma \ref{lemm:H*GtoH*T}, since the canonical map $\cal{O}_{\wh{\fk{t}}^*/W}\to \cal{O}_{\wh{\fk{t}}^*}=\rm{Sym}(\wh{\fk{t}})$ is $W_{F}$-equivariant with respect to the $W_{F}$-action $a_{\cal{O}_{\wh{\fk{t}}^*/W}}$ on $\cal{O}_{\wh{\fk{t}}^*/W}$ and the Weil action on $\rm{Sym}(\wh{\fk{t}})$ as in Corollary \ref{cor:H*TWequiv}.
\end{proof}
\subsection{Cochain algebra of Hecke stack}
We equip $\cal{O}_{\bb{T}(\wh{\fk{t}}^*/W)}$ with a grading such that the coordinates of $\wh{\fk{t}}^*$ in the base direction have degree 1 and those in the fiber direction have degree 0.
Applying the shearing operation (i.e. the shift of grading defined in \cite[\S A.2]{AG} or \cite[\S 6]{BSV}), we get an object
\[
\cal{O}_{\bb{T}(\wh{\fk{t}}^*/W)}^{\shear}.
\]
This is a formal dg-algebra, such that the coordinates of $\wh{\fk{t}}^*$ in the base directions have cohomological degree 2 and those in the fiber direction have cohomological degree 0.
Equipping $\cal{O}_{\bb{T}(\wh{\fk{t}}^*/W)}$ with a $W_{F}$-action $a_{\cal{O}_{\bb{T}(\wh{\fk{t}}^*/W)}}$, we regard $\cal{O}_{\bb{T}(\wh{\fk{t}}^*/W)}^{\shear}$ as an object of $\DDRep(W_{F},\Qlb)$.
Similarly, we define
$\cal{O}_{\wh{\fk{t}}^*/W}^{\shear}\in \DDRep(W_{F},\Qlb)$ such that the coordinates of $\wh{\fk{t}}^*$ have cohomological degree 2.

In this section, we show that the dg-algebra $R\Gamma(\Hck_{G,\Spd C},\Qlb)$ is canonically isomorphic to the formal dg-algebra $\cal{O}_{\bb{T}(\wh{\fk{t}}^*/W)}^{\shear}$ when $G$ is semisimple and simply-connected.
First, we show that the $W_{F}$-action on $R\Gamma(\Hck_{G,\Spd C},\Qlb)$ factors through $W_F/I_E$.
Put
\[
\DDRep(W_{F}/I_E,\Qlb):=(\Vect_{\Qlb})^{W_F/I_E}.
\]
By the definition of $(-)^{W_F/I_E}$, 
the diagram
\[
\xymatrix{
W_F/I_E&W_F\ar[l]\\
1\ar[u]&I_E\ar[u]\ar[l]
}
\]
of groups induces a cartesian diagram
\[
\xymatrix{
\DDRep(W_{F}/I_E,\Qlb)\ar[r]\ar[d]&\DDRep(W_{F},\Qlb)\ar[d]\\
\Vect_{\Qlb}\ar[r]&\DDRep(I_E,\Qlb)
}
\]
of $\infty$-categories.
\begin{lemm}\label{lemm:RHckItriv}
The $I_E$-action on $R\Gamma(\Hck_{G,\Spd C},\Qlb)\in \rm{CAlg}(\DDRep(W_{F},\Qlb))$ is trivial.
That is, there is an essentially unique lift 
\[
R_{\Hck}'\in \rm{CAlg}(\DDRep(W_{F}/I_E,\Qlb))
\]
of $R\Gamma(\Hck_{G,\Spd C},\Qlb)\in \rm{CAlg}(\DDRep(W_{F},\Qlb))$.
\end{lemm}
\begin{proof}
Since $G_{\breve{E}}$ is split and we only consider an $I_E$-action, we can assume that $G=G^{\spl}_F$.
We abusively write $G$ for $G^{\spl}$.
Since $\Spd \cal{O}_C\to \Spd \cal{O}_{\breve{E}}$ factors through the stack $[(\Spd \cal{O}_C)/I_E]$, the $I_E$-action on $\Hck_{G,\Spd C}$ extends to a $I_E$-action $\Hck_{G,\Spd \cal{O}_C}$ and restricts to a $I_E$-action $\Hck_{G,\Spd k}$, which is trivial.
Therefore, by Proposition \ref{prop:DULAequiv}, we have
\[
R\Gamma(\Hck_{G,\Spd C},\Qlb)\simeq R\Gamma(\Hck_{G,\Spd k},\Qlb)\in \rm{CAlg}(\DDRep(I_E))
\]
and the $I_E$-action on $R\Gamma(\Hck_{G,\Spd k},\Qlb)$ is trivial.
\end{proof}
From this, we can show that $R\Gamma(\Hck_{G,\Spd C},\Qlb)\in \rm{CAlg}(\DDRep(W_{F},\Qlb))$ is isomorphic to $\cal{O}_{\bb{T}(\wh{\fk{t}}^*/W)}^{\shear}$ by the following purity argument:
\begin{prop}\label{prop:RHckequiv}
Assume that $G$ is semisimple and simply-connected.
There is a canonical equivalence
\[
R\Gamma(\Hck_{G,\Spd C},\Qlb)\simeq \cal{O}_{\bb{T}(\wh{\fk{t}}^*/W)}^{\shear}
\]
in $\rm{CAlg}(\DDRep(W_{F},\Qlb))$.
\end{prop}
\begin{proof}
It suffices to establish an equivalence $R_{\Hck}'\simeq \cal{O}_{\bb{T}(\wh{\fk{t}}^*/W)}^{\shear}$ in $\rm{CAlg}(\DDRep(W_F/I_E,\Qlb))$.
By Theorem \ref{thm:H*HckWeil}, there is a canonical isomorphism between $\bigoplus_i H^i(R_{\Hck}')$ and $\cal{O}_{\bb{T}(\wh{\fk{t}}^*/W)}^{\shear}$.
Put $H^i=H^i(\cal{O}_{\bb{T}(\wh{\fk{t}}^*/W)}^{\shear})$.
Note that $H^i=0$ for odd $i$ and that $H^i$ is pure of Frobenius weight $i$.
If an object $V\in \Rep^{\rm{f.d.}}(W_{F}/I_E,\Qlb)$ is pure of weight $w\neq 0$, then the group cohomologies vanish:
\begin{align}\label{eqn:gcohvanish}
H^j(W_{F}/I_E,V)=0
\end{align}
for any $j$.
From this, we have
\begin{align*}
\rm{Ext}^j_{\DDRep(W_{F}/I_E,\Qlb)}\left(H^i[-i],\bigoplus_{i'<i}H^{i'}[-i']\right)=0
\end{align*}
for any $i$ and $j$.
Thus, it follows by induction that the space of isomorphism between $\tau_{\leq i} R'_{\Hck}$ and $\bigoplus_{i'\leq i}H^{i'}[-i']$ in $\DDRep^{\perf}(W_{F}/I_E))$ is non-empty:
\begin{align}\label{eqn:R'Hcknoncan}
\rm{Isom}_{\DDRep(W_{F}/I_E,\Qlb)}\left(\tau_{\leq i} R'_{\Hck}, \bigoplus_{i'\leq i}H^{i'}[-i']\right)\neq \varnothing,
\end{align}
where $\tau_{\leq i}$ and $\tau_{\geq i}$ are the truncation functors with respect to the standard $t$-structure.
Since $\wh{\fk{t}}^*/W$ is isomorphic to an affine space, there is an isomorphism of the form
\[
\cal{O}_{\bb{T}(\wh{\fk{t}}^*/W)}^{\shear}\cong \Sym\left(\bigoplus_{i=1}^r V_i(-d_i)[-2d_i]\right)
\]
for some pure $V_i\in\Rep^{\rm{f.d.}}(W_{F}/I_E)$ of weight 0 and $0<d_1<\cdots <d_r\in \bb{Z}$.
Let $f$ be the composition
\begin{align*}
V_i(-d_i)[-2d_i]\subset  H^{2d_i}[-2d_i]\cong H^{2d_i}(R'_{\Hck})[-2d_i].
\end{align*}
This defines a point in $Map_{\DDRep^{\perf}(W_{F}/I_E))}(V_{i}(-d_i)[-2d_i], \tau_{\geq 2d_i}\tau_{\leq 2d_i}R_{\Hck}').$
Moreover, the canonical map
\begin{align*}
&Map_{\DDRep(W_{F}/I_E,\Qlb)}(V_{i}(-d_i)[-2d_i], \tau_{\leq 2d_i}R_{\Hck}')\\
&\to Map_{\DDRep(W_{F}/I_E,\Qlb)}(V_{i}(-d_i)[-2d_i], \tau_{\geq 2d_i}\tau_{\leq 2d_i}R_{\Hck}')    
\end{align*}
is a homotopy equivalence since the fibers are equivalent to
\[
Map_{\DDRep(W_{F}/I_E,\Qlb)}\left(V_{i}(-d_i)[-2d_i], \bigoplus_{i'<2d_i}H^{i'}[-i']\right),
\]
which is contractible by (\ref{eqn:gcohvanish}).
Therefore, $f$ has a unique lift $\wt{f}$ in 
\[
Map_{\DDRep(W_{F}/I_E,\Qlb)}(V_{i}(-d_i)[-2d_i], \tau_{\leq d_i}R_{\Hck}')
\]
up to contractible ambiguity.
Thus we obtain a map
\[
V_{i}(-d_i)[-2d_i]\overset{\wt{f}}{\to}  \tau_{\leq d_i}R_{\Hck}' \to R_{\Hck}'
\]
in $\DDRep(W_{F}/I_E,\Qlb)$.
This defines a map
\[
\cal{O}_{\bb{T}(\wh{\fk{t}}^*/W)}^{\shear}=\Sym\left(\bigoplus_{i=1}^r V_i(-d_i)[-2d_i]\right)\to R_{\Hck}'
\]
in $\DDRep(W_{F}/I_E,\Qlb)$.
Note that, by \cite[Construction 3.1.3.9, Example 3.1.3.14]{LurHA}, a symmetric algebra $\Sym(V)$ in the derived sense is of the form $\bigoplus_{n}\Sym^n(V)$ where $\Sym^n(V)$ is the homotopy colimit of a diagram indexed by $N(B\Sigma_n)$ that takes the object to $V^{\otimes n}$ and acts by permuting the factors.
(Here, $N(B\Sigma_n)$ is the nerve of the groupoid $B\Sigma_n$, which has a single object and whose automorphism group is the symmetric group $\Sigma_n$.)
Since the group homology of a finite group vanishes in nonzero degrees over $\Qlb$, this coincides with the classical symmetric algebra regarded as a formal dg-algebra.
Therefore, the above homomorphism induces isomorphisms on their cohomologies by construction, and hence it is an equivalence.
\end{proof}
By the same argument, we can show the following: 
\begin{prop}\label{prop:RBL+Gequiv}
There is a canonical equivalence
\[
R\Gamma([\Spd C/L^+_{\Spd C}G],\Qlb)\simeq \cal{O}_{\wh{\fk{t}}^*/W}^{\shear}
\]
in $\rm{CAlg}(\DDRep(W_{F},\Qlb))$.
\end{prop}
\section{Derived Satake equivalence over $\Spd C$}\label{sec:derSatoverSpdC}
In this section, we discuss the derived Satake equivalence (\ref{eqn:mainSpdC}) over $\Spd C$.

We have the following three methods to prove this:
\begin{enumerate}
\item \label{item:Banmethod}
Just by composing the equal characteristic derived Satake equivalence and the equivalence in Proposition \ref{prop:DULAequiv}. This is the method of \cite{Ban}. 

This approach is easy, but the relation to $W_{F}$-actions is not clear.
\item Assume that $G$ is semisimple and simply-connected (since the general case follows from this), and construct a fully faithful functor
\[
\wt{S}_{qc}\colon \cal{O}_{\wh{\fk{g}}^*}^{\shear}\lmod^{fr}(\cal{D}(\Rep(\wh{G},\Qlb)))\to \cal{D}^{\ULA}(\Hck_{G,\Spd C},\Qlb)
\]
with some conditions, by using a similar functor in equal characteristic constructed in \cite[Theorem 4]{BF} and Proposition \ref{prop:DULAequiv}.
Then we can show the desired equivalence by the same argument as \cite[\S 6.6]{BF}.

This approach can make the relation to the $W_{F}$-action transparent, but the following method (\ref{item:3rdmethod}) is simpler.
\label{item:BFmethod}
\item 
Use the fact that both sides of (\ref{eqn:mainSpdC}) are generated by semisimple objects under extensions, and define the desired equivalence (\ref{eqn:mainSpdC}) directly, using an argument parallel to the definition of $\wt{S}_{qc}$.
More precisely, it is as follows:
In the construction of $\wt{S}_{qc}$ in the method (\ref{item:BFmethod}), we used full subcategories consisting of semisimple objects, such as \[
\cal{O}_{\wh{\fk{g}}^*}^{\shear}\lmod^{fr}(\cal{D}(\Rep(\wh{G},\Qlb)))\subset \cal{O}_{\wh{\fk{g}}^*}^{\shear}\lmod^{\perf}(\cal{D}(\Rep(\wh{G},\Qlb)))
\]
or $\cal{IC}\subset \cal{D}^{\ULA}(\Hck_{G,\Spd C},\Qlb)$ (defined later).
However, since these subcategories generate the whole categories under extensions, the propositions appearing in the method (\ref{item:BFmethod}) extend to the entire categories.
This allows us to construct the equivalence (\ref{eqn:mainSpdC}) directly, analogously to the definition of $\wt{S}_{qc}$.

This method can be applied to the equal characteristic case, and replacing the arguments in \cite[\S 6.5, \S 6.6]{BF} of Bezrukavnikov--Finkelberg with a simpler one.
\label{item:3rdmethod}
\end{enumerate}
We only use the method (\ref{item:3rdmethod}).
By construction, we can easily show that the equivalences obtained by (\ref{item:Banmethod}) and (\ref{item:3rdmethod}) agree.
The procedure of the method (\ref{item:3rdmethod}) is as follows:
In \S \ref{sssc:derSatSpdCsssc:Sqc}, assuming semisimplicity and simple connectedness of $G$, we will first define a functor
\[
R\Gamma\colon \cal{D}^{\ULA}(\Hck_{G,\Spd C},\Qlb)\to \cal{O}_{\bb{T}(\wh{\fk{t}}^*/W)}^{\shear}\lmod
\]
using the cohomology of the Hecke stack.
We will also define a purely algebraic functor, called Kostant functor
\[
\kappa\colon \Coh([\wh{\fk{g}}^*/\wh{G}])^{\shear}\to \cal{O}_{\bb{T}(\wh{\fk{t}}^*/W)}^{\shear}\lmod.
\]
Then we will prove that $R\Gamma$ and $\kappa$ are fully faithful and have the same essential images.
This implies the desired equivalence $\cal{S}^{der}_{\Spd C}$.
In \S \ref{sssc:monoidalonSqc}, we will show that $R\Gamma$ and $\kappa$ have a monoidal structure, and hence so does $\cal{S}^{der}_{\Spd C}$.
In \S \ref{sssc:compatimonoidalnonderived}, we will treat compatibility with the non-derived Satake equivalence.
In \S \ref{ssc:toruscase}, we will prove the desired equivalence for tori, and in \S \ref{ssc:generalcase} for general $G$, using a central isogeny to $G$ from a product of a torus and a semisimple and simply-connected group.
\subsection{Non-derived Satake functor}
The (non-derived) geometric Satake equivalence (\ref{eqn:nonderSatSpdCLambda}) over $\Spd C$ defines a monoidal functor
\[
\Rep^{\rm{f.d.}}(\wh{G},\Qlb)\to \cal{D}^{\ULA}(\Hck_{G,\Spd C},\Qlb).
\]
By \cite[Corollary 7.4.12]{BCKW}, this canonically extends to a monoidal functor
\[
\cal{S}_{\Spd C}:=\cal{S}_{G,\Spd C}\colon \cal{D}^b(\Rep^{\rm{f.d.}}(\wh{G},\Qlb))\to \cal{D}^{\ULA}(\Hck_{G,\Spd C},\Qlb)
\]
of monoidal stable $\infty$-categories.
Let 
\[
\cal{IC}\subset \cal{D}^{\ULA}(\Hck_{G,\Spd C},\Qlb)
\]
denote the essential image of $\cal{S}_{\Spd C}$.
This is the full subcategory consisting of objects of the form $\ds\bigoplus_{i=1}^n A_i[d_i]$ with $n\in \bb{Z}_{\geq 0}$, $d_i\in \bb{Z}$, and $A_i\in \Sat(\Hck_{G,\Spd C},\Qlb)$.
\begin{lemm}\label{lemm:genbyIC}
The category $\cal{D}^{\ULA}(\Hck_{G,\Spd C},\Qlb)$ is generated by $\cal{IC}$ under extensions.
\end{lemm}
\begin{proof}
It suffices to show that $\cal{D}^{\ULA}(\Hck_{G,\Spd C},\Qlb)$ is generated by the objects in $\Sat(\Hck_{G,\Spd C},\Qlb)$ under extensions and shifts.
By Lemma \ref{lemm:SatPervequiv}, it suffices to show that $\cal{D}^{\ULA}(\Hck^{\Witt}_{G,\Spec k},\Qlb)$ is generated by perverse sheaves under extensions and shifts.
This follows from Lemma \ref{lemm:DULAschemeboundedperverse}.
\end{proof}
\subsection{Semisimple and simply-connected case}\label{ssc:derSatSpdCsssc}
\subsubsection{Construction of derived Satake equivalence} \label{sssc:derSatSpdCsssc:Sqc}
For any $\cal{F}\in \cal{D}^{\ULA}(\Hck_{G,\Spd C},\Qlb)$, the cohomology $R\Gamma(\Hck_{G,\Spd C},\cal{F})$ is a dg-module over the dg-algebra $\cal{O}_{\bb{T}(\wh{\fk{t}}^*/W)}^{\shear}$ by Proposition \ref{prop:RHckequiv}.
In fact, the same proposition lifts $R\Gamma$ to a functor
\begin{align}\label{eqn:defofH*}
R\Gamma \colon \cal{D}^{\ULA}(\Hck_{G,\Spd C},\Qlb)\to \cal{O}_{\bb{T}(\wh{\fk{t}}^*/W)}^{\shear}\lmod,
\end{align}
where the target is obtained by applying the shearing operation to the derived category of coherent sheaves on $\bb{T}(\wh{\fk{t}}^*/W)$.
\begin{lemm}\label{lemm:RGammafullyfaithful}
The functor $R\Gamma$ is fully faithful.
\end{lemm}
\begin{proof}
By Lemma \ref{lemm:genbyIC}, it suffices to show full faithfulness on $\cal{IC}$.
A similar functor 
\[
D^{\ULA}(\Hck^{\eq}_{G,\Spec k},\Qlb)\to \cal{O}_{\bb{T}(\wh{\fk{t}}^*/W)}^{\shear}\lmod
\]
is defined in the paragraph preceding Theorem 4 in \cite{BF}.
By construction, the square
\begin{align}\label{eqn:RGammareducetoeq}
\vcenter{
\xymatrix{
D^{\ULA}(\Hck_{G,\Spd C},\Qlb)\ar[r]^-{R\Gamma}\ar[d]^{\rotatebox{-90}{$\sim$}}_{(\ref{eqn:DULAequiv})}& \cal{O}_{\bb{T}(\wh{\fk{t}}^*/W)}^{\shear}\lmod\ar@{=}[d]\\
D^{\ULA}(\Hck^{\eq}_{G,\Spec k},\Qlb)\ar[r]& \cal{O}_{\bb{T}(\wh{\fk{t}}^*/W)}^{\shear}\lmod
}
}
\end{align}
naturally commutes.
Therefore, the claim reduces to the case of equal characteristic, which is proved in \cite[Lemma 13]{BF}.
\end{proof}

Let $\Coh([\wh{\fk{g}}^*/\wh{G}])^{fr}$ denote the essential image of the $*$-pullback functor 
\[
fr\colon \cal{D}^b(\Rep^{\rm{f.d.}}(\wh{G}))=\Coh([\Spec \Qlb/\wh{G}])\to \Coh([\wh{\fk{g}}^*/\wh{G}]). 
\]
Since $\Coh([\wh{\fk{g}}^*/\wh{G}])\simeq \Sym(\wh{\fk{g}})\lmod^{\perf}(\cal{D}(\Rep(\wh{G})))$, one can easily show the following lemma:
\begin{lemm}\label{lemm:genbyfr}
The category $\Coh([\wh{\fk{g}}^*/\wh{G}])$ is generated by $\Coh([\wh{\fk{g}}^*/\wh{G}])^{fr}$ under extensions.
\end{lemm}
In \cite[\S 2.6]{BF}, the authors define the Kostant functor
\[
\kappa\colon \Coh^{\bb{G}_m}([\wh{\fk{g}}^*/\wh{G}])
^{\heartsuit}
\to \Coh^{\bb{G}_m}(\bb{T}(\wh{\fk{t}}^*/W))
^{\heartsuit}.
\]
We will recall its definition in the paragraph after Lemma \ref{lemm:monoidalonkappa}.
The definition extends to 
\[
\kappa\colon \QCoh^{\bb{G}_m}([\wh{\fk{g}}^*/\wh{G}])
^{\heartsuit}
\to \QCoh^{\bb{G}_m}(\bb{T}(\wh{\fk{t}}^*/W))
^{\heartsuit}.
\]
Since there are natural equivalences 
\begin{align*}
\cal{D}(\QCoh^{\bb{G}_m}([\wh{\fk{g}}^*/\wh{G}])
^{\heartsuit})&\simeq \QCoh^{\bb{G}_m}([\wh{\fk{g}}^*/\wh{G}])\\
\cal{D}(\QCoh^{\bb{G}_m}(\bb{T}(\wh{\fk{t}}^*/W))^{\heartsuit})&\simeq \QCoh^{\bb{G}_m}(\bb{T}(\wh{\fk{t}}^*/W))
\end{align*}
by \cite[Proposition I.3.2.4.3, Remark I.3.2.4.4]{GR}, it induces a functor 
\[
\kappa\colon \QCoh^{\bb{G}_m}([\wh{\fk{g}}^*/\wh{G}])\to \QCoh^{\bb{G}_m}(\bb{T}(\wh{\fk{t}}^*/W)).
\]
Using the equivalence of equivariantization and de-equivariantization (e.g.~(A.2) and (A.3) in \cite{AG}) for the group $\bb{G}_m$, we obtain
\[
\kappa\colon \QCoh([\wh{\fk{g}}^*/\wh{G}])\to \QCoh(\bb{T}(\wh{\fk{t}}^*/W))
\]
and by applying the shearing operator defined in \cite[\S A.2]{AG}, we also obtain
\[
\kappa^{\shear}\colon \Coh([\wh{\fk{g}}^*/\wh{G}])^{\shear}\to \cal{O}_{\bb{T}(\wh{\fk{t}}^*/W)}^{\shear}\lmod. 
\]
In fact, by unraveling the meaning of the ``full embedding'' asserted in \cite[Lemma 4]{BF}, its restriction
\[
\kappa^{\shear}|_{\Coh([\wh{\fk{g}}^*/\wh{G}])^{\shear,fr}}\colon \Coh([\wh{\fk{g}}^*/\wh{G}])^{\shear,fr}\to \cal{O}_{\bb{T}(\wh{\fk{t}}^*/W)}^{\shear}\lmod
\]
is fully faithful.
Since $\Coh([\wh{\fk{g}}^*/\wh{G}])^{\shear}$ is generated by $\Coh([\wh{\fk{g}}^*/\wh{G}])^{\shear,fr}$ under extensions, we get the following lemma:
\begin{lemm}\label{lemm:kappafullyfaithful}
\[
\kappa^{\shear}\colon \Coh([\wh{\fk{g}}^*/\wh{G}])^{\shear}\to \cal{O}_{\bb{T}(\wh{\fk{t}}^*/W)}^{\shear}\lmod
\]
is fully faithful.
\end{lemm}
We also have the following lemma:
\begin{lemm}\label{lemm:imageRGammakappaagree}
The essential image of $R\Gamma \colon \cal{D}^{\ULA}(\Hck_{G,\Spd C},\Qlb)\to \cal{O}_{\bb{T}(\wh{\fk{t}}^*/W)}^{\shear}\lmod$ and that of $\kappa^{\shear}\colon \Coh([\wh{\fk{g}}^*/\wh{G}])^{\shear}\to \cal{O}_{\bb{T}(\wh{\fk{t}}^*/W)}^{\shear}\lmod$ agree.
\end{lemm}
\begin{proof}
By the diagram (\ref{eqn:RGammareducetoeq}), it reduces to the equal characteristic case.
By \cite[Theorem 4]{BF}, the essential images of $R\Gamma|_{\cal{IC}}$ and $\kappa^{\shear}|_{\Coh([\wh{\fk{g}}^*/\wh{G}])^{\shear,fr}}$ agree.
Since both functors preserve extensions, the claim follows from Lemma \ref{lemm:genbyIC} and Lemma \ref{lemm:genbyfr}.
\end{proof}
Then we have the following proposition:
\begin{prop}\label{prop:SqcSpdC}
There exists an essentially unique equivalence
\[
\cal{S}^{der}_{\Spd C}:=\cal{S}^{der}_{G,\Spd C}\colon \Coh([\wh{\fk{g}}^*/\wh{G}])^{\shear}\xrightarrow{\sim} \cal{D}^{\ULA}(\Hck_{G,\Spd C},\Qlb), 
\]
such that there are natural isomorphisms
\begin{align}
S_{\Spd C}\simeq \cal{S}^{der}_{\Spd C}\circ fr|_{\Rep^{\rm{f.d.}}(\wh{G})}\ &\colon\ \Rep^{\rm{f.d.}}(\wh{G})\to \cal{D}^{\ULA}(\Hck_{G,\Spd C},\Qlb)\label{eqn:SSpdCSqcFr}, \\
\kappa^{\shear}\simeq R\Gamma\circ \cal{S}^{der}_{\Spd C}\ &\colon\  \Coh([\wh{\fk{g}}^*/\wh{G}])^{\shear}\to \cal{O}_{\bb{T}(\wh{\fk{t}}^*/W)}^{\shear}\lmod\label{eqn:kappaH*Sqc}. 
\end{align}
\end{prop}
\begin{proof}
Reducing to the equal characteristic case and using \cite[Theorem 4]{BF}, we can show that there exists an essentially unique fully faithful functor
\[
\wt{S}_{qc}\colon \Coh([\wh{\fk{g}}^*/\wh{G}])^{\shear,fr}\to \cal{D}^{\ULA}(\Hck_{G,\Spd C},\Qlb), 
\]
such that $\cal{S}_{\Spd C}\simeq \wt{S}_{qc}\circ fr$ and $\kappa^{\shear}|_{\Coh([\wh{\fk{g}}^*/\wh{G}])^{\shear,fr}}\simeq R\Gamma\circ \wt{S}_{qc}$ hold.
Here, we use Lemma \ref{lemm:SatPervequiv} to ensure that $\cal{S}_{\Spd C}$ corresponds to the functor induced by the non-derived geometric Satake equivalence when reducing to the equal characteristic case.

By Lemma \ref{lemm:RGammafullyfaithful}, Lemma \ref{lemm:kappafullyfaithful} and Lemma \ref{lemm:imageRGammakappaagree}, this essentially uniquely extends to a desired equivalence $\cal{S}^{der}_{\Spd C}$.
\end{proof}
\begin{rmk}\label{rmk:SqcSpdk}
The same proof applies when replacing $\Spd C$ by $\Spd k$.
We denote the resulting functor by $\cal{S}^{der}_{\Spd k}$
\end{rmk}
We want to show that this equivalence $\cal{S}^{der}_{\Spd C}$ is monoidal.
\subsubsection{Monoidal structure of derived Satake}\label{sssc:monoidalonSqc}
In this subsection, we endow monoidal structures with $R\Gamma$, $\kappa$ and $\cal{S}^{der}_{\Spd C}$.
The monoidal structure on $R\Gamma$ and $\kappa$ is defined similarly to \cite[\S 3.5]{BF} and \cite[\S 4.6]{BF}.
However, they only define them on the full subcategory of perverse sheaves, whereas we define them on the whole category.

The $\infty$-category 
\[
\Coh([\wh{\fk{g}}^*/\wh{G}])^{\shear}=\cal{O}_{\wh{\fk{g}}}^{\shear}\lmod^{\perf}(\cal{D}(\Rep(\wh{G})))
\]
is monoidal via the tensor product $-\otimes_{\cal{O}_{\wh{\fk{g}}^*}^{\shear}}-$.
Consider the diagram
\[
\bb{T}(\wh{\fk{t}}^*/W)\times \bb{T}(\wh{\fk{t}}^*/W)
\xleftarrow{p_{\bb{T}}} 
\bb{T}(\wh{\fk{t}}^*/W)\times_{\wh{\fk{t}}^*/W}\bb{T}(\wh{\fk{t}}^*/W)
\xrightarrow{m_{\bb{T}}} \bb{T}(\wh{\fk{t}}^*/W), 
\]
where $p_{\bb{T}}$ is the natural closed immersion and $m_{\bb{T}}$ is the fiberwise addition defined in (\ref{eqn:defofmT}).
The $\infty$-category $\QCoh(\bb{T}(\wh{\fk{t}}^*/W))$ is monoidal via the product
\[
-\star -:=Rm_{\bb{T}*}Lp_{\bb{T}}^*(-\boxtimes -),
\]
and hence so is $\cal{O}_{\bb{T}(\wh{\fk{t}}^*/W)}^{\shear}\lmod$.
Similarly, the categories $\QCoh^{\bb{G}_m}(\wh{\fk{g}}^*/\wh{G})^{\heartsuit}$ and $\QCoh^{\bb{G}_m}(\bb{T}(\wh{\fk{t}}^*/W))^{\heartsuit}$ are monoidal in a non-derived sense.

Similarly to the construction of (\ref{eqn:RGammaHckconv}), for $\cal{F}_1,\cal{F}_2\in \cal{D}^{\ULA}(\Hck_{G,\Spd C})$, there is a natural homomorphism
\begin{multline}
R\Gamma(\Hck_{G,\Spd C},\cal{F}_1)\otimes_{R\Gamma([\Spd C/L^+_{\Spd C}G],\Qlb)}R\Gamma(\Hck_{G,\Spd C},\cal{F}_2)\\
\to R\Gamma(\Hck_{G,\Spd C},\cal{F}_1\star \cal{F}_2) \label{eqn:monoidalonRGamma}
\end{multline}
of dg-modules, and by the argument in \cite[\S 3.5]{BF}, it is an equivalence for $\cal{F}_1,\cal{F}_2\in \cal{IC}$, and thus for any $\cal{F}_1,\cal{F}_2$ by Lemma \ref{lemm:genbyIC}.
This, together with Lemma \ref{lemm:H*GtoH*Hck}, Proposition \ref{prop:RHckequiv}, Proposition \ref{prop:RBL+Gequiv} and Lemma \ref{lemm:H*conv}, implies the following lemma:
\begin{lemm}\label{lemm:monoidalonRGamma}
There is a monoidal structure on the functor
\[
R\Gamma\colon \cal{D}^{\ULA}(\Hck_{G,\Spd C})\to \cal{O}_{\bb{T}(\wh{\fk{t}}^*/W)}^{\shear}\lmod.
\]
\end{lemm}
On the other hand, we can show the monoidality of $\kappa$:
\begin{lemm}\label{lemm:monoidalonkappa}
There is a monoidal structure on
\[
\kappa\colon (\QCoh^{\bb{G}_m}(\wh{\fk{g}}^*/\wh{G})^{\heartsuit},-\otimes_{\cal{O}_{\wh{\fk{g}}^*}}- )\to (\QCoh^{\bb{G}_m}(\bb{T}(\wh{\fk{t}}^*/W))^{\heartsuit},-\star- ).
\]
\end{lemm}
Let us recall the definition of $\kappa$ in \cite[\S 2.6]{BF}.
Let $(\wh{\fk{g}}^*)^{\rm{reg}}\subset \wh{\fk{g}}^*$ be the subset of regular elements.
Let $\fk{z}\subset \wh{\fk{g}}\otimes \cal{O}_{(\wh{\fk{g}}^*)^{\rm{reg}}}$ be a universal centralizer sheaf on $(\wh{\fk{g}}^*)^{\rm{reg}}$, whose fiber at $\xi\in (\wh{\fk{g}}^*)^{\rm{reg}}$ is the centralizer $\fk{z}(\xi)\subset \wh{\fk{g}}$ of $\xi$.
For a complex $\cal{F}$ of $\wh{G}$-equivariant coherent sheaves on $\wh{\fk{g}}^*$, the restriction $\cal{F}|_{(\wh{\fk{g}}^*)^{\rm{reg}}}$ carries an action of the commutative Lie algebra $\fk{z}$, which defines a quasi-coherent complex $\wt{\cal{F}}$ on the total space $\rm{Tot}(\fk{z})$ of $\fk{z}$.
By restricting $\cal{F}$ to a ``Kostant slice'' $\Sigma\subset (\wh{\fk{g}}^*)^{\rm{reg}}$, we get a coherent sheaf on $X_{\Sigma}:=\rm{Tot}(\fk{z})\times_{(\wh{\fk{g}}^*)^{\rm{reg}}}\Sigma$, which is isomorphic to $\bb{T}(\wh{\fk{t}}^*/W)$ via a composition of the form
\[
X_{\Sigma}\hookrightarrow \rm{Tot}(\fk{z})\twoheadrightarrow \bb{T}(\wh{\fk{t}}^*/W).
\]
This defines $\kappa$.
Note that $\kappa$ does not depend on the choice of Kostant slices since all these slices are $\wh{G}$-conjugate and $\cal{F}$ is $\wh{G}$-equivariant.
\begin{proof}[Proof of Lemma \ref{lemm:monoidalonkappa}]
Consider the diagram
\begin{align}\label{eqn:convdiagTot}
\rm{Tot}(\fk{z})\times \rm{Tot}(\fk{z})\xleftarrow{p_{\rm{Tot}}} \rm{Tot}(\fk{z})\times_{(\wh{\fk{g}}^*)^{\rm{reg}}}\rm{Tot}(\fk{z})\xrightarrow{m_{\rm{Tot}}} \rm{Tot}(\fk{z}),
\end{align}
where $m_{\rm{Tot}}$ is the fiberwise addition map over $(\wh{\fk{g}}^*)^{\rm{reg}}$, and put 
\[
-\star_{\rm{Tot}} -:=(m_{\rm{Tot}})_*(p_{\rm{Tot}})^*(-\boxtimes -).
\]
For $\cal{F}_1,\cal{F}_2\in \QCoh^{\bb{G}_m}([\wh{\fk{g}}^*/\wh{G}])^{\heartsuit}$ and $\xi\in (\wh{\fk{g}}^*)^{\rm{reg}}$, the $\fk{z}(\xi)$-action on $(\cal{F}_1\otimes_{\cal{O}_{\wh{\fk{g}}^*}}\cal{F}_2)_{\xi}$ is given by $z\cdot (m\otimes n)=(z\cdot m)\otimes n+m\otimes(z\cdot n)$.
Therefore, by the construction of $\wt{\cal{F}}$, it follows that
\begin{align}\label{eqn:convtilde}
(\cal{F}_1\otimes_{\cal{O}_{\wh{\fk{g}}^*}}\cal{F}_2)^{\sim}\cong \wt{\cal{F}_1}\star_{\rm{Tot}} \wt{\cal{F}_2}.
\end{align}
Now the lemma follows from this and the definition of the monoidal structure on $\QCoh^{\bb{G}_m}(\bb{T}(\wh{\fk{t}}^*/W))^{\heartsuit}$.
\end{proof}
Since the functor
\[
\kappa^{\shear}\colon \Coh([\wh{\fk{g}}^*/\wh{G}])^{\shear}\to \cal{O}_{\bb{T}(\wh{\fk{t}}^*/W)}^{\shear}\lmod
\]
is constructed from $\kappa\colon \QCoh^{\bb{G}_m}(\wh{\fk{g}}^*/\wh{G})^{\heartsuit}\to \QCoh^{\bb{G}_m}(\bb{T}(\wh{\fk{t}}^*/W))^{\heartsuit}$ as in \S \ref{sssc:derSatSpdCsssc:Sqc}, we obtain the monoidal structure of $\kappa^{\shear}$.
This implies the following:
\begin{thm}\label{thm:mainSpdCssscbody}
There exists a monoidal structure on
\[
\cal{S}^{der}_{\Spd C}\colon \Coh([\wh{\fk{g}}^*/\wh{G}])^{\shear}\simeq \cal{D}^{\ULA}(\Hck_{G,\Spd C},\Qlb), 
\]
such that the natural isomorphism $\kappa^{\shear}\cong R\Gamma\circ \cal{S}^{der}_{\Spd C}$ in (\ref{eqn:kappaH*Sqc}) is canonically monoidal.
\end{thm}
\begin{proof}
This follows from Lemma \ref{lemm:monoidalonRGamma}, Lemma \ref{lemm:monoidalonkappa} and the full faithfulness of $R\Gamma\colon \cal{IC}\to \cal{O}_{\bb{T}(\wh{\fk{t}}^*/W)}^{\shear}\lmod$.
\end{proof}
\subsubsection{Compatibility of monoidality with the non-derived Satake}\label{sssc:compatimonoidalnonderived}
We equip $\cal{S}^{der}_{\Spd C}$ with the monoidal structure by Theorem \ref{thm:mainSpdCssscbody}.
In this section, we will show that this monoidal structure is compatible with the non-derived geometric Satake equivalence:
\begin{prop}\label{prop:monoidalSSpdCSqcFr}
The natural isomorphism
\[
S_{\Spd C}\cong \cal{S}^{der}_{\Spd C}\circ fr|_{\Rep^{\rm{f.d.}}(\wh{G})}
\]
in (\ref{eqn:SSpdCSqcFr}) is canonically monoidal.
\end{prop}
\begin{proof}
We have defined the functors
\[
\Rep^{\rm{f.d.}}(\wh{G})\xrightarrow{fr}\Coh([\wh{\fk{g}}^*/\wh{G}])^{\shear}\xrightarrow{\cal{S}^{der}_{\Spd C}}\cal{D}^{\ULA}(\Hck_{G,\Spd C})\xrightarrow{R\Gamma}\cal{O}_{\bb{T}(\wh{\fk{t}}^*/W)}^{\shear}\lmod.
\]
Let $\cal{C}_1,\cal{C}_2,\cal{C}_3$ denote the essential images of $\Rep^{\rm{f.d.}}(\wh{G})$ in $\Coh([\wh{\fk{g}}^*/\wh{G}])^{\shear}$, $\cal{D}^{\ULA}(\Hck_{G,\Spd C})$, $\cal{O}_{\bb{T}(\wh{\fk{t}}^*/W)}^{\shear}\lmod$, respectively.
Note that $\cal{C}_1\simeq \cal{C}_2\simeq \cal{C}_3$ by the full faithfulness of $\cal{S}^{der}_{\Spd C}$ and $R\Gamma$, and that $\cal{C}_2=\Sat(\Hck_{G,\Spd C})$.
In particular, these $\infty$-categories are the nerves of 1-categories.
In the diagram
{\footnotesize
\[
\xymatrix{
\Rep^{\rm{f.d.}}(\wh{G})\ar[r]\ar@{=}[d]&\cal{C}_1\ar[r]\ar[d]&\cal{C}_2\ar[r]\ar[d]&\cal{C}_3\ar[d]\\
\Rep^{\rm{f.d.}}(\wh{G})\ar[r]^-{fr}&\Coh([\wh{\fk{g}}^*/\wh{G}])^{\shear}\ar[r]^-{\cal{S}^{der}_{\Spd C}}&\cal{D}^{\ULA}(\Hck_{G,\Spd C})\ar[r]^-{R\Gamma}&\cal{O}_{\bb{T}(\wh{\fk{t}}^*/W)}^{\shear}\lmod,
}
\]
}
all the vertical functors are canonically monoidal, and three squares monoidally commute.
Thus, it suffices to show that the isomorphism 
\begin{align}\label{eqn:h(SSpdCSqcFr)}
(\Rep^{\rm{f.d.}}(\wh{G})\xrightarrow{fr}\cal{C}_1\xrightarrow{\cal{S}^{der}_{\Spd C}}\cal{C}_2)\simeq (\Rep^{\rm{f.d.}}(\wh{G})\xrightarrow{S_{\Spd C}}\cal{C}_2) 
\end{align}
 is monoidal.
Consider the monoidal functor
\[
H^*\colon \cal{C}_3\to \QCoh(\bb{T}(\wh{\fk{t}}^*/W))^{\heartsuit},\ \cal{F}\mapsto \bigoplus_{i}H^i(\cal{F})
\]
of 1-categories.
Moreover, consider the homomorphism $\cal{O}_{\wh{\fk{t}}^*/W}\to \cal{O}_{\bb{T}(\wh{\fk{t}}^*/W)}$ corresponding to the natural projection $\pi_{\bb{T}}\colon \bb{T}(\wh{\fk{t}}^*/W)\to \wh{\fk{t}}^*/W$, and $\cal{O}_{\wh{\fk{t}}^*/W}\to \Qlb$ corresponding to the inclusion $\{0\}\hookrightarrow \wh{\fk{t}}^*/W$.
Let $d$ be the composition
\[
d\colon \cal{C}_3\xrightarrow{H^*} \QCoh(\bb{T}(\wh{\fk{t}}^*/W))^{\heartsuit}\xrightarrow{\Qlb\otimes_{\cal{O}_{\wh{\fk{t}}^*/W}}-}\Vect_{\Qlb}.
\]
We can prove the following lemma:
\begin{lemm}\label{lemm:H*dRGammaH*Grmonoidal}
There is a natural monoidal isomorphism between the composition
\[
\Sat(\Hck_{G,\Spd C})=\cal{C}_2\xrightarrow{R\Gamma} \cal{C}_3\overset{d}{\longrightarrow}\Vect_{\Qlb}^{\heartsuit}
\]
and 
\[
F_{G,\Spd C}\colon \Sat(\Hck_{G,\Spd C})\to \Vect_{\Qlb}^{\heartsuit},\quad \cal{F}\mapsto \bigoplus_i H^i(\Gr_{G,\Spd C},\cal{F}).
\]
\end{lemm}
\begin{proof}
The isomorphism
\begin{align}\label{eqn:H*dRGammaH*Gr}
d\circ R\Gamma\simeq F_{G,\Spd C}
\end{align}
follows from Lemma \ref{lemm:H*GtoH*Hck}, Proposition \ref{prop:RHckequiv}, Proposition \ref{prop:RBL+Gequiv} and the isomorphism
\begin{align}\label{eqn:H*deeq}
H^*(\Spd C) \otimes_{H^*([\Spd C/L^+_{\Spd C}G])} H^*(\Hck_{G,\Spd C},\cal{F})\xrightarrow{\sim}  H^*(\Gr_{G,\Spd C},\cal{F}).
\end{align}
This isomorphism (\ref{eqn:H*deeq}) follows, for example, from the argument of \cite[\S A.1.3]{Zhueq},  reducing to the equal characteristic case via the results in \S \ref{sec:cohofHecke} and noting that $H^*(\Hck^{\eq}_{G^{\spl},\Spec k},\cal{F})$ is free over $H^*([\Spec k/L^{+,\eq}_{\Spec k}G^{\spl}])$.

It remains to show the monoidality.
Reducing to the Witt-vector case by Proposition \ref{prop:DULAequiv}, and if we endow $F_{G,\Spd C}$ with a monoidal structure as in \cite[\S 2.3]{Zhumixed}, then the isomorphism (\ref{eqn:H*dRGammaH*Gr}) is monoidal by definition.
Since Zhu's monoidal structure on $F_{G,\Spd C}$ coincides with our monoidal structure by \cite{Bantwomonoidal}, the claim follows.
\end{proof}
Moreover, we have the following lemma:
\begin{lemm}\label{lemm:FSderfrFSisomdecomp}
The isomorphism (\ref{eqn:h(SSpdCSqcFr)}) composed with $F_{G,\Spd C}$:
\begin{align}
&(\Rep^{\rm{f.d.}}(\wh{G})\xrightarrow{fr}\cal{C}_1\xrightarrow{\cal{S}^{der}_{\Spd C}}\Sat(\Hck_{G,\Spd C})\xrightarrow{F_{G,\Spd C}}\Vect^{\heartsuit})\notag\\
&\simeq (\Rep^{\rm{f.d.}}(\wh{G})\xrightarrow{S_{\Spd C}}\Sat(\Hck_{G,\Spd C})\xrightarrow{F_{G,\Spd C}}\Vect^{\heartsuit})\label{eqn:composedwithH*GrG}. 
\end{align}
decomposes as
\begin{align}
&(\Rep^{\rm{f.d.}}(\wh{G})\xrightarrow{fr}\cal{C}_1\xrightarrow{\cal{S}^{der}_{\Spd C}}\Sat(\Hck_{G,\Spd C})\xrightarrow{F_{G,\Spd C}}\Vect^{\heartsuit})\notag\\
&\simeq (\Rep^{\rm{f.d.}}(\wh{G})\xrightarrow{fr}\cal{C}_1\xrightarrow{\cal{S}^{der}_{\Spd C}}\Sat(\Hck_{G,\Spd C})\xrightarrow{R\Gamma}\cal{C}_3\xrightarrow{d}\Vect^{\heartsuit})\notag\\
&\simeq (\Rep^{\rm{f.d.}}(\wh{G})\xrightarrow{fr}\cal{C}_1\xrightarrow{\kappa}\cal{C}_3\xrightarrow{d}\Vect^{\heartsuit})\notag\\
&\simeq (\Rep^{\rm{f.d.}}(\wh{G})\xrightarrow{\rm{forget}}\Vect^{\heartsuit})\notag\\
&\simeq (\Rep^{\rm{f.d.}}(\wh{G})\xrightarrow{S_{\Spd C}}\Sat(\Hck_{G,\Spd C})\xrightarrow{F_{G,\Spd C}}\Vect^{\heartsuit}).\label{eqn:composedwithH*GrCdecomp}
\end{align}
Here the first isomorphism is the one in Lemma \ref{lemm:H*dRGammaH*Grmonoidal} and the second is (\ref{eqn:SSpdCSqcFr}).
The third is given by a purely algebraic arugument: by \cite[Lemma 9]{BF}, we have 
\[
\Qlb\otimes_{\cal{O}_{\wh{\fk{t}}^*/W}}\kappa(fr(V)) \cong \Qlb\otimes_{\cal{O}_{\wh{\fk{t}}^*/W}}\cal{O}_\Sigma\otimes V\cong V.
\]
The forth is given by the construction of the (non-derived) geometric Satake equivalence.
\end{lemm}
We assume this lemma for the moment.
The first isomorphism in (\ref{eqn:composedwithH*GrCdecomp}) is monoidal by Lemma \ref{lemm:H*dRGammaH*Grmonoidal}, the second by the definition of the monoidal structure on $\cal{S}^{der}_{\Spd C}$, the third by an easy and algebraic argument, and the fourth by the construction of (non-derived) geometric Satake equivalence.
Therefore, by Lemma \ref{lemm:FSderfrFSisomdecomp}, the isomorphism (\ref{eqn:composedwithH*GrG}) is monoidal.
The claim follows since the functor $F_{G,\Spd C}$ of 1-categories is faithful on $\Sat(\Hck_{G,\Spd C})$.
\end{proof}
We conclude this subsection by providing the proof of Lemma \ref{lemm:FSderfrFSisomdecomp}, which was deferred earlier.
\begin{proof}[Proof of Lemma \ref{lemm:FSderfrFSisomdecomp}]
By Lemma \ref{lemm:SatPervequiv}, it reduces to a similar statement for the equal characteristic case, where we use $\wt{S}_{qc}$ in \cite[Theorem 4]{BF} instead of $\cal{S}^{der}_{\Spd C}$.
Then, the claim is proved as follows:

In what follows, we use some notation from \cite{BF} without explanation.
By definition, the isomorphism $\wt{S}_{qc}\circ fr\simeq S_{\Spec k}^{\eq}$ is induced from 
\begin{align}\label{eqn:kappafrRGammaSeq}
\kappa \circ fr\simeq R\Gamma\circ S_{\Spec k}^{\eq}
\end{align}
via the full faithfulness of $R\Gamma$.
This isomorphism (\ref{eqn:kappafrRGammaSeq}) is obtained by restricting the isomorphism 
\[
\eta_V:=\eta_{V,\psi}\colon H^*
_{L^{+,\eq}_{\Spec k}G\rtimes \bb{G}_m}
(\Gr^{\eq}_{G,\Spec k},S_{\Spec k}^{\eq}(V))\cong \phi(V)
\]
in \cite[Theorem 6]{BF} to $\hbar=0$.
Here we write $\eta$ with the subscript $\psi$ since $\phi$ is defined depending on the additive character $\psi$.
Recall that $\eta_V$ is an isomorphism of modules over 
\[
H^*
_{L^{+,\eq}_{\Spec k}G\rtimes \bb{G}_m\ltimes L^{+,\eq}_{\Spec k}G}(*,\Qlb)\cong \cal{O}_{\wh{\fk{t}}^*/W\times \wh{\fk{t}}^*/W\times \bb{A}^1}.
\]
It suffices to show that, under the isomorphisms
\begin{multline}
\Qlb\otimes_{H^*
_{L^{+,\eq}_{\Spec k}G\rtimes \bb{G}_m}(*)} H^*
_{L^{+,\eq}_{\Spec k}G\rtimes \bb{G}_m}
(\Gr^{\eq}_{G,\Spec k},S_{\Spec k}^{\eq}(V))\\\cong  H^*
(\Gr^{\eq}_{G,\Spec k},S_{\Spec k}^{\eq}(V))\cong V,\label{eqn:equivcohrestrisomtoV}    
\end{multline}
\begin{align}
\Qlb \otimes_{\cal{O}_{\wh{\fk{t}}^*/W\times  \bb{A}^1}} (\phi(V))\cong \Qlb \otimes_{\cal{O}_{\wh{\fk{t}}^*/W}} (\cal{O}_{\Sigma}\otimes V)\cong V,\label{eqn:phi(V)restrisomtoV}    
\end{align}
the homomorphism $\Qlb \otimes_{\cal{O}_{\wh{\fk{t}}^*/W\times  \bb{A}^1}}\eta_{V}$ agrees with $\rm{id}_V$.
Here, the second isomorphism in (\ref{eqn:equivcohrestrisomtoV}) is given by the construction of the geometric Satake equivalence $S_{\Spec k}^{\eq}$ and the first isomorphism in (\ref{eqn:phi(V)restrisomtoV}) is given by \cite[Lemma 9]{BF}.
Let
\[
(\rm{id},\pi,\rm{id})\colon \wh{\fk{t}}^*/W\times \wh{\fk{t}}^*\times \bb{A}^1 \to \wh{\fk{t}}^*/W\times \wh{\fk{t}}^*/W\times \bb{A}^1
\]
denote the projection.
As written in \cite{BF}, there is a $\bb{X}_*(T)$-filtration on both sides of $(\rm{id},\pi,\rm{id})^*\eta_V$ whose $\lambda$-th graded piece is canonically isomorphic to
\[
\cal{O}(\Gamma_{\lambda})\otimes {}_{\lambda}V,
\]
where ${}_{\lambda}V$ is the $\lambda$-weight space of $V$ and $\Gamma_{\lambda}=\{(x,y,a)\in (\wh{\fk{t}}^*/W)^2\times \bb{A}^1\mid y=x+a\lambda\}$.
In addition, $\eta_V$ is the unique isomorphism such that $(\rm{id},\pi,\rm{id})^*\eta_V$ preserves these filtrations and induces the identity of the associated graded.
Therefore, $\Qlb \otimes_{\cal{O}_{\wh{\fk{t}}^*/W\times  \bb{A}^1}}\eta_{V}$ preserves the induced filtrations on (\ref{eqn:equivcohrestrisomtoV}) and (\ref{eqn:phi(V)restrisomtoV}), and induces the identity of the associated graded.
Since the induced filtration on (\ref{eqn:equivcohrestrisomtoV}) and (\ref{eqn:phi(V)restrisomtoV}) is 
\[
F^{\lambda}V=\bigoplus_{\lambda'\leq \lambda}{}_{\lambda'}V,
\]
it follows that
\[
{}_{\lambda}V\hookrightarrow V\xrightarrow[\simeq]{\Qlb \otimes_{\cal{O}_{\wh{\fk{t}}^*/W\times  \bb{A}^1}}\eta_{V}} V\twoheadrightarrow {}_{\lambda'}V
\]
is the identity if $\lambda=\lambda'$.

It suffices to show that it is zero if $\lambda\neq \lambda'$.
We can assume that $V$ is the irreducible representation $V_{\mu}$ of $\wh{G}$ with highest weight $\mu$.
Then, this follows from the fact that the isomorphism $\eta_V$ is compatible with the Frobenius action (cf. the last paragraph of \cite[\S 6.7]{BSV}).
Namely, it is as follows:
Assume that $V\in \Rep^{\rm{f.d.}}(\wh{G},\Qlb)$ has a Frobenius action such that $V$ is $\wh{G}\rtimes \langle Fr\rangle$-representation, where the action of $Fr$ on $\wh{G}$ is twisted via the cyclotomic action as in \S \ref{ssc:dualgp}.
Then, the corresponding perverse sheaf $S_{\Spec k}^{\eq}(V)$ has the corresponding Frobenius structure, see \cite[Theorem 5.5.12]{Zhueq} or \cite[Proposition 6.7.2]{BSV}.
For example, let $IC_{\mu}\in \Perv(\Hck^{\eq}_{G,\Spec k},\Qlb)$ denote the intersection cohomology sheaf on $\Gr^{\eq}_{G,\leq \mu}$.
If $IC_{\mu}$ has the Frobenius structure induced by the trivial Frobenius structure on the constant sheaf on $\Gr^{\eq}_{G,\mu}$, then the corresponding $Fr$-action on $V_{\mu}$ is given by the direct sum of the multiplication by $q^{\langle \rho,\lambda-\mu\rangle}$ on ${}_{\lambda}V_{\mu}$.
The Frobenius structure on $S_{\Spec k}^{\eq}(V)$ induces the Frobenius action on
\[
H^*
_{L^{+,\eq}_{\Spec k}G\rtimes \bb{G}_m}
(\Gr^{\eq}_{G,\Spec k},S_{\Spec k}^{\eq}(V)).
\]
Moreover, we define the Frobenius actions on the quantized enveloping algebra $U_{\hbar}(\wh{\fk{g}})$ by
\begin{align*}
\wh{\fk{g}}\ni x\mapsto qx,\quad \hbar\mapsto q\hbar, 
\end{align*}
and on $\cal{O}_{\wh{\fk{t}}}$ and $\cal{O}_{\bb{A}^1}$ via the multiplication map $\wh{\fk{t}}\to \wh{\fk{t}}$ and $\bb{A}^1\to \bb{A}^1$ by $q^{-1}$.
Other objects such as $\cal{O}_{\Gamma_{\lambda}}$ have induced Frobenius actions.
These induce the map
\[
\phi_{\psi}(V)\xrightarrow{\simeq} \phi_{Fr\circ \psi\circ Fr^{-1}}(V),
\]
where we write $\phi$ with the subscript $\psi$ since it depends on the choice of the character $\psi\colon U_{\hbar}(\wh{\fk{n}}_{-})\to \Qlb[\hbar]$.
By straightforward calculations, we can show that the isomorphism
\begin{align*}
gr((\rm{id},\pi,\rm{id})^*H^*
_{L^{+,\eq}_{\Spec k}G\rtimes \bb{G}_m}
(\Gr^{\eq}_{G,\Spec k},S_{\Spec k}^{\eq}(V)))\\
&\cong \bigoplus_{\lambda}\cal{O}(\Gamma_{\lambda})\otimes {}_{\lambda}V
\end{align*}
is Frobenius-equivariant, and the square
\[
\xymatrix{
gr((\rm{id},\pi,\rm{id})^*\phi_{\psi}(V))\ar[d]_{Fr}\ar[r]^-{\cong}& \bigoplus_{\lambda}\cal{O}(\Gamma_{\lambda})\otimes {}_{\lambda}V\ar[d]^{Fr}\\
gr((\rm{id},\pi,\rm{id})^*\phi_{Fr\circ \psi\circ Fr^{-1}}(V))\ar[r]^-{\cong}& \bigoplus_{\lambda}\cal{O}(\Gamma_{\lambda})\otimes {}_{\lambda}V
\\
}
\]
is commutative.
By the characterization of $\eta_V$, it follows that the square
\[
\xymatrix{
H^*
_{L^{+,\eq}_{\Spec k}G\rtimes \bb{G}_m}
(\Gr^{\eq}_{G,\Spec k},S_{\Spec k}^{\eq}(V))) \ar[r]_-{\cong}^-{\eta_{V,\psi}}\ar[d]_{Fr}&\phi_{\psi}(V)\ar[d]^{Fr}\\
H^*
_{L^{+,\eq}_{\Spec k}G\rtimes \bb{G}_m}
(\Gr^{\eq}_{G,\Spec k},S_{\Spec k}^{\eq}(V))) \ar[r]_-{\cong}^-{\eta_{V,r\circ \psi\circ Fr^{-1}}}&\phi_{Fr\circ \psi\circ Fr^{-1}}(V).
}
\]
By straightforward calculations, we can show that the isomorphisms (\ref{eqn:equivcohrestrisomtoV}) and (\ref{eqn:phi(V)restrisomtoV}) are Frobenius equivariant, and that the resulting functor
\[
V\xrightarrow{\simeq} V
\]
is Frobenius equivariant.
Since the Frobenius weights on ${}_{\lambda}V_{\mu}$ differ for each $\lambda$, we get the result.
\end{proof}
\subsection{Case of Tori}\label{ssc:toruscase}
Assume that $G=T$ is a torus.
Then
\begin{align}\label{eqn:HckTsplit}
\Hck_{T,\Spd C}=\coprod_{\lambda\in \bb{X}_*(T_{\ol{F}})}[\Spd C/L^+_{\Spd C}T].
\end{align}
In this case, the $\infty$-category $\cal{D}^{\ULA}(\Hck_{T,\Spd C})$ is $\bb{E}_{\infty}$-monoidal.

To prove the derived Satake equivalence for $\Hck_{T,\Spd C}$, we first show the following:
\begin{prop}\label{prop:DBL+Tequiv}
There is a canonical symmetric monoidal equivalence
\[
\cal{D}^{\ULA}([\Spd C/L^+_{\Spd C}T],\Qlb)\simeq \Sym(\wh{\fk{t}}[-2])\lmod^{\perf}
\]
of symmetric monoidal $\infty$-categories.
\end{prop}
Here the monoidal structure of the left-hand side is given by $-\otimes_{\Qlb}-$.
\begin{proof}
By Proposition \ref{prop:RBL+Gequiv}, there is a functor
\begin{align*}
R\Gamma_{BL^+T}\colon \cal{D}^{\ULA}([\Spd C/L^+_{\Spd C}T],\Qlb)&\to R\Gamma([\Spd C/L^+_{\Spd C}T],\Qlb)\lmod \simeq \Sym(\wh{\fk{t}}[-2])\lmod\\
A&\mapsto R\Gamma([\Spd C/L^+_{\Spd C}T],A).
\end{align*}
Let $\cal{IC}_{BL^+T}\subset \cal{D}^{\ULA}([\Spd C/L^+_{\Spd C}T],\Qlb)$ be the essential image of the unit functor
\[
\rm{Vect}_{\Qlb}^{\perf}\to \cal{D}^{\ULA}([\Spd C/L^+_{\Spd C}T],\Qlb).
\]
We claim that $\cal{D}^{\ULA}([\Spd C/L^+_{\Spd C}T],\Qlb)$ is generated by $\cal{IC}_{BL^+T}$ under extensions.
In fact, by truncations, $\cal{D}^{\ULA}([\Spd C/L^+_{\Spd C}T],\Qlb)$ is generated by 
\begin{multline*}
\cal{D}^{\ULA}([\Spd C/L^+_{\Spd C}T],\Qlb)^{\heartsuit}:=\\
\{A\in \cal{D}^{\ULA}([\Spd C/L^+_{\Spd C}T],\Qlb)\mid \text{$A$ is concentrated on degree 0}\}    
\end{multline*}
under shifts and extensions.
By the connectedness of $T$, we can show that any object in $\cal{D}^{\ULA}([\Spd C/L^+_{\Spd C}T],\Qlb)^{\heartsuit}$ is isomorphic to $\Qlb^{\oplus n}$ for some $n$.
This can be shown, for example, by reducing to a similar statement for $\cal{D}^{\ULA}([\Spec k/L^+_{\Spec k}T^{\spl}],\Qlb)^{\heartsuit}$ via Proposition \ref{prop:DULAequiv}, and by using the connectedness of $L^+_{\Spec k}T^{\spl}$.
Therefore, $\cal{D}^{\ULA}([\Spd C/L^+_{\Spd C}T],\Qlb)$ is generated by $\Qlb$ under shifts and extensions, and thus by $\cal{IC}_{BL^+T}$ under extensions.

For $A_1,A_2\in \cal{D}^{\ULA}([\Spd C/L^+_{\Spd C}T],\Qlb)$, there is a canonical homomorphism
\begin{align}\label{eqn:RGammaBL+TKunnethhom}
R\Gamma_{BL^+T}(A_1)\otimes_{R\Gamma_{BL^+T}(\Qlb)}R\Gamma_{BL^+T}(A_2)\to R\Gamma_{BL^+T}(A_1\otimes A_2) 
\end{align}
by Lemma \ref{lemm:RGammaKunnethformal}.
This is an equivalence, since we can reduce to the case $A_1,A_2\in \cal{IC}_{BL^+T}$.
Therefore, the functor $R\Gamma_{BL^+T}$ is monoidal.
Moreover, it is symmetric monoidal.
We claim that $R\Gamma_{BL^+T}$ is fully faithful.
In fact, since it is enough to prove full faithfulness on $\cal{IC}_{BL^+T}$, which follows from \ref{prop:RBL+Gequiv}.

Therefore, $R\Gamma_{BL^+T}$ induces the monoidal equivalence between $\cal{D}^{\ULA}([\Spd C/L^+_{\Spd C}T],\Qlb)$ and a full subcategory of $\Sym(\wh{\fk{t}}[-2])\lmod$ generated by the essential image of 
\[
\rm{Vect}_{\Qlb}^{\perf}\to \Sym(\wh{\fk{t}}[-2])\lmod, V\mapsto \Sym(\wh{\fk{t}}[-2])\otimes V
\]
under extensions.
The latter is $\Sym(\wh{\fk{t}}[-2])\lmod^{\perf}$, so we get the result.
\end{proof}
Now we can prove the derived Satake equivalence over $\Spd C$ for tori:
\begin{thm}\label{thm:mainSpdCtorusbody}
There is a canonical symmetric monoidal equivalence
\[
\cal{S}^{der}_{T,\Spd C}\colon \Sym(\wh{\fk{t}}[-2])\lmod^{\perf}(\cal{D}(\Rep(\wh{T})))\simeq \cal{D}^{\ULA}(\Hck_{T,\Spd C},\Qlb)
\]
of symmetric monoidal $\infty$-categories.

Moreover, there is a canonical monoidal equivalence
\[
S_{T,\Spd C}\simeq \cal{S}^{der}_{T,\Spd C}\circ fr|_{\Rep^{\rm{f.d.}}(\wh{T},\Qlb)}\colon \Rep^{\rm{f.d.}}(\wh{T},\Qlb)\to \cal{D}^{\ULA}(\Hck_{T,\Spd C},\Qlb).
\]
\end{thm}
\begin{proof}
By (\ref{eqn:HckTsplit}), there is an equivalence
\[
\cal{D}^{\ULA}(\Hck_{T,\Spd C},\Qlb)\simeq \cal{D}^{\ULA}([\Spd C/L^+_{\Spd C}T],\Qlb) \otimes \bigoplus_{\bb{X}_*(T)}\Vect_{\Qlb}^{\perf}, 
\]
and this is monoidal when we endow $\bigoplus_{\bb{X}_*(T_{\ol{F}})}\Vect_{\Qlb}$ with a canonical $\bb{X}_*(T_{\ol{F}})$-graded monoidal structure.
On the other hand, since $\Sym(\wh{\fk{t}}[-2])$ has the trivial $\wh{T}$-action, there is a canonical monoidal equivalence
\[
\Sym(\wh{\fk{t}}[-2])\lmod^{\perf}(\cal{D}(\Rep(\wh{T})))\simeq \Sym(\wh{\fk{t}}[-2])\lmod^{\perf}\otimes \cal{D}^b(\Rep^{\rm{f.d.}}(\wh{T})).
\]
We also have a monoidal equivalence
\[
\Rep^{\rm{f.d.}}(\wh{T},\Qlb)\simeq \bigoplus_{\bb{X}_*(T_{\ol{F}})}\Vect_{\Qlb}^{\heartsuit, \rm{f.d.}}, 
\]
where the right-hand side is endowed with a canonical $\bb{X}_*(T_{\ol{F}})$-graded (non-derived) monoidal structure.
It follows from the universality of $\cal{D}^b(-)$ in \cite[Corollary 7.4.12]{BCKW} that $\cal{D}^b(-)$ preserves direct sums.
Therefore, we have a monoidal equivalence
\begin{align*}
\Sym(\wh{\fk{t}}[-2])\lmod^{\perf}\otimes \cal{D}^b(\Rep^{\rm{f.d.}}(\wh{T}))&\simeq \Sym(\wh{\fk{t}}[-2])\lmod^{\perf}\otimes \cal{D}^b\left(\bigoplus_{\bb{X}_*(T_{\ol{F}})}\Vect_{\Qlb}^{\heartsuit,\rm{f.d.}}\right)\\
&\simeq \Sym(\wh{\fk{t}}[-2])\lmod^{\perf}\otimes \bigoplus_{\bb{X}_*(T_{\ol{F}})}\Vect_{\Qlb}^{\perf}.
\end{align*}
Now the first statement of the theorem follows from Proposition \ref{prop:DBL+Tequiv}.

The (non-derived) geometric Satake equivalence 
\[
\Sat(\Hck_{T,\Spd C},\Qlb)\simeq \Rep^{\rm{f.d.}}(\wh{T},\Qlb)
\]
is, by construction, isomorphic to the composition
\begin{align*}
\Sat(\Hck_{T,\Spd C},\Qlb)&= \cal{D}^{\ULA}(\Hck_{T,\Spd C},\Qlb)^{\heartsuit}
\\
&=\cal{D}^{\ULA}\left(\coprod_{\lambda\in \bb{X}^*(T_{\ol{F}})}[\Spd C/L^+_{\Spd C}T],\Qlb\right)^{\heartsuit}\\
&\simeq \bigoplus_{\lambda\in \bb{X}^*(T_{\ol{F}})}\cal{D}^{\ULA}([\Spd C/L^+_{\Spd C}T],\Qlb)^{\heartsuit}\\
&\xrightarrow[\simeq]{\bigoplus (-)^*} \bigoplus_{\lambda\in \bb{X}^*(T_{\ol{F}})}\cal{D}^{\ULA}(\Spd C,\Qlb)^{\heartsuit}\\
&\simeq \bigoplus_{\lambda\in \bb{X}^*(T_{\ol{F}})}\Vect_{\Qlb}^{\heartsuit,\rm{f.d.}}
\simeq \Rep^{\rm{f.d.}}(\wh{T},\Qlb).
\end{align*}
Therefore, it is easy to check that there is a monoidal equivalence $\cal{S}^{der}_{T,\Spd C}\circ fr\simeq \cal{S}_{T,\Spd C}$.
\end{proof}
\subsection{General case}\label{ssc:generalcase}
Assume that $G$ is any connected reductive group over $F$.
In this section, we prove the derived Satake equivalence:
\begin{thm}\label{thm:mainSpdCbody}
There is a canonical monoidal equivalence
\[
\cal{S}^{der}_{G,\Spd C}\colon \cal{D}^{\ULA}(\Hck_{G,\Spd C},\Qlb)\simeq \Sym(\wh{\fk{g}}[-2])\lmod^{\perf}(\cal{D}\Rep^{\rm{f.d.}}(\wh{G})).
\]
\end{thm}
For this, we need the following lemmas:
\begin{lemm}\label{lemm:exttensorDULA}
Let $G'$ be a reductive group over $F$ and $T'$ a torus over $F$.
Then the natural monoidal functor
\begin{align*}
-\boxtimes_{\Spd C} -\colon \cal{D}^{\ULA}(\Hck_{T',\Spd C},\Qlb) \otimes \cal{D}^{\ULA}(\Hck_{G',\Spd C},\Qlb)\to \cal{D}^{\ULA}(\Hck_{T'\times G',\Spd C},\Qlb)
\end{align*}
is an equivalence.
\end{lemm}
\begin{proof}
By Proposition \ref{prop:DULAequiv}, it suffices to show that
\[
\cal{D}^{\ULA}(\Hck^{\eq}_{T',\Spec k},\Qlb) \otimes \cal{D}^{\ULA}(\Hck^{\eq}_{G',\Spec k},\Qlb)\to \cal{D}^{\ULA}(\Hck^{\eq}_{T'\times G',\Spec k},\Qlb)
\]
is an equivalence, where we omit the superscript $(-)^{\spl}$.

First, we show the full faithfulness.
By the same argument as \cite[{\sc Corollary VI.6.7}]{FS}, it suffices to show that the functor
\begin{multline}\label{eqn:DULAKunnethStrata}
\cal{D}^{\ULA}([\Spec k/L^+_{\Spec k}T'],\Qlb) \otimes \cal{D}^{\ULA}(\Hck^{\eq}_{G',\Spec k,\mu},\Qlb)\\
\to \cal{D}^{\ULA}([\Spec k/L^+_{\Spec k}T']\times_{\Spec k}\Hck_{G',\Spec k,\mu},\Qlb)    
\end{multline}
for any cocharacter $\mu$ of a maximal torus of $G'_{\ol{F}}$.
Since $T'$ is a quotient of $L^+_{\Spd C}T'$ by a pro-unipotent group, we can replace $L^+_{\Spd C}T'$ by $T'$.

Therefore, it suffices to show that 
\begin{multline*}
R\Hom_{D^{\ULA}([\Spec k/T'],\Qlb)}(A_1,A_2)\otimes^{\bb{L}}R\Hom_{D^{\ULA}(\Hck_{G',\Spec k}^{\eq},\Qlb)}(B_1,B_2)\\
\to R\Hom_{D^{\ULA}([\Spec k/T']\times \Hck_{G',\Spec k}^{\eq},\Qlb)}(A_1\boxtimes B_1,A_2\boxtimes B_2)
\end{multline*}
is an isomorphism for any $A_1,A_2\in D^{\ULA}([\Spec k/T'],\Qlb)$ and $B_1,B_2\in D^{\ULA}(\Hck_{G',\Spec k}^{\eq},\Qlb)$.
By truncation, we can assume that $A_1$ and $A_2$ are concentrated in degree 0.
Since $T'$ is connected, we can assume that $A_1=A_2=\Qlb$.
Therefore, it suffices to show
\begin{multline*}
R\Gamma([\Spec k/T'],\Qlb)\otimes^{\bb{L}}R\Hom_{D^{\ULA}(\Hck_{G',\Spec k}^{\eq},\Qlb)}(B_1,B_2)\\
\to R\Hom_{D^{\ULA}([\Spec k/T']\times \Hck_{G',\Spec k}^{\eq},\Qlb)}(\rm{pr}_2^* B_1,\rm{pr}_2^* B_2),    
\end{multline*}
where $\rm{pr}_2\colon [\Spec k/T']\times \Hck_{G',\Spec k}^{\eq}\to \Hck_{G',\Spec k}^{\eq}$ is the projection.
We claim that the natural map
\[
\rm{pr}_2^* R\HHom(B_1, B_2) \to R\HHom(\rm{pr}_2^* B_1,\rm{pr}_2^* B_2)
\]
is an isomorphism.
It is enough to check this after the pullback along the natural projection $\pi_{T'}\colon \Spec k\times \Hck_{G',\Spec k}^{\eq}\to [\Spec k/T']\times \Hck_{G',\Spec k}^{\eq}$.
Since $\pi_{T'}$ is smooth, it suffices to show that
\[
\pi_{T'}^*\rm{pr}_2^* R\HHom(B_1, B_2) \to R\HHom(\pi_{T'}^*\rm{pr}_2^* B_1,\pi_{T'}^*\rm{pr}_2^* B_2)
\]
is an isomorphism, which is true since $\rm{pr}_2\circ \pi_{T'}=\rm{id}$.
We get the claim.

By this claim, replacing $B_1$ by $\Qlb$ and $B_2$ by $R\HHom(B_1,B_2)$, we may assume that $B_1=\Qlb$.
Moreover, by truncation, we may assume that $B_2$ is perverse.
Then it suffices to show the isomorphism of the form 
\[
H^*_{T'}(*,\Qlb)\otimes H^*_{L^{+,\eq}_{\Spec k}G'}(\Gr_{G,\Spec k}^{\eq},B_2)\cong H^*_{T'\times L^{+,\eq}_{\Spec k}G'}(\Gr_{G,\Spec k}^{\eq},B_2).
\]
After replacing $\Gr_{G,\Spec k}^{\eq}$ with the support of $B_2$ and $L^{+,\eq}_{\Spec k}G'$ with its certain quotient of finite type, this follows from an equivariant K\"{u}nneth formula, see \cite[(A.1.12)]{Zhueq}.

It remains to show the essential surjectivity.
By excisions, it is enough to prove that (\ref{eqn:DULAKunnethStrata}) is essentially surjective.
By truncations, it suffices to show that $A\in \cal{D}^{\ULA}([\Spec k/L^+_{\Spec k}T']\times_{\Spec k}\Hck_{G',\Spec k,\mu},\Qlb)$ concentrated in degree 0 is in the essential image.
By the above argument using the connectedness of $T'$ and $G'$, we can assume that $A=\Qlb$, which is an image of $\Qlb\boxtimes \Qlb$.
\end{proof}
\begin{lemm}\label{lemm:gerbecoh}
Let $\alpha\colon Y\to X$ be a map of small v-stacks and $K\to X$ a finite \'{e}tale group over $X$.
Assume that there is a $\ell$-cohomologically smooth morphism $X'\to X$ such that $Y\times_{X}X'$ is isomorphic to $[X/K]\times_X X'$ over $X'$.
Then for any $\cal{F}\in \cal{D}_{\et}(X,\Qlb)$, the natural map
\[
\cal{F}\to R\alpha_*\alpha^* \cal{F}
\]
is an isomorphism.
\end{lemm}
\begin{proof}
By the smooth base change, we can assume $Y=[X/K]$.
Let $\beta\colon X\to [X/K]$ be the natural map.
Since the composition
\begin{align}\label{eqn:Falphacomposition}
\cal{F}\to R\alpha_*\alpha^*\cal{F}\to R\alpha_*R\beta_*\beta^*\alpha^*\cal{F}=\cal{F}
\end{align}
is the identity, $\cal{F}\to R\alpha_*\alpha^*\cal{F}$ is a split monomorphism.
On the other hand, since $K$ is a finite \'{e}tale group over $X$, the map $\beta$ is \'{e}tale and proper.
Hence, for $\cal{G}\in \cal{D}_{\et}([X/K],\Qlb)$ we have a natural composition 
\[
u\colon \cal{G}\to R\beta_*\beta^*\cal{G}\simeq R\beta_!R\beta^!\cal{G}\to \cal{G}.
\]
By the base change, the induced map on the fibers at any geometric point $x$ is of the form
\[
u_x\colon \cal{G}_x\to \cal{G}_x^{\oplus r}\to \cal{G}_x
\]
for some positive integer $r$, where the first map is the diagonal map, and the second map is the addition.
Therefore, $u_x$ is the multiplication by $r$, which is an isomorphism since the coefficient ring $\Qlb$ contains $\QQ$.
Thus $u$ is an isomorphism, and applying to the case $\cal{G}=\alpha^*\cal{F}$, the second map in (\ref{eqn:Falphacomposition}) is a split monomorphism.
This proves the claim.
\end{proof}
\begin{proof}[Proof of Theorem \ref{thm:mainSpdCbody}]
If $G$ is the product $T'\times G'$ of a torus $T'$ and a semisimple and simply-connected group $G'$, the theorem follows from Lemma \ref{lemm:exttensorDULA}, Theorem \ref{thm:mainSpdCssscbody}, and Theorem \ref{thm:mainSpdCtorusbody}.

Now for a general $G$, let $\varphi\colon \wt{G}\to G$ be a central isogeny such that $\wt{G}$ is the product of a torus and a simply-connected semisimple group.
Let $K$ denote the finite kernel of $\wt{G}\to G$.
Put $\bb{X}_*^+:= \bb{X}_*^+(T)$ and $\wt{\bb{X}}_*^+:=\bb{X}_*^+(\varphi^{-1}T)$.

It is known that the natural map 
\[
\Gr_{\wt{G},\Spd C}\to \Gr_{G,\Spd C}
\]
is the open and closed immersion, whose image is $\bigcup_{\mu\in \wt{\bb{X}}_*^+\subset \bb{X}_*^+}\Gr_{G,\Spd C,\mu}$ (see the argument in \cite[\S 5]{PR}, for example).
Moreover, letting $\Hck_{G,\Spd C,\wt{\bb{X}}^+_*}\subset \Hck_{G,\Spd C}$ denote the image of $\Gr_{\wt{G},\Spd C}\subset \Gr_{G,\Spd C}$,  there is a decomposition
\[
\Hck_{G,\Spd C}\simeq \coprod_{\bb{X}_*^+/\wt{\bb{X}}_*^+}\Hck_{G,\Spd C,\wt{\bb{X}}^+_*}.
\]
Therefore, we have an equivalence
\begin{align}\label{eqn:DULAnonssscsplit}
\cal{D}^{\ULA}(\Hck_{G,\Spd C},\Qlb)\simeq \cal{D}^{\ULA}(\Hck_{G,\Spd C,\wt{\bb{X}}^+_*},\Qlb)\otimes \bigoplus_{\bb{X}_*^+/\wt{\bb{X}}_*^+}\Vect_{\Qlb}^{\perf}, 
\end{align}
and when we endow $\bigoplus_{\bb{X}_*^+/\wt{\bb{X}}_*^+}\Vect_{\Qlb}^{\perf}$ with a canonical $\bb{X}_*^+/\wt{\bb{X}}_*^+$-graded symmetric monoidal structure, this equivalence is monoidal.

We first claim that there is a monoidal equivalence
\begin{align}\label{eqn:DULAcentralisogenyequiv}
\cal{D}^{\ULA}(\Hck_{G,\Spd C,\wt{\bb{X}}^+_*},\Qlb)\simeq \cal{D}^{\ULA}(\Hck_{\wt{G},\Spd C},\Qlb)
\end{align}
Let
\begin{align*}
r_\varphi \colon \Hck_{\wt{G},\Spd C}=[L^+_{\Spd C}\wt{G}\backslash \Gr_{\wt{G},\Spd C}]\to [L^+_{\Spd C}G\backslash \Gr_{\wt{G},\Spd C}]\simeq \Hck_{G,\Spd C,\wt{\bb{X}}^+_*}
\end{align*}
denote the natural maps.

Then the kernel of $L^+_{\Spd C}\wt{G}\to L^+_{\Spd C}G$ is $K^{\diamondsuit}$, which is a finite group over $\Spd C$.
It follows that the pullback functor
\[
r_{\varphi}^*\colon \cal{D}^{\ULA}(\Hck_{G,\Spd C,\wt{\bb{X}}^+_*},\Qlb)\to \cal{D}^{\ULA}(\Hck_{\wt{G},\Spd C},\Qlb)
\]
is fully faithful by Lemma \ref{lemm:gerbecoh}.
The essential surjectivity of $r_{\varphi}$ reduces to that of 
\[
\cal{D}^{\ULA}(\Hck_{G,\Spd C,\mu},\Qlb)\to \cal{D}^{\ULA}(\Hck_{\wt{G},\Spd C,\mu},\Qlb)
\]
for any $\mu\in \wt{\bb{X}}^+_*$ by excisions.
By the same argument as the proof of Lemma \ref{lemm:exttensorDULA} using truncations, it suffices to show that $\Qlb$ is in the essential image, which is clear.
Therefore, $r_{\varphi}^*$ is an equivalence.

Next, we claim that $r_{\varphi}^*$ is monoidal.
Let $(L_{\Spd C}G)_{\wt{\bb{X}}_*^+}$ denote the image of $L_{\Spd C}\wt{G}\to L_{\Spd C}G$.
Note that the natural homomorphism $[L_{\Spd C}G/L^+_{\Spd C}G]\to [(L_{\Spd C}G)_{\wt{\bb{X}}_*^+}/L^+_{\Spd C}G]$ is an isomorphism.
The category $\cal{D}^{\ULA}(\Hck_{G,\Spd C,\wt{\bb{X}}^+_*},\Qlb)$ is monoidal via the convolution product with respect to a convolution diagram of the form
\begin{multline*}
[L^+_{\Spd C}G\backslash (L_{\Spd C}G)_{\wt{\bb{X}}_*^+}/L^+_{\Spd C}G]\times_{\Spd C}[L^+_{\Spd C}G\backslash (L_{\Spd C}G)_{\wt{\bb{X}}_*^+}/L^+_{\Spd C}G]\\
\leftarrow[L^+_{\Spd C}G\backslash (L_{\Spd C}G)_{\wt{\bb{X}}_*^+}\times_{\Spd C}^{L^+_{\Spd C}G}(L_{\Spd C}G)_{\wt{\bb{X}}_*^+}/L^+_{\Spd C}G]\\
\to[L^+_{\Spd C}G\backslash (L_{\Spd C}G)_{\wt{\bb{X}}_*^+}/L^+_{\Spd C}G].
\end{multline*}
On the other hand, the convolution product of $\cal{D}^{\ULA}(\Hck_{\wt{G},\Spd C},\Qlb)$ comes from a convolution diagram of the form
\begin{multline*}
[L^+_{\Spd C}\wt{G}\backslash L_{\Spd C}\wt{G}/L^+_{\Spd C}\wt{G}]\times_{\Spd C}[L^+_{\Spd C}\wt{G}\backslash L_{\Spd C}\wt{G}/L^+_{\Spd C}\wt{G}]\\
\leftarrow[L^+_{\Spd C}\wt{G}\backslash L_{\Spd C}\wt{G}\times_{\Spd C}^{L^+_{\Spd C}\wt{G}}L_{\Spd C}\wt{G}/L^+_{\Spd C}\wt{G}]\\
\to[L^+_{\Spd C}\wt{G}\backslash L_{\Spd C}\wt{G}/L^+_{\Spd C}\wt{G}].
\end{multline*}
From the isomorphism $[L_{\Spd C}G/L^+_{\Spd C}G]\to [(L_{\Spd C}G)_{\wt{\bb{X}}_*^+}/L^+_{\Spd C}G]$, one can easily show that the square
\[
\xymatrix{
[L^+_{\Spd C}\wt{G}\backslash L_{\Spd C}\wt{G}\times_{\Spd C}^{L^+_{\Spd C}\wt{G}}L_{\Spd C}\wt{G}/L^+_{\Spd C}\wt{G}]\ar[r]\ar[d]&
[L^+_{\Spd C}\wt{G}\backslash L_{\Spd C}\wt{G}/L^+_{\Spd C}\wt{G}]\ar[d]^{r_{\varphi}^*}\\
L^+_{\Spd C}G\backslash (L_{\Spd C}G)_{\wt{\bb{X}}_*^+}\times_{\Spd C}^{L^+_{\Spd C}G}(L_{\Spd C}G)_{\wt{\bb{X}}_*^+}/L^+_{\Spd C}G]\ar[r]&[L^+_{\Spd C}G\backslash (L_{\Spd C}G)_{\wt{\bb{X}}_*^+}/L^+_{\Spd C}G]
}
\]
is cartesian.
This and similar results on the convolution products of higher degrees show monoidality of $r_{\varphi}^*$.
Thus we get a monoidal equivalence (\ref{eqn:DULAcentralisogenyequiv}).

Combining this with (\ref{eqn:DULAnonssscsplit}), we have a monoidal equivalence
\[
\cal{D}^{\ULA}(\Hck_{G,\Spd C},\Qlb)\simeq \cal{D}^{\ULA}(\Hck_{\wt{G},\Spd C},\Qlb)\otimes \bigoplus_{\bb{X}_*^+/\wt{\bb{X}}_*^+}\Vect_{\Qlb}^{\perf}.
\]
On the other hand, via the decomposition by central characters, there is a monoidal equivalence
\[
\Rep^{\rm{f.d.}}(\wt{G},\Qlb)=\bigoplus_{\bb{X}_*^+/\wt{\bb{X}}_*^+} \Rep^{\rm{f.d.}}(\wh{\wt{G}},\Qlb),
\]
where the right-hand side has a canonical $\bb{X}_*^+/\wt{\bb{X}}_*^+$-graded monoidal structure.
Since $\wh{G}$-action on $\Sym(\wh{\fk{g}}[-2])$ factors through $\wh{\wt{G}}$, there is a natural isomorphism
\[
\Sym(\wh{\fk{g}}[-2])\lmod^{\perf}(\cal{D}(\Rep(\wh{G})))\simeq \Sym(\wh{\fk{g}}[-2])\lmod^{\perf}(\cal{D}(\Rep(\wh{\wt{G}})))\otimes_{\cal{D}^b(\Rep^{\rm{f.d.}}(\wh{\wt{G}}))}\cal{D}^b(\Rep^{\rm{f.d.}}(\wh{G})).
\]
Using this, we can prove the proposition by the same argument as the proof of Lemma \ref{thm:mainSpdCtorusbody}.
\end{proof}
The compatibility with the non-derived geometric Satake holds:
\begin{prop}
There is a canonical monoidal equivalence
\[
S_{G,\Spd C}\simeq \cal{S}^{der}_{G,\Spd C}\circ fr|_{\Rep^{\rm{f.d.}}(\wh{G},\Qlb)}\colon \cal{D}^b(\Rep^{\rm{f.d.}}(\wh{G},\Qlb))\to \cal{D}^{\ULA}(\Hck_{G,\Spd C},\Qlb).
\]
\end{prop}
\begin{proof}
First, assume that $G$ is the product $T'\times G'$ of a torus $T'$ and a semisimple and simply-connected group $G'$.
By the construction of the non-derived geometric Satake
\[
S_{G,\Spd C}\colon \Sat(\Hck_{G,\Spd C},\Qlb)\simeq \Rep^{\rm{f.d.}}(\wh{G},\Qlb),
\]
one can prove that, for $A\in \Sat(\Hck_{G',\Spd C},\Qlb)$ and $B\in \Sat(\Hck_{T',\Spd C},\Qlb)$, there is a canonical isomorphism
\[
S_{G',\Spd C}(A)\boxtimes S_{T',\Spd C}(B)\simeq S_{G'\times T',\Spd C}(A\boxtimes_{\Spd C}B) \in \Rep^{\rm{f.d.}}(\wh{G}'\times \wh{T}',\Qlb), 
\]
compatible with the monoidal structures.
This shows that there is a monoidal natural isomorphism that makes the square 
\[
\xymatrix{
\cal{D}^b(\Rep^{\rm{f.d.}}(\wh{G}'))\otimes \cal{D}^b(\Rep^{\rm{f.d.}}(\wh{T}'))
\ar[d]_{\cal{S}_{G',\Spd C}\otimes \cal{S}_{T',\Spd C}}
\ar[r]^-{-\boxtimes -}
&
\cal{D}^b(\Rep^{\rm{f.d.}}(\wh{G}'\times \wh{T}'))
\ar[d]^{\cal{S}_{G'\times T',\Spd C}}
\\
\cal{D}^{\ULA}(\Hck_{G',\Spd C})\otimes \cal{D}^{\ULA}(\Hck_{T',\Spd C})
\ar[r]_-{-\boxtimes -}
&
\cal{D}^{\ULA}(\Hck_{G'\times T',\Spd C})
}
\]
commute.
Since the proposition holds for $G'$ and $T'$, it follows that it holds for $G'\times T'$.

Now for a general $G$, let $\varphi\colon \wt{G}\to G$ be a central isogeny such that $\wt{G}$ is the product of a torus and a simply-connected semisimple group.
By the construction of the non-derived geometric Satake, one can show that the composition 
\begin{align*}
\cal{D}^{\ULA}(\Hck_{G,\Spd C},\Qlb)
&\xrightarrow[\simeq]{(\ref{eqn:DULAnonssscsplit})} 
\bigoplus_{\bb{X}_*^+/\wt{\bb{X}}_*^+}\cal{D}^{\ULA}(\Hck_{G,\Spd C,\wt{\bb{X}}^+_*},\Qlb)\\
&\xrightarrow[\simeq]{(\ref{eqn:DULAnonssscsplit})} \bigoplus_{\bb{X}_*^+/\wt{\bb{X}}_*^+}\cal{D}^{\ULA}(\Hck_{\wt{G},\Spd C},\Qlb)
\end{align*}
restrict to 
\[
\Sat(\Hck_{G,\Spd C},\Qlb)\simeq\bigoplus_{\bb{X}_*^+/\wt{\bb{X}}_*^+}\Sat(\Hck_{\wt{G},\Spd C},\Qlb),
\]
and the composition
\begin{align*}
\Sat(\Hck_{G,\Spd C},\Qlb)&\simeq\bigoplus_{\bb{X}_*^+/\wt{\bb{X}}_*^+}\Sat(\Hck_{\wt{G},\Spd C},\Qlb)\\
&\xrightarrow[\simeq]{S_{\wt{G},\Spd C}}\bigoplus_{\bb{X}_*^+/\wt{\bb{X}}_*^+}\Rep^{\rm{f.d.}}(\wh{\wt{G}},\Qlb)\\
&\simeq \Rep^{\rm{f.d.}}(\wh{G},\Qlb)
\end{align*}
is monoidally isomorphic to $S_{G,\Spd C}$.
This implies the claim for $G$.
\end{proof}
\section{Derived Satake over $\Spd \breve{E}$ and Weil action}\label{sec:derSatoverSpdEbreve}
Let $\Catinf{\Qlb}$ denote the $\infty$-category of $\Qlb$-linear $\infty$-categories.
For any (discrete) group $H$, let $\Catinf{\Qlb}^{BH}:=\rm{Fun}(BH\to \Catinf{\Qlb}^{mon})$ denote the $\infty$-category of $\Qlb$-linear monoidal $\infty$-categories with $H$-actions.
We write $\Catinf{\Qlb}^{mon, BH}$ for the $\infty$-category of $\Qlb$-linear monoidal $\infty$-categories with monoidal $H$-actions.
Let $E/F$ be a finite Galois extension such that $G$ is unramified over $E$.
We write $W_{F}^{\rm{disc}}$ (resp.~$I_E^{\rm{disc}}$) for the group $W_{F}$ (resp.~$I_E$) with the discrete topology.
In this section, we prove the derived Satake equivalence over $\Spd \breve{E}$ in $\Catinf{\Qlb}^{mon, B(W_{F}/I_E)}$.
The strategy is similar to \S \ref{sec:derSatoverSpdC}, except that we need to take $W_F/I_E$-equivariance into account:
First in \S \ref{ssc:SpdEbrevetoSpdkIE}, we will show that a degeneration result similar to \cite[{\sc Corollary} VI.6.7]{FS} holds for $\Spd \breve{E}$ and $[\Spd k/I_E]$, instead of $\Spd C$ and $\Spd k$.
In \S \ref{ssc:semisimplesimplyconnectedEbreve}, using the result in \S \ref{ssc:SpdEbrevetoSpdkIE} and assuming semisimplicity and simply-connectedness of $G$, we will define the cochain functor
\[
R\Gamma_{\Spd \breve{E}} \colon \cal{D}^{\ULA}(\Hck_{G,\Spd \breve{E}},\Qlb)\to \cal{O}_{\bb{T}(\wh{\fk{t}}^*/W)}^{\shear}\lmod(\DDRep(I_E)).
\]
in $\Catinf{\Qlb}^{mon, B(W_{F}/I_E)}$ and the Kostant functor
\[
\kappa_{I_E}^{\shear}\colon \rm{Sym}(\wh{\fk{g}}[-2])\lmod^{\perf}(\DDRep(\wh{G}\times I_E)) \to \cal{O}_{\bb{T}(\wh{\fk{t}}^*/W)}^{\shear}\lmod(\DDRep(I_E))
\]
in $\Catinf{\Qlb}^{mon, B(W_{F}/I_E)}$.
Then we will show that these two functors are fully faithful and have the same essential image, and hence we get the desired equivalence in $\Catinf{\Qlb}^{mon, B(W_{F}/I_E)}$.
In \S \ref{ssc:toruscaseEbreve}, we will prove the desired equivalence for tori, and in \S \ref{ssc:generalcaseEbreve} for general $G$, using a central isogeny to $G$ from a product of a torus and a semisimple and simply-connected group.
\subsection{From $\Spd \breve{E}$ to $[\Spd k/I_E]$}\label{ssc:SpdEbrevetoSpdkIE}
In this subsection, we prove several lemmas that we use later.
First, problems on $\cal{D}^{\ULA}(\Hck_{G,[\Spd \cal{O}_C/I_E]},\Lambda)$ can be reduced to that on $\cal{D}^{\ULA}(\Hck_{G,[\Spd k/I_E]},\Lambda)$.
\begin{lemm}\label{lemm:DULAequivIF}
There are canonical monoidal equivalences
\begin{align}
\cal{D}^{\ULA}(\Hck_{G,[\Spd C/I_E]},\Lambda)&\simeq \cal{D}^{\ULA}(\Hck_{G^{\spl},[\Spd C/I_E]},\Lambda)\notag\\
&\leftarrow \cal{D}^{\ULA}(\Hck_{G^{\spl},[\Spd \cal{O}_C/I_E]},\Lambda)\notag\\
&\to \cal{D}^{\ULA}(\Hck_{G^{\spl},[\Spd k/I_E]},\Lambda) \label{eqn:DULAequivIF}
\end{align}
\end{lemm}
\begin{proof}
We may assume that $\Lambda$ is a torsion ring.
The lemma follows from the same argument as \cite[{\sc Corollary} VI.6.7]{FS} (or \cite[Theorem 5.1]{Ban}), by using that the two pullback functors
\[
\cal{D}_{\rm{lc}}([\Spd k/I_E],\Lambda)\xleftarrow{\sim}\cal{D}_{\rm{lc}}([\Spd \cal{O}_C/I_E],\Lambda)\xrightarrow{\sim} \cal{D}_{\rm{lc}}([\Spd C/I_E],\Lambda)
\]
are equivalences by Lemma \ref{lemm:DlcSpdkIFRepIFequiv}, Lemma \ref{lemm:SpdkIFSpdCIFequiv} and Lemma \ref{lemm:SpdkIFSpdOCIFequiv}.
\end{proof}
Moreover, the category $\cal{D}^{\ULA}(\Hck_{G,[\Spd k/I_E]},\Lambda)$ also decomposes into a tensor product:
\begin{lemm}\label{lemm:HckBIFKunneth}
The canonical functor
\begin{multline}\label{eqn:HckBIFKunneth}
-\boxtimes -\colon \cal{D}^{\ULA}(\Hck_{G,\Spd k},\Lambda) \otimes_{\Lambda\lmod^{\perf}} \cal{D}_{\rm{lc}}([\Spd k/I_E],\Lambda)\\
\to \cal{D}^{\ULA/[\Spd k/I_E]}(\Hck_{G,\Spd k}\times [\Spd k/I_E],\Lambda)
\end{multline}
is an equivalence.
\end{lemm}
\begin{proof}
We can assume that $\Lambda$ is a torison ring.

First, we show the full faithfulness.
By Lemma \ref{lemm:BIFSpdFbreveff}, it suffices to show the full faithfulness of the functor
\begin{multline*}
-\boxtimes -\colon \cal{D}^{\ULA}(\Hck_{G,\Spd k},\Lambda) \otimes \cal{D}_{\rm{lc}}(\Spd \breve{E},\Lambda)\\
\to \cal{D}^{\ULA/[\Spd k/I_E]}(\Hck_{G,\Spd k}\times \Spd \breve{E},\Lambda).
\end{multline*}
Consider the diagram
\[
\xymatrix{
\Hck_{G,\Spd k}\times \Spd \breve{E}\ar[r]^-{\rm{pr}_1}\ar[d]_{\rm{pr}_2}&\Hck_{G,\Spd k}\ar[d]^{\alpha}\\
\Spd \breve{E}\ar[r]_-{\beta}&\Spd k.
}
\]
It suffices to show that for $A,A'\in \cal{D}^{\ULA}(\Hck_{G,\Spd k},\Lambda)$ and for $B,B'\in \cal{D}_{\rm{lc}}(\Spd \breve{E},\Lambda)$, there is a canonical isomorphism
\[
R\alpha_*R\HHom_{\Lambda}(A,A')\otimes_{\Lambda}R\beta_*R\HHom_{\Lambda}(B,B')\xrightarrow{\sim} R(\alpha\times \beta)_*R\HHom_{\Lambda}(A\boxtimes B,A'\boxtimes B')
\]
in $\cal{D}^{\ULA/[\Spd k/I_E]}(\Hck_{G,\Spd k}\times \Spd \breve{E},\Lambda)$, see \cite[Proposition 3.8]{RS_Kunneth}.
Since $B$ and $B'$ are locally constant complexes with perfect fibers, it follows that
\begin{align*}
R\HHom_{\Lambda}(A\boxtimes B,A'\boxtimes B')&\cong R\HHom_{\Lambda}(\rm{pr}_1^*A,\rm{pr}_1^*A')\otimes (\rm{pr}_2^*B)^{\vee}\otimes \rm{pr}_2^*B'\\
&\cong \rm{pr}_1^*R\HHom(A,A')\otimes (\rm{pr}_2^*B)^{\vee}\otimes \rm{pr}_2^*B'\\
&\cong R\HHom(A,A')\boxtimes R\HHom(B,B'), 
\end{align*}
where the second isomorphism follows from the cohomologically smoothness of $\beta$ (implied by  \cite[Proposition 24.5]{Sch}).
Now it suffices to show the the K\"{u}nneth formula for $\Hck_{G,\Spd k}\to \Spd k \leftarrow \Spd \breve{E}$, i.e. that there is a canonical isomorphism
\[
R\alpha_*A\otimes R\beta_*B\xrightarrow{\sim}R(\alpha\times \beta)_*(A\boxtimes B)
\]
for $A\in \cal{D}_{\et}(\Hck_{G,\Spd k},\Lambda)$ and $B\in \cal{D}_{\rm{lc}}(\Spd \breve{E},\Lambda)$.
This follows from the following three isomorphisms:
\begin{align*}
R\alpha_*A\otimes R\beta_* B&\cong R\beta_*(\beta^*R\alpha_*A\otimes B), \\
\beta^*R\alpha_*A&\cong R\rm{pr}_{2*}\rm{pr}_1^*A,\\
B\otimes R\rm{pr}_{2*}\rm{pr}_1^*A&\cong R\rm{pr}_{2*}(\rm{pr}_2^*B\otimes \rm{pr}_1^*A).
\end{align*}
The first isomorphism follows from the projection formula, noting that we can show that $\beta$ is proper by the same argument as \cite[{\sc Proposition} II.1.21]{FS}, the second follows from the smooth base change and the cohomologically smoothness of $\beta$, and the third follows the fact that $B$ is locally constant with perfect fibers.

It remains to show the essential surjectivity.
By excisions, it is enough to prove that 
\begin{multline*}
\cal{D}^{\ULA}(\Hck_{G,\Spd k,\mu},\Lambda) \otimes \cal{D}_{\rm{lc}}([\Spd k/I_E],\Lambda)\\
\to \cal{D}^{\ULA/[\Spd k/I_E]}(\Hck_{G,\Spd k,\mu}\times [\Spd k/I_E],\Lambda)    
\end{multline*}
is essentially surjective for any $\mu\in \bb{X}_*^+(T_{\ol{F}})$.
By truncations, it suffices to show that an object $A\in \cal{D}^{\ULA/[\Spd k/I_E]}(\Hck_{G,\Spd k,\mu}\times [\Spd k/I_E],\Lambda)$ concentrated in degree 0 is in the essential image.
Since $\Hck_{G,\Spd k,\mu}$ is isomorphic to a classifying space of a connected group over $\Spd k$ as in the proof of Lemma \ref{lemm:exttensorDULA}, we may assume that $A$ is the pullback of a sheaf in $\cal{D}_{\rm{lc}}([\Spd k/I_E],\Lambda)$, which is obviously in the essential image.
\end{proof}
\subsection{Non-derived Satake equivalence and Weil action}\label{sssc:nonderWeil}
As in (\ref{eqn:nonderivedSpdbreveFLambda}), there is a symmetric monoidal equivalence
\begin{align}\label{eqn:nonderivedSpdbreveF}
S_{\Spd \breve{E}}\colon \Rep^{\rm{f.d.}}(\wh{G}\times I_E,\Qlb)\xrightarrow{\simeq} \Sat(\Hck_{G,\Spd \breve{E}},\Qlb)\subset \cal{D}^{\ULA}(\Hck_{G,\Spd \breve{E}},\Qlb).
\end{align}
which is $W_{F}/I_E$-equivalent by the argument in \cite[\S VI.11]{FS}.

Since any object in $\Rep^{\rm{f.d.}}(\wh{G}\times I_E,\Qlb)$ is isomorphic to a finite direct sum of the form
\[
\ds\bigoplus_{i=1}^n V_i\boxtimes W_i,
\]
where $n\in \bb{Z}_{\geq 0}$, $V_i\in \Rep^{\rm{f.p.}}(\wh{G},\Qlb)$ and $W_i\in \Rep^{\rm{f.p.}}(I_E,\Qlb)$, there is a functor
\begin{align}
\Rep^{\rm{f.d.}}(\wh{G}\times I_E,\Qlb)
&\to \cal{D}^{\ULA}(\Hck_{G,\Spd k},\Qlb) \otimes \cal{D}_{\rm{lc}}([\Spd k/I_E],\Qlb) \label{eqn:SSpdkotimesid}\\
\ds\bigoplus_{i=1}^n V_i\boxtimes W_i 
&\mapsto \bigoplus_{i=1}^n S_{\Spd k}(V_i)\otimes \cal{L}_{W_i},
\end{align}
where $\cal{L}_{W_i}$ is the local system corresponding to $W_i$ under (\ref{eqn:DlcSpdEbreveDRepIEeqiv}).
Write
\begin{multline*}
\Sat(\Hck_{G,\Spd k},\Qlb)\otimes \LocSys([\Spd k/I_E],\Qlb)\\
\subset\cal{D}^{\ULA}(\Hck_{G,\Spd k},\Qlb) \otimes_{\Vect_{\Qlb}^{\perf}} \cal{D}_{\rm{lc}}([\Spd k/I_E],\Qlb)    
\end{multline*}
for its essential image.
This is the full subcategory consisting of objects isomorphic to a finite direct sum of the form
\[
\ds\bigoplus_{i=1}^n A_i\boxtimes \cal{L}_i,
\]
where $n\in \bb{Z}_{\geq 0}$, $A_i\in \Sat(\Hck_{G,\Spd k},\Qlb)$ and $\cal{L}_i\in \LocSys([\Spd k/I_E],\Qlb)$.

The rest of this subsection is devoted to proving the following lemma:
\begin{lemm}\label{lemm:nonderSatSpdEbreve}
\begin{enumerate}
\item The equivalences (\ref{eqn:DULAequivIF}) induce equivalences
\begin{align}\label{eqn:SatequivIF}
\Sat(\Hck_{G,[\Spd k/I_F]},\Lambda)\simeq  \Sat(\Hck_{G,[\Spd \cal{O}_C/I_F]},\Lambda)\simeq \Sat(\Hck_{G,[\Spd C/I_F]},\Lambda).    
\end{align}
\item 
The equivalence (\ref{eqn:HckBIFKunneth}) induces an equivalence
\begin{align}\label{eqn:SatLSKunneth}
\Sat(\Hck_{G,\Spd k},\Qlb)\otimes \LocSys([\Spd k/I_E],\Qlb)\xrightarrow{\simeq} \Sat(\Hck_{G,\Spd k}\times [\Spd k/I_E],\Qlb).    
\end{align}
\item 
There is a canonical monoidal natural isomorphism that makes the diagram
{\footnotesize
\begin{align}\label{eqn:diag:nonderSatSpdEbreve}
\xymatrix@C=40pt{
\Rep^{\rm{f.d.}}(\wh{G}\times I_E,\Qlb)\ar[r]^-{S_{\Spd \breve{E}}}\ar@{=}[dd]& \cal{D}^{\ULA}(\Hck_{G,\Spd \breve{E}},\Qlb)\ar[d]_{\simeq}^{(\ref{eqn:DULAequivIF})}\\
&\cal{D}^{\ULA}(\Hck_{G,[\Spd k/I_E]},\Qlb)\ar@{<-}[d]_{\simeq}^{(\ref{eqn:HckBIFKunneth})}\\
\Rep^{\rm{f.d.}}(\wh{G}\times I_E,\Qlb)\ar[r]^-{(\ref{eqn:SSpdkotimesid})}&\cal{D}^{\ULA}(\Hck_{G,\Spd k},\Qlb) \otimes \cal{D}_{\rm{lc}}([\Spd k/I_E],\Qlb)
}
\end{align}
}
commute.
\end{enumerate}
\end{lemm}
\begin{proof}
\begin{enumerate}
\item We can assume that $\Lambda$ is a torsion ring.
By characterization of ULA sheaves and Satake sheaves in \cite[Proposition VI.6.4, Proposition VI.7.7]{FS}, the first (resp. second) equivalence reduces to the claim that all fibers of an object in $\cal{D}_{\rm{lc}}([\Spd \cal{O}_C/I_E],\Lambda)$ are finite projective $\Lambda$-modules in degree 0 if and only if its pullback to $\cal{D}_{\rm{lc}}([\Spd k/I_E],\Lambda)$ (resp. $\cal{D}_{\rm{lc}}([\Spd C/I_E],\Lambda)$) satisfies the same condition.
This follows from (\ref{eqn:DlcSpdEbreveDRepIEeqiv}).
\item 
The category $\Rep^{\rm{f.d.}}(\wh{G}\times I_E,\Qlb)$ is the category of objects isomorphic to the direct sum of the form
\[
\bigoplus_{i=1}^n V_i\boxtimes W_i
\]
where $n\in \ZZ_{\geq 0}$, $V_i\in \Rep^{\rm{f.d.}}(\wh{G},\Qlb)$ and $W_i\in \Rep^{\rm{f.d.}}(I_E,\Qlb)$.
Therefore, $\Sat(\Hck_{G,\Spd \breve{E}})$ is the category of objects isomorphic to the direct sum of the form
\begin{align}\label{eqn:bigoplusSSpdEbreveformofobjectsinSat}
\bigoplus_{i=1}^n S_{\Spd \breve{E}}(V_i\boxtimes \Qlb)\star S_{\Spd \breve{E}}(\Qlb\boxtimes W_i).
\end{align}
By the construction of $S_{\Spd \breve{E}}$, it holds that $S_{\Spd \breve{E}}(\Qlb\boxtimes W_i)=i_{0,\Spd \breve{E}*}\cal{L}_{W_i}$.
The object $i_{0,\Spd \breve{E}*}\cal{L}_{W_i}\in \Sat(\Hck_{G,\Spd \breve{E}},\Qlb)$  corresponds to 
\[
i_{0,[\Spd k/I_E]*}\cal{L}_{W_i}=i_{0,[\Spd k/I_E]*}\Qlb \boxtimes \cal{L}_{W_i}\in \Sat(\Hck_{G,[\Spd k/I_E]},\Qlb)
\]
under (\ref{eqn:SatequivIF}).
Therefore, the object $S_{\Spd \breve{E}}(\Qlb\boxtimes W_i)$ corresponds to 
\[
i_{0,[\Spd k/I_E]*}\Qlb \boxtimes \cal{L}_{W_i} \in \cal{D}^{\ULA}(\Hck_{G,\Spd k},\Qlb) \otimes_{\Vect_{\Qlb}^{\perf}} \cal{D}_{\rm{lc}}([\Spd k/I_E],\Qlb) 
\]
under the composition of (\ref{eqn:SatequivIF}) and (\ref{eqn:HckBIFKunneth}).

Take any $\mu\in \bb{X}_*^+(T_{\ol{F}})$.
Let $V_{\mu}\in \Rep^{\rm{f.p.}}(\wh{G},\Qlb)$ denote the irreducible representation of $\wh{G}$ with highest weight $\mu$.
The intersection cohomology sheaf $IC_{\mu}\in \Sat(\Hck_{G,\Spd C},\Qlb)$ equipped with the $I_E$-equivariant structure induced by the trivial $I_E$-equivariant structure on $\ol{\QQ}_{\ell,\Hck_{G,\Spd C,\mu}}$ defines an object in $\Sat(\Hck_{G,\Spd \breve{E}},\Qlb)$.
By the construction of the non-derived Satake equivalence (\ref{eqn:nonderivedSpdbreveF}), which follows the normalization in \cite[\S VI.11]{FS}, this object in $\Sat(\Hck_{G,\Spd \breve{E}},\Qlb)$ corresponds to $V_{\mu}\boxtimes \Qlb\in \Rep^{\rm{f.p.}}(\wh{G}\times I_E,\Qlb)$ under $S_{\Spd \breve{E}}$.
Moreover, $IC_{\mu}\in \Sat(\Hck_{G,\Spd C},\Qlb)$ and $IC_{\mu,\Spd k}\in \Sat(\Hck_{G,\Spd k},\Qlb)$ correspond under (\ref{eqn:SatPervequiv}), and taking the $I_E$-equivariant structures into account, we can show that the above object in $\Sat(\Hck_{G,\Spd \breve{E}},\Qlb)$ and the object $IC_{\mu,\Spd k}\boxtimes \Qlb \in \Sat(\Hck_{G,[\Spd \breve{k}/I_E]},\Qlb)$ correspond under (\ref{eqn:SatequivIF}).
This implies that the object $S_{\Spd \breve{E}}(V_{\mu}\boxtimes \Qlb)$ corresponds to 
\[
IC_{\mu} \boxtimes \Qlb \in \cal{D}^{\ULA}(\Hck_{G,\Spd k},\Qlb) \otimes_{\Vect_{\Qlb}^{\perf}} \cal{D}_{\rm{lc}}([\Spd k/I_E],\Qlb) 
\]
under the composition of (\ref{eqn:SatequivIF}) and (\ref{eqn:HckBIFKunneth}).
Therefore, the object (\ref{eqn:bigoplusSSpdEbreveformofobjectsinSat}) corresponds to 
\[
\bigoplus_{i=1}^n (S_{\Spd k}^{-1}(V_i)\boxtimes \Qlb)\star (\Qlb \boxtimes \cal{L}_{W_i})=\bigoplus_{i=1}^n (S_{\Spd k}^{-1}(V_i)\boxtimes \cal{L}_{W_i}).
\]
The full subcategory of objects of this form agrees with the full subcategory of the statement.
\item 
By the same argument as in the proof of Lemma \ref{lemm:SatPervequiv}  (\ref{item:threenonderSatcompati}), there are canonical natural isomorphisms that make the squares
{\footnotesize
\begin{align}\label{eqn:diag:fiberfunctorsIFcompatible}
\xymatrix@C=45pt{
\Sat(\Hck_{G,\Spd \breve{E}},\Qlb)\ar[r]^-{F_{G,\Spd \breve{E}}}\ar@{<-}[d]^{\simeq}_{(\ref{eqn:SatequivIF})}&\LocSys(\Spd \breve{E},\Qlb)\ar@{<-}[d]^{\simeq}\ar[r]^-{\simeq}\ar@{}[rd]|{(\ref{eqn:DlcSpdEbreveDRepIEeqiv})}&\Rep^{\rm{f.d.}}(I_E,\Qlb)\ar@{=}[d]\\
\Sat(\Hck_{G,[\Spd \cal{O}_C/I_E]},\Qlb)\ar[r]_-{F_{G,[\Spd \cal{O}_C/I_E]}}\ar[d]^{\simeq}_{(\ref{eqn:SatequivIF})}&\LocSys([\Spd \cal{O}_C/I_E],\Qlb)\ar@{<-}[d]^{\simeq}\ar[r]^-{\simeq}\ar@{}[rd]|{(\ref{eqn:DlcSpdEbreveDRepIEeqiv})}&\Rep^{\rm{f.d.}}(I_E,\Qlb)\ar@{=}[d]\\
\Sat(\Hck_{G,[\Spd k/I_E]},\Qlb)\ar[r]^-{F_{G,[\Spd k/I_E]}}&\LocSys([\Spd k/I_E],\Qlb)\ar[r]^-{\simeq}&\Rep^{\rm{f.d.}}(I_E,\Qlb)
}
\end{align}
}
commute.
We define
\[
F_{G,\Spd k}\otimes \rm{\rm{id}}\colon \Sat(\Hck_{G,\Spd k},\Qlb)\otimes \LocSys([\Spd k/I_E],\Qlb)\to \LocSys([\Spd k/I_E],\Qlb)
\]
by
\[
\bigoplus_{i=1}^n A_i\boxtimes \cal{L}_i\mapsto \bigoplus_{i=1}^n F_{G,\Spd k}(A_i) \otimes \cal{L}_i,
\]
where $F_{G,\Spd k}(A_i)\in \LocSys(\Spd k,\Qlb)=\Vect_{\Qlb}^{\heartsuit, \rm{f.d.}}$ is regarded as a constant local system on $[\Spd k/I_E]$.
By the K\"{u}nneth formula, there is a natural isomorphism that makes the square
\begin{align}\label{eqn:diag:fiberfunctorsIFKunnethcompatible}
\vcenter{
\xymatrix@C=40pt{
\Sat(\Hck_{G,[\Spd k/I_E]},\Qlb)\ar[r]^-{F_{G,[\Spd k/I_E]}}\ar@{<-}[d]_{(\ref{eqn:SatLSKunneth})}
&\LocSys([\Spd k/I_E],\Qlb)
\ar@{=}[d]
\\
\Sat(\Hck_{G,\Spd k},\Qlb)\otimes \LocSys([\Spd k/I_E],\Qlb)\ar[r]_-{F_{G,\Spd k}\otimes \rm{\rm{id}}}
& \LocSys([\Spd k/I_E],\Qlb).
}    
}
\end{align}
By composing (\ref{eqn:diag:fiberfunctorsIFcompatible}) and (\ref{eqn:diag:fiberfunctorsIFKunnethcompatible}), we get a square
\begin{align}\label{eqn:diag:fiberfunctorcompaticomposed}
\vcenter{
\xymatrix{
\Sat(\Hck_{G,\Spd \breve{E}},\Qlb)
\ar[r]
\ar[d]^{\simeq}
&\Rep^{\rm{f.d.}}(I_E,\Qlb)
\ar@{=}[d]
\\
\Sat(\Hck_{G,\Spd k},\Qlb)\otimes \LocSys([\Spd k/I_E],\Qlb)
\ar[r]
&\Rep^{\rm{f.d.}}(I_E,\Qlb)
}
}
\end{align}
commutative up to a natural isomorphism.
By the proof of \ref{lemm:SatPervequiv} (\ref{item:threenonderSatcompati}), the composition of $F_{\Spd k}$ and the equivalence $\Sat(\Hck_{G,\Spd C})\simeq \Sat(\Hck_{G,\Spd k})$ in (\ref{eqn:SatPervequiv}) is isomorphic to $F_{\Spd C}$. We define a symmetric monoidal structure on $F_{\Spd k}$ as the pullback of that on $F_{\Spd C}$.
This defines the symmetric monoidal structure on $F_{G,\Spd k}\otimes \rm{id}$.
We claim that the natural isomorphism (\ref{eqn:diag:fiberfunctorcompaticomposed}) is monoidal with respect to these monoidal structures.

The diagram (\ref{eqn:diag:fiberfunctorsIFcompatible}) is compatible with the diagram
{\footnotesize
\begin{align*}
\xymatrix@C=45pt{
\Sat(\Hck_{G,\Spd C},\Qlb)\ar[r]^-{F_{G,\Spd C}}\ar@{<-}[d]^{\simeq}_{(\ref{eqn:SatPervequiv})}&\LocSys(\Spd C,\Qlb)\ar@{<-}[d]^{\simeq}\ar[r]^-{\simeq}&\Vect_{\Qlb}^{\heartsuit, \rm{f.d.}}\ar@{=}[d]\\
\Sat(\Hck_{G,\Spd \cal{O}_C]},\Qlb)\ar[r]_-{F_{G,\Spd \cal{O}_C}}\ar[d]^{\simeq}_{(\ref{eqn:SatPervequiv})}&\LocSys(\Spd \cal{O}_C,\Qlb)\ar@{<-}[d]^{\simeq}\ar[r]^-{\simeq}&\Vect_{\Qlb}^{\heartsuit, \rm{f.d.}}\ar@{=}[d]\\
\Sat(\Hck_{G,\Spd k},\Qlb)\ar[r]^-{F_{G,\Spd k}}&\LocSys(\Spd k,\Qlb)\ar[r]^-{\simeq}&\Vect_{\Qlb}^{\heartsuit, \rm{f.d.}}
}
\end{align*}
}
and the square (\ref{eqn:diag:fiberfunctorsIFKunnethcompatible}) is compatible with the square
\[
\xymatrix@C=40pt{
\Sat(\Hck_{G,\Spd k},\Qlb)\ar[r]^-{F_{G,\Spd k}}\ar@{=}[d]
&\LocSys(\Spd k,\Qlb)
\ar@{=}[d]
\\
\Sat(\Hck_{G,\Spd k},\Qlb)\ar[r]_-{F_{G,\Spd k}}
& \LocSys(\Spd k,\Qlb).
}
\]
under various pullback functors along $\Spd C\to \Spd \breve{E}$ and $\Spd k\to [\Spd k/I_E]$, and the forgetful functor $\Rep^{\rm{f.p.}}(I_E)\to \Vect_{\Qlb}^{\heartsuit, \rm{f.d.}}$.
Therefore, the square (\ref{eqn:diag:fiberfunctorcompaticomposed}) is compatible with a natural isomorphism that makes the square of the form 
\begin{align}\label{eqn:diag:fiberfunctorcompaticomposednonequiv}
\xymatrix{
\Sat(\Hck_{G,\Spd C},\Qlb)
\ar[r]
\ar[d]^{\simeq}
&\Vect_{\Qlb}^{\heartsuit, \rm{f.d.}}
\ar@{=}[d]
\\
\Sat(\Hck_{G,\Spd k},\Qlb)
\ar[r]
&\Vect_{\Qlb}^{\heartsuit, \rm{f.d.}}
}    
\end{align}
commute.
Since this natural isomorphism is monoidal by definition, and since the forgetful functor $\Rep^{\rm{f.p.}}(I_E)\to \Vect_{\Qlb}^{\heartsuit, \rm{f.d.}}$ is faithful.
It implies that the natural isomorphism (\ref{eqn:diag:fiberfunctorcompaticomposed}) is monoidal.

Since the Tannakian group of the fiber functor $F_{G,\Spd C}$ is isomorphic to $\wh{G}$, so is that of $F_{G,\Spd k}$.
Hence, the Tannakian group of $F_{G,\Spd k}\otimes \rm{id}$ is isomorphic to $\wh{G}$ with the trivial $I_E$-action.
Therefore, the square (\ref{eqn:diag:fiberfunctorcompaticomposed}) induces a square
\begin{align}\label{eqn:diag:geomSatIFequivuptopsi}
\vcenter{
\xymatrix{
\Sat(\Hck_{G,\Spd \breve{E}},\Qlb)
\ar[r]^-{\simeq}
\ar[d]^{\simeq}
&\Rep^{\rm{f.d.}}(\wh{G}\times I_E,\Qlb)
\ar[d]^{(-)^{\psi}}
\\
\Sat(\Hck_{G,\Spd k},\Qlb)\otimes \LocSys([\Spd k/I_E],\Qlb)
\ar[r]^-{\simeq}
&\Rep^{\rm{f.d.}}(\wh{G}\times I_E,\Qlb)
}
}
\end{align}
commutative up to a monoidal natural isomorphism, where $(-)^{\psi}$ is a functor twisting a $\wh{G}$-action via an isomorphism $\psi\colon \wh{G}\xrightarrow{\simeq} \wh{G}$.
The square (\ref{eqn:diag:fiberfunctorcompaticomposednonequiv}) induces the square 
\[
\xymatrix{
\Sat(\Hck_{G,\Spd C},\Qlb)
\ar[r]^-{\simeq}
\ar[d]^{\simeq}
&\Rep^{\rm{f.d.}}(\wh{G},\Qlb)
\ar@{=}[d]
\\
\Sat(\Hck_{G,\Spd k},\Qlb)
\ar[r]^-{\simeq}
&\Rep^{\rm{f.d.}}(\wh{G},\Qlb).
}
\]
Since this is compatible with (\ref{eqn:diag:geomSatIFequivuptopsi}), it follows that $\psi=\rm{id}$.
Therefore, the monoidal natural isomorphism (\ref{eqn:diag:geomSatIFequivuptopsi}) induces the desired monoidal natural isomorphism.
\end{enumerate}
\end{proof}
\subsection{Semisimple and simply-connected case}\label{ssc:semisimplesimplyconnectedEbreve}
\subsubsection{Cochain functor, its full faithfulness and Weil action}\label{ssc:CohWeil}
Let $q_{I_E}\colon \Hck_{G,\Spd C}\to \Hck_{G,\Spd \breve{E}}$ denote the natural projection.
For any object $A\in \cal{D}^{\ULA}(\Hck_{G,\Spd \breve{E}},\Qlb)$, the dg-algebra $R\Gamma(\Hck_{G,\Spd C},q_{I_E}^*A)$ has both the structure of an object in $\DDRep(I_E,\Qlb)$ (by the same argument as \S \ref{ssc:WeilonCohHck}) and the structure of an $R\Gamma(\Hck_{G,\Spd C},\Qlb)$-module that are compatible.
By Proposition \ref{prop:RHckequiv}, we obtain a functor
\[
R\Gamma_{\Spd \breve{E}} \colon \cal{D}^{\ULA}(\Hck_{G,\Spd \breve{E}},\Qlb)\to \cal{O}_{\bb{T}(\wh{\fk{t}}^*/W)}^{\shear}\lmod(\DDRep(I_E)).
\]
First, we show the following:
\begin{lemm}\label{lemm:RGammaIFff}
The functor $R\Gamma_{\Spd \breve{E}}$ is fully faithful.
\end{lemm}
\begin{proof}
By Lemma \ref{lemm:DULAequivIF}, the claim reduces to the case over $[\Spd k/I_E]$ instead of $\Spd \breve{E}=[\Spd C/I_E]$.
That is, it suffices to show the full faithfulness of 
\[
R\Gamma_{[\Spd k/I_E]}\colon \cal{D}^{\ULA}(\Hck_{G,[\Spd k/I_E]})\to \cal{O}_{\bb{T}(\wh{\fk{t}}^*/W)}^{\shear}\lmod(\DDRep(I_E))
\]
defined in the same way as $R\Gamma_{\Spd \breve{E}}$.
Since the $I_E$-action on $\cal{O}_{\bb{T}(\wh{\fk{t}}^*/W)}$ is trivial, there is a canonical equivalence
\begin{multline}
\cal{O}_{\bb{T}(\wh{\fk{t}}^*/W)}^{\shear}\lmod\otimes \DDRep(I_E)  
=\cal{O}_{\bb{T}(\wh{\fk{t}}^*/W)}^{\shear}\lmod(\Vect_{\Qlb})\otimes \DDRep(I_E)  \\
\xrightarrow{\sim} \cal{O}_{\bb{T}(\wh{\fk{t}}^*/W)}^{\shear}\lmod(\DDRep(I_E)).\label{eqn:OTtWmodsplit}
\end{multline}
The composition
\begin{align*}
&\cal{D}^{\ULA}(\Hck_{G,\Spd k}) \otimes \DDRep^{\perf}(I_E)\simeq \cal{D}^{\ULA}(\Hck_{G,\Spd k}) \otimes \cal{D}_{\rm{lc}}([\Spd k/I_E])\\
&\xrightarrow{-\boxtimes-} \cal{D}^{\ULA}(\Hck_{G,\Spd k}\times [\Spd k/I_E],\Lambda)=\cal{D}^{\ULA}(\Hck_{G,[\Spd k/I_E]},\Lambda)\\
&\xrightarrow{R\Gamma_{[\Spd k/I_E]}} \cal{O}_{\bb{T}(\wh{\fk{t}}^*/W)}^{\shear}\lmod(\DDRep(I_E))\simeq \cal{O}_{\bb{T}(\wh{\fk{t}}^*/W)}^{\shear}\lmod\otimes \DDRep(I_E)
\end{align*}
is isomorphic to 
\[
R\Gamma_{\Spd k}\otimes \rm{incl}\colon \cal{D}^{\ULA}(\Hck_{G,\Spd k}) \otimes \DDRep^{\perf}(I_E))\to \cal{O}_{\bb{T}(\wh{\fk{t}}^*/W)}^{\shear}\lmod\otimes \DDRep(I_E)
\]
where the functor $R\Gamma_{\Spd k}$ is defined in the same way as $R\Gamma$, and $\rm{incl}$ is the inclusion.
Since $R\Gamma_{\Spd k}$ is fully faithful by Lemma \ref{lemm:RGammafullyfaithful} replacing $\Spd C$ with $\Spd k$, we get the claim.
\end{proof}
We end this section by giving a lemma on the Weil action.
The $W_{F}$-action on $\bb{T}(\wh{\fk{t}}^*/W)$ and the $W_{F}$-action on $I_E$ by an inner automorphism make $\cal{O}_{\bb{T}(\wh{\fk{t}}^*/W)}^{\shear}\lmod(\DDRep(I_E))$ an object in $\Catinf{\Qlb}^{mon,\ BW_{F}^{\rm{disc}}}$.
Since the $I_E$-action on $\bb{T}(\wh{\fk{t}}^*/W)$ is trivial and the action of any $w\in I_E$ on $\DDRep(I_E)$ is naturally isomorphic to the identity, we can regard the target as an object in $\Catinf{\Qlb}^{mon,\ B(W_{F}/I_E)}$.
\begin{lemm}\label{lemm:WonRGammaIF}
The functor 
\[
R\Gamma_{\Spd \breve{E}}\colon \cal{D}^{\ULA}(\Hck_{G,\Spd \breve{E}},\Qlb)\to \cal{O}_{\bb{T}(\wh{\fk{t}}^*/W)}^{\shear}\lmod(\DDRep(I_E,\Qlb))
\]
naturally has the structure of a functor in $\Catinf{\Qlb}^{mon,\ B(W_{F}/I_E)}$.
\end{lemm}
\begin{proof}
This follows from the same argument as Lemma \ref{lemm:monoidalonRGamma} and from the fact that the isomorphism (\ref{eqn:monoidalonRGamma}) is $W_{F}$-equivalent.
\end{proof}
\subsubsection{Kostant functor and Weil action}\label{ssc:KosWeil}
We also need a result on the Kostant functor
\[
\kappa^{\shear}\colon \Coh([\wh{\fk{g}}^*/\wh{G}])^{\shear}\to \cal{O}_{\bb{T}(\wh{\fk{t}}^*/W)}^{\shear}\lmod.
\]
Let $\rm{act}^{\rm{geom}}_{\wh{\fk{g}}^*(1)}(w)\colon \wh{\fk{g}}^*\to \wh{\fk{g}}^*$ be the composition of $\rm{act}^{\rm{geom}}_{\wh{\fk{g}}^*}(w)$ (defined in \S \ref{ssc:dualgp}) and the multiplication by $q^{\deg w}$ (Tate twist).
That is, the $W_{F}$-action on $\cal{O}_{\wh{\fk{g}}^*}$ induced by $\rm{act}^{\rm{geom}}_{\wh{\fk{g}}^*(1)}$ agrees with the $W_{F}$-action on $\Sym(\wh{\fk{g}}(-1)[-2])$ written in Theorem \ref{thm:main} up to the shift $[-2]$.
The actions $\rm{act}^{\rm{geom}}_{\wh{G}}$ and $\rm{act}^{\rm{geom}}_{\wh{\fk{g}}^*(1)}$ induce a $W_{F}$-action $\rm{act}_{\wh{\fk{g}}^*(1)/\wh{G}}$ on $[\wh{\fk{g}}^*/\wh{G}]$, where $I_E$ acts trivially.
This action makes the categories $\QCoh([\wh{\fk{g}}^*/\wh{G}])$ and $\Coh([\wh{\fk{g}}^*/\wh{G}])$ objects of $\Catinf{\Qlb}^{mon,\ B(W_{F}/I_E)}$.
Explicitly, under the identification 
\[
\QCoh([\wh{\fk{g}}^*/\wh{G}])\simeq\cal{O}_{\wh{\fk{g}}^*}\lmod(\cal{D}(\Rep(\wh{G}))),
\]
the action of $w\in W_{F}$ agrees with the functor $(-)^w$ twisting $\wh{G}$-action by $\rm{act}^{\rm{geom}}_{\wh{G}}(w)$ and $\cal{O}_{\wh{\fk{g}}^*}$-action by the homomorphism $\cal{O}_{\wh{\fk{g}}^*}\to \cal{O}_{\wh{\fk{g}}^*}$ induced by $\rm{act}^{\rm{geom}}_{\wh{\fk{g}}^*(1)}(w)\colon \wh{\fk{g}}^*\to \wh{\fk{g}}^*$.
Therefore, the categories $\QCoh([\wh{\fk{g}}^*/\wh{G}])^{\shear}$ and $\Coh([\wh{\fk{g}}^*/\wh{G}])^{\shear}$ are also objects of $\Catinf{\Qlb}^{mon,\ B(W_{F}/I_E)}$.

Recall that, as we discussed in \S \ref{sssc:derSatSpdCsssc:Sqc}, the functor $\kappa^{\shear}$ is constructed from
\[
\kappa\colon \Coh^{\bb{G}_m}([\wh{\fk{g}}^*/\wh{G}])^{\heartsuit}\to \QCoh^{\bb{G}_m}(\bb{T}(\wh{\fk{t}}^*/W))^{\heartsuit}.
\]
The $W_{F}/I_E$-action on $\Coh([\wh{\fk{g}}^*/\wh{G}])^{\shear}$ (resp. $\cal{O}_{\bb{T}(\wh{\fk{t}}^*/W)}^{\shear}\lmod$) is equivalent to that constructed from the $W_{F}/I_E$-action on $\Coh^{\bb{G}_m}([\wh{\fk{g}}^*/\wh{G}])^{\heartsuit}$ (resp. $\QCoh^{\bb{G}_m}(\bb{T}(\wh{\fk{t}}^*/W))^{\heartsuit}$).

We have the following result:
\begin{lemm}\label{lemm:Wonkappa}
For $w\in W_{F}$, there exists a monoidal natural isomorphism that makes the square
\begin{align}\label{eqn:kappaWeil}
\vcenter{
\xymatrix{
\Coh^{\bb{G}_m}([\wh{\fk{g}}^*/\wh{G}])^{\heartsuit}\ar[r]^-{\kappa}\ar[d]_{w^*}&\QCoh^{\bb{G}_m}(\bb{T}(\wh{\fk{t}}^*/W))^{\heartsuit}\ar[d]^{w^*}\\
\Coh^{\bb{G}_m}([\wh{\fk{g}}^*/\wh{G}])^{\heartsuit}\ar[r]^-{\kappa}&\QCoh^{\bb{G}_m}(\bb{T}(\wh{\fk{t}}^*/W))^{\heartsuit}
}
}
\end{align}
commute.
\end{lemm}
\begin{proof}
The product $\rm{act}^{\rm{geom}}_{\wh{\fk{g}}^*(1)}\times \rm{act}^{\rm{geom}}_{\wh{\fk{g}}}$ defines a $W_{F}$-action on $\wh{\fk{g}}^*\times \wh{\fk{g}}$.
Since the multiplication by $q$ does not change a centralizer, this induces a $W_{F}$-action on $\rm{Tot}(\fk{z})\subset (\wh{\fk{g}}^*)^{\rm{reg}}\times \wh{\fk{g}}$.
For $w\in W_{F}$, let
\[
w^*\colon \Coh(\rm{Tot}(\fk{z}))\to \Coh(\rm{Tot}(\fk{z})).
\]
be the pullback along the action of $w$.
For $\cal{F}\in \Coh^{\bb{G}_m}([\wh{\fk{g}}^*/\wh{G}])^{\heartsuit}$, let $\wt{\cal{F}}\in \Coh^{\bb{G}_m}(\rm{Tot}(\fk{z}))^{\heartsuit}$ be a coherent sheaf defined as written in the definition of $\kappa$.
Then we have
\begin{align}\label{eqn:w*tilde}
\wt{w^*\cal{F}}\cong w^*\wt{\cal{F}}.
\end{align}
To show this, it suffices to show that for any point $\xi\colon \Spec K\to (\wh{\fk{g}}^*)^{\rm{reg}}$, the action of $Z_{\wh{G}_K}(w(\xi))$ on $\cal{F}_{w(\xi)}$ (induced by the $\wh{G}$-equivariant structure of $\cal{F}$) agrees with the action of $Z_{\wh{G}_K}(\xi)$ on $(w^*\cal{F})_{\xi}$ (induced by the $\wh{G}$-equivariant structure of $w^*\cal{F}$), under the identification $(\rm{act}^{\rm{geom}}_{\wh{G}}(w))_K\ \colon\ Z_{\wh{G}_K}(\xi)\overset{\sim}{\to} Z_{\wh{G}_K}(w(\xi))$ and $\cal{F}_{w(\xi)}\cong (w^*\cal{F})_{\xi}$.
This is straightforward.
Now the existence of the desired natural isomorphism follows easily from the fact that the natural projection
\[
\rm{Tot}(\fk{z})\to \bb{T}(\wh{\fk{t}}^*/W)
\]
is $W_{F}$-equivariant with respect to the $W_{F}$-actions we are considering.
The monoidality of the isomorphism follows from a straightforward argument, using the fact that there is a commutative diagram of the form
\[
\xymatrix{
w^*((\cal{F}_1\otimes_{\cal{O}_{\wh{\fk{g}}^*}}\cal{F}_2)^{\sim})\ar[r]^-{\sim}_-{(\ref{eqn:convtilde})}
\ar[d]_{\rotatebox{-90}{$\sim$}}^{(\ref{eqn:w*tilde})}
&w^*(\wt{\cal{F}_1}\star_{\rm{Tot}} \wt{\cal{F}_2})\ar[d]^{\rotatebox{-90}{$\sim$}}\\
(w^*(\cal{F}_1\otimes_{\cal{O}_{\wh{\fk{g}}^*}}\cal{F}_2))^{\sim}
\ar[d]_{\rotatebox{-90}{$\sim$}}
&w^*\wt{\cal{F}_1}\star_{\rm{Tot}} w^*\wt{\cal{F}_2}\ar[d]^{\rotatebox{-90}{$\sim$}}_{(\ref{eqn:w*tilde})} \\
(w^*\cal{F}_1\otimes_{\cal{O}_{\wh{\fk{g}}^*}}w^*\cal{F}_2)^{\sim}
\ar[r]^-{\sim}_-{(\ref{eqn:convtilde})}
&\wt{w^*\cal{F}_1}\star_{\rm{Tot}} \wt{w^*\cal{F}_2}.
}
\]
\end{proof}
From this and the compatibilities of various natural isomorphisms, the functor
\[
\kappa\colon \Coh^{\bb{G}_m}([\wh{\fk{g}}^*/\wh{G}])^{\heartsuit}\to \QCoh^{\bb{G}_m}(\bb{T}(\wh{\fk{t}}^*/W))^{\heartsuit}
\]
has the structure of a functor in $\Catinf{\Qlb}^{mon,\ BW_{F}^{\rm{disc}}}$.
Moreover, since the natural morphism (\ref{eqn:kappaWeil})  for $w\in I_E$ is trivial, it also has the structure of a functor in $\Catinf{\Qlb}^{mon, B(W_{F}/I_E)}$.
Therefore, by the construction in \S \ref{sssc:derSatSpdCsssc:Sqc}, we get the following:
\begin{lemm}\label{lemm:Wonkappashear}
The functor
\[
\kappa^{\shear}\colon \Coh([\wh{\fk{g}}^*/\wh{G}])^{\shear}\to \cal{O}_{\bb{T}(\wh{\fk{t}}^*/W)}^{\shear}\lmod
\]
has the canonical structure of a functor in $\Catinf{\Qlb}^{mon,\ B(W_{F}/I_E)}$.
\end{lemm}
Also, we define
\[
\kappa_{I_E}^{\shear}\colon \rm{Sym}(\wh{\fk{g}}[-2])\lmod^{\perf}(\DDRep(\wh{G}\times I_E)) \to \cal{O}_{\bb{T}(\wh{\fk{t}}^*/W)}^{\shear}\lmod(\DDRep(I_E))
\]
so that under the natural equivalence
\begin{align*}
&\Coh([\wh{\fk{g}}^*/\wh{G}])^{\shear}\otimes \DDRep^{\perf}(I_E)\\
&=\rm{Sym}(\wh{\fk{g}}[-2])\lmod^{\perf}(\cal{D}(\Rep(\wh{G}))\otimes \DDRep^{\perf}(I_E)\\
&\overset{\sim}{\to} \rm{Sym}(\wh{\fk{g}}[-2])\lmod^{\perf}(\DDRep(\wh{G}\times I_E))
\end{align*}
and the equivalence (\ref{eqn:OTtWmodsplit}), the functor $\kappa_{I_E}$ agrees with the tensor product of $\kappa$ and the inclusion $\DDRep^{\perf}(I_E)\subset \DDRep(I_E)$.
Since we can regard $\DDRep^{\perf}(I_E))$ as an object in $\Catinf{\Qlb}^{B(W_{F}/I_E)}$, and since $\kappa^{\shear}$ is fully faithful as written in \S \ref{sssc:derSatSpdCsssc:Sqc}, we have
\begin{lemm}\label{lemm:WonkappashearIF}
The functor
\[
\kappa_{I_E}^{\shear}\colon \rm{Sym}(\wh{\fk{g}}[-2])\lmod^{\perf}(\DDRep(\wh{G}\times I_E)) \to \cal{O}_{\bb{T}(\wh{\fk{t}}^*/W)}^{\shear}\lmod(\DDRep(I_E))
\]
has the canonical structure of a functor in $\Catinf{\Qlb}^{mon,\ B(W_{F}/I_E)}$.
Moreover, it is fully faithful.
\end{lemm}
By Lemma \ref{lemm:RGammaIFff}, Lemma \ref{lemm:WonRGammaIF} and Lemma \ref{lemm:WonkappashearIF}, we have the following:
\begin{thm}\label{thm:mainSpdEbrevessscbody}
There is an essentially unique equivalence
\[
\cal{S}^{der}_{\Spd \breve{E}}\colon \rm{Sym}(\wh{\fk{g}}[-2])\lmod^{\perf}(\DDRep(\wh{G}\times I_E,\Qlb))\xrightarrow{\sim} \cal{D}^{\ULA}(\Hck_{G,\Spd \breve{E}},\Qlb).
\]
in $\Catinf{\Qlb}^{mon,\ B(W_{F}/I_E)}$ such that there are natural isomorphisms
\begin{align}\label{eqn:RGammaSderkappaIE}
R\Gamma_{I_E}\circ \cal{S}^{der}_{\Spd \breve{E}}\simeq \kappa_{I_E}^{\shear}
\end{align}
in $\Catinf{\Qlb}^{mon,\ B(W_{F}/I_E)}$ and 
\begin{align}\label{eqn:SSpdFSqcFFrIF}
S_{\Spd \breve{E}}\simeq \cal{S}^{der}_{\Spd \breve{E}}\circ fr|_{\Rep^{\rm{f.p.}}(\wh{G}\times I_E,\Qlb)}
\end{align}
in $\Catinf{\Qlb}^{mon}$.
\end{thm}
We will show in the next section that (\ref{eqn:SSpdFSqcFFrIF}) can be canonically lifted to a morphism in $\Catinf{\Qlb}^{mon,\ B(W_{F}/I_E)}$.
\begin{proof}
Consider the non-equivariant version of the statement:
\begin{quote}
There is an essentially unique equivalence $\cal{S}^{der}_{\Spd \breve{E}}$ in $\Catinf{\Qlb}^{mon}$ such that there are natural isomorphism (\ref{eqn:RGammaSderkappaIE}) and (\ref{eqn:SSpdFSqcFFrIF}) in $\Catinf{\Qlb}^{mon}$.
\end{quote}
If this version is proved, it follows from the full faithfulness of $R\Gamma_{I_E}$ and $\kappa_{I_E}^{\shear}$ that there is an essentially unique $W_E/I_F$-equivarinat structure on $\cal{S}^{der}_{\Spd \breve{E}}$ that makes (\ref{eqn:RGammaSderkappaIE}) equivariant.
Therefore, the theorem follows.

It remains to show the non-equivariant version.
Essential uniqueness follows from the full faithfulness of $R\Gamma_{I_E}$ and $\kappa_{I_E}^{\shear}$.
By Lemma \ref{lemm:nonderSatSpdEbreve} and the arguments in \S \ref{ssc:CohWeil} and \S \ref{ssc:KosWeil}, it suffices to show that there is a functor
\begin{multline*}
\cal{S}^{der}_{\Spd k, I_E}\colon \rm{Sym}(\wh{\fk{g}}[-2])\lmod^{fr}(\cal{D}(\Rep(\wh{G})))\otimes \DDRep^{\perf}(I_E)\\
\to \cal{D}^{\ULA}(\Hck_{G,\Spd k})\otimes \DDRep^{\perf}(I_E)    
\end{multline*}
such that there are natural isomorphisms
\[
(R\Gamma_{\Spd k}\otimes \rm{id}_{\DDRep^{\perf}(I_E)})\circ \cal{S}^{der}_{\Spd k, I_E}\simeq \kappa\otimes \rm{id}_{\DDRep^{\perf}(I_E)}.
\]
and
\[
S_{\Spd k}\otimes \rm{id}_{\Rep^{\rm{f.p.}}(I_E)}\simeq \cal{S}^{der}_{\Spd k, I_E}\circ (fr|_{\Rep^{\rm{f.p.}}(\wh{G})}\otimes \rm{id}_{\Rep^{\rm{f.p.}}(I_E)})
\]
in $\Catinf{\Qlb}^{mon}$.
Here, $S_{\Spd k}\otimes \rm{id}_{\Rep^{\rm{f.p.}}(I_E)}$ is a functor of the form
\[
\bigoplus_{i=1}^n V_i\boxtimes W_i\mapsto \bigoplus_{i=1}^n S_{\Spd k}(V_i)\boxtimes W_i,
\]
with $n\in \bb{Z}_{\geq 0}$, $V_i\in \Rep^{\rm{f.p.}}(\wh{G})$, and $W_i\in \Rep^{\rm{f.p.}}(I_E)$.
We define the functor $fr|_{\Rep^{\rm{f.p.}}(\wh{G})}\otimes \rm{id}_{\Rep^{\rm{f.p.}}(I_E)}$ similarly.

However, by Proposition \ref{prop:SqcSpdC} and Remark \ref{rmk:SqcSpdk}, 
\[
\cal{S}^{der}_{\Spd k, I_E}= \cal{S}^{der}_{\Spd k}\otimes \rm{id}_{\DDRep^{\perf}(I_E)}
\]
satisfies the condition.
\end{proof}
\subsubsection{Compatibility with non-derived Satake equivalence}
Since the map $[\wh{\fk{g}}^*/\wh{G}]\to [*/\wh{G}]$ is $W_{F}/I_E$-equivariant, the functor
\[
fr|_{\Rep^{\rm{f.d.}}(\wh{G})}:\Rep^{\rm{f.d.}}(\wh{G})\subset \cal{D}^b(\Rep^{\rm{f.d.}}(\wh{G}))= \Coh([*/\wh{G}])\to \Coh([\wh{\fk{g}}^*/\wh{G}])^{\shear}
\]
has the structure of a functor in $\Catinf{\Qlb}^{mon,\ BW_{F}^{\rm{disc}}}$ and $\Catinf{\Qlb}^{mon,\ B(W_{F}/I_E)}$.
Explicitly, for an object $V\in \Rep^{\rm{f.d.}}(\wh{G})=\cal{D}^b(\Rep^{\rm{f.d.}}(\wh{G}))^{\heartsuit}$ and under the identification $\Coh([\wh{\fk{g}}^*/\wh{G}])^{\shear}\simeq \rm{Sym}(\wh{\fk{g}}[-2])\lmod^{\perf}(\cal{D}(\Rep(\wh{G})))$, the $W_{F}$-equivariant structure of $fr$ is the isomorphism 
\[
\rm{Sym}(\wh{\fk{g}}[-2])\otimes (V^w)\overset{\sim}{\to} (\rm{Sym}(\wh{\fk{g}}[-2])\otimes V)^w, a\otimes v\mapsto w(a)\otimes v.
\]
By the same argument as the case of $\kappa_{I_E}$, there is a canonical functor
\begin{align*}
fr|_{\Rep^{\rm{f.d.}}(\wh{G}\times I_E)}\colon \Rep^{\rm{f.d.}}(\wh{G}\times I_E)&\to \rm{Sym}(\wh{\fk{g}}[-2])\lmod^{\perf}(\DDRep(\wh{G}\times I_E)),\\
V&\mapsto \rm{Sym}(\wh{\fk{g}}[-2])\otimes V
\end{align*}
in $\Catinf{\Qlb}^{mon,\ B(W_{F}/I_E)}$.
In this section, we will prove the following proposition
\begin{prop}\label{prop:compatiSatIF}
The natural isomorphism (\ref{eqn:SSpdFSqcFFrIF}) has the structure of a morphism in $\Catinf{\Qlb}^{mon,\ B(W_{F}/I_E)}$ with respect to $(W_{F}/I_E)$-equivariant structures defined above.
\end{prop}
The proof is similar to the argument in \S \ref{sssc:compatimonoidalnonderived}.
We have defined the functors
\begin{align*}
&\Rep^{\rm{f.d.}}(\wh{G}\times I_E)\overset{fr}{\longrightarrow}
\rm{Sym}(\wh{\fk{g}}[-2])\lmod^{fr}(\cal{D}\Rep(\wh{G}\times I_E))\\
&\overset{\cal{S}^{der}_{\Spd \breve{E}}}{\longrightarrow}\cal{D}^{\ULA}(\Hck_{G,\Spd \breve{E}})\overset{R\Gamma}{\longrightarrow}\cal{O}_{\bb{T}(\wh{\fk{t}}^*/W)}^{\shear}\lmod(\DDRep(I_E)).
\end{align*}
Let $\cal{C}'_1,\cal{C}'_2,\cal{C}'_3$ be the essential image of $\Rep^{\rm{f.d.}}(\wh{G}\times I_E)$ in the second, third, and fourth categories above, respectively.
As in \S \ref{sssc:compatimonoidalnonderived}, it suffices to show that the natural isomorphism 
\begin{align}\label{eqn:h(SSpdCSqcFr)IF}
(\Rep^{\rm{f.d.}}(\wh{G}\times I_E)\overset{fr}{\longrightarrow}\cal{C}'_1\overset{\cal{S}^{der}_{\Spd \breve{E}}}{\longrightarrow}\cal{C}'_2)\simeq (\Rep^{\rm{f.d.}}(\wh{G}\times I_E)\overset{S_{\Spd \breve{E}}}{\longrightarrow}\cal{C}'_2) 
\end{align}
of functors between 1-categories is monoidal and $(W_{F}/I_E)$-equivariant.
Consider the monoidal functor
\[
H^*\colon \cal{C}_3'\subset \cal{O}_{\bb{T}(\wh{\fk{t}}^*/W)}^{\shear}\lmod(\DDRep(I_E))\to \cal{O}_{\bb{T}(\wh{\fk{t}}^*/W)}\lmod(\DDRep(I_E))^{\heartsuit},\ \cal{F}\mapsto \bigoplus_{i}H^i(\cal{F})
\]
of 1-categories.
Let $d_{I_E}$ be the composition
\begin{align*}
d_{I_E}\colon \cal{C}_3'&\overset{H^*}{\to} \cal{O}_{\bb{T}(\wh{\fk{t}}^*/W)}\lmod(\DDRep(I_E))^{\heartsuit}\overset{\Qlb\otimes_{\cal{O}_{\wh{\fk{t}}^*/W}}-}{\longrightarrow}\Rep(I_E).    
\end{align*}
Then we can prove the following lemma:
\begin{lemm}\label{lemm:H*dRGammaH*GrmonoidalIF}
There is a $(W_{F}/I_E)$-equivariant, monoidal natural isomorphism between the composition
\[
\Sat(\Hck_{G,\Spd \breve{E}})=\cal{C}'_2\overset{R\Gamma}{\longrightarrow} \cal{C}'_3\overset{d_{I_E}}{\longrightarrow}\Rep^{\rm{f.d.}}(I_E)
\]
and 
\begin{align*}
F_{\Spd \breve{E}}:=H^*(\Gr_{G,\Spd C},q_{I_E}^*-)\colon \Sat(\Hck_{G,\Spd \breve{E}})&\to \Rep^{\rm{f.d.}}(I_E)\\
\cal{F}&\mapsto \bigoplus_i H^i(\Gr_{G,\Spd C},q_{I_E}^*\cal{F}).
\end{align*}
\end{lemm}
\begin{proof}
By Lemma \ref{lemm:H*dRGammaH*Grmonoidal}, it suffices to show that the following square is commutative for any $\cal{F}\in \Sat(\Hck_{G,\Spd C})$ and $w\in W_{F}$, which is obvious:
\begin{align*}
\xymatrix{
H^*(\Spd C) \otimes_{H^*([\Spd C/L^+_{\Spd C}G])} H^*(\Hck_{G,\Spd C},\cal{F})\ar[r]^-{(\ref{eqn:H*deeq})}\ar[d]_{\rm{id} \otimes w^*}&  H^*(\Gr_{G,\Spd C},\cal{F})\ar[d]^{w^*}\\
H^*(\Spd C) \otimes_{H^*([\Spd C/L^+_{\Spd C}G])} H^*(\Hck_{G,\Spd C},w^*\cal{F})\ar[r]^-{(\ref{eqn:H*deeq})}&  H^*(\Gr_{G,\Spd C},w^*\cal{F}).
}
\end{align*}
\end{proof}
\begin{proof}[Proof of Proposition \ref{prop:compatiSatIF}]
Consider the isomorphism (\ref{eqn:h(SSpdCSqcFr)}) composed with $F_{\Spd \breve{E}}$:
\begin{align}
&(\Rep^{\rm{f.d.}}(\wh{G}\times I_E)\overset{fr}{\longrightarrow}\cal{C}'_1\overset{\cal{S}^{der}_{\Spd \breve{E}}}{\longrightarrow}\Sat(\Hck_{G,\Spd \breve{E}})\overset{F_{\Spd \breve{E}}}{\longrightarrow}\Rep^{\rm{f.d.}}(I_E))\notag\\
&\simeq (\Rep^{\rm{f.d.}}(\wh{G}\times I_E)\overset{S_{\Spd \breve{E}}}{\longrightarrow}\Sat(\Hck_{G,\Spd \breve{E}})\overset{F_{\Spd \breve{E}}}{\longrightarrow}\Rep^{\rm{f.d.}}(I_E))\label{eqn:composedwithH*GrGIF}
\end{align}
We claim that this isomorphism agrees with the composition of the form
\begin{align}
&(\Rep^{\rm{f.d.}}(\wh{G}\times I_E)\overset{fr}{\longrightarrow}\cal{C}'_1\overset{\cal{S}^{der}_{\Spd \breve{E}}}{\longrightarrow}\Sat(\Hck_{G,\Spd \breve{E}})\overset{F_{\Spd \breve{E}}}{\longrightarrow}\Rep^{\rm{f.d.}}(I_E))\notag\\
&\simeq (\Rep^{\rm{f.d.}}(\wh{G}\times I_E)\overset{fr}{\longrightarrow}\cal{C}'_1\overset{\cal{S}^{der}_{\Spd \breve{E}}}{\longrightarrow}\Sat(\Hck_{G,\Spd \breve{E}})\overset{R\Gamma_{\Spd \breve{E}}}{\longrightarrow}\cal{C}'_3\overset{d_{I_E}}{\longrightarrow}\Rep^{\rm{f.d.}}(I_E))\notag\\
&\simeq (\Rep^{\rm{f.d.}}(\wh{G}\times I_E)\overset{fr}{\longrightarrow}\cal{C}'_1\overset{\kappa_{I_E}}{\longrightarrow}\cal{C}'_3\overset{d_{I_E}}{\longrightarrow}\Rep^{\rm{f.d.}}(I_E))\notag\\
&\simeq (\Rep^{\rm{f.d.}}(\wh{G}\times I_E)\overset{\rm{forget}}{\longrightarrow}\Rep^{\rm{f.d.}}(I_E)))\notag\\
&\simeq (\Rep^{\rm{f.d.}}(\wh{G}\times I_E)\overset{S_{\Spd \breve{E}}}{\longrightarrow}\Sat(\Hck_{G,\Spd \breve{E}})\overset{F_{\Spd \breve{E}}}{\longrightarrow}\Rep^{\rm{f.d.}}(I_E)).\label{eqn:FSderfrFSIEisomdecomp}
\end{align}
By Lemma \ref{lemm:DULAequivIF}, it reduces to a similar claim for $\Sat(\Hck_{G,[\Spd k/I_E]})$.
Then, by the arguments in \S \ref{sssc:nonderWeil}, \ref{ssc:CohWeil} and \ref{ssc:KosWeil}, all appearing functors can be written as a restriction of the tensor product of a functor and $\rm{id}_{\DDRep^{\perf}(I_E,\Qlb)}$.
Thus, it reduces to a similar claim for $\Sat(\Hck_{G,\Spd k})$.
This claim follows from the proof of Lemma \ref{lemm:FSderfrFSisomdecomp}.
We get the claim.

The first isomorphism in (\ref{eqn:FSderfrFSIEisomdecomp}) is monoidal and $(W_{F}/I_E)$-equivariant by Lemma \ref{lemm:H*dRGammaH*GrmonoidalIF}, the second by the definition of the monoidal and $(W_{F}/I_E)$-equivariant structure on $\cal{S}^{der}_{\Spd \breve{E}}$, the third by an easy, purely algebraic argument, and the fourth by the construction of (non-derived) geometric Satake equivalence.
Therefore, the isomorphism (\ref{eqn:composedwithH*GrGIF}) is monoidal and $(W_{F}/I_E)$-equivariant.
The claim follows since the functor $F_{\Spd \breve{E}}$ of 1-categories is faithful on $\Sat(\Hck_{G,\Spd \breve{E}})$.
\end{proof}
\subsection{Case of tori}\label{ssc:toruscaseEbreve}
Assume $G=T$ is a torus.
By an argument similar to Theorem \ref{thm:mainSpdCtorusbody}, we can show the derived Satake equivalence over $\Spd \breve{E}$ for tori.
\begin{thm}\label{thm:mainSpdEbrevetorusbody}
There is a canonical equivalence
\[
\cal{D}^{\ULA}(\Hck_{T,\Spd \breve{E}},\Qlb)\simeq \rm{Sym}(\wh{\fk{t}}[-2])\lmod^{\perf}(\DDRep(\wh{T}\times I_E,\Qlb)).
\]
in $\Catinf{\Qlb}^{mon,\ B(W_{F}/I_E)}$.

Moreover, there is a canonical equivalence
\[
S_{T,\Spd C}\simeq \cal{S}^{der}_{T,\Spd \breve{E}}\circ fr|_{\Rep^{\rm{f.d.}}(\wh{T}\times I_E,\Qlb)} \colon \Rep^{\rm{f.d.}}(\wh{T}\times I_E,\Qlb)\to \cal{D}^{\ULA}(\Hck_{T,\Spd \breve{E}},\Qlb).
\]
in $\Catinf{\Qlb}^{mon,\ B(W_{F}/I_E)}$.
\end{thm}
\begin{proof}
There is a decomposition
\[
\Hck_{T,\Spd \breve{E}}=\coprod_{\bb{X}_*(T_{\ol{F}})} [\Spd \breve{E}/L^+_{\Spd \breve{E}}T].
\]
Taking into account the $W_E$-action, there is an equivalence 
\begin{align}\label{eqn:HckTbreveEsplitiing}
\cal{D}^{\ULA}(\Hck_{T,\Spd \breve{E}},\Qlb)\simeq \cal{D}^{\ULA}([\Spd \breve{E}/L^+_{\Spd \breve{E}}T],\Qlb)\otimes \bigoplus_{\bb{X}_*(T_{\ol{F}})}\Vect^{\perf}_{\Qlb}    
\end{align}
in $\Catinf{\Qlb}^{mon,\ B(W_{F}/I_E)}$, where $\bigoplus_{\bb{X}_*(T_{\ol{F}})}\Vect^{\perf}_{\Qlb}$ has a canonical structure as an object in $\Catinf{\Qlb}^{mon,\ B(W_{F}/I_E)}$ such that its monoidal structure is $\bb{X}_*(T_{\ol{F}})$-graded and its $W_{F}/I_E$-equivariant structure is the one induced by the $W_F/I_E$-action on $\bb{X}_*(T_{\ol{F}})$.

Let $q_{BL^+T}\colon [\Spd C/L^+_{\Spd C}T]\to [\Spd \breve{E}/L^+_{\Spd \breve{E}}T]$ be the projection.
Similarly to $R\Gamma_{\Spd \breve{E}}$ above, we can define the functor
\begin{align*}
\cal{D}^{\ULA}([\Spd \breve{E}/L^+_{\Spd \breve{E}}T])&\to \rm{Sym}(\wh{\fk{t}}[-2])\lmod(\DDRep(I_E)),\\
A&\mapsto R\Gamma(\Spd C,q_{BL^+T}^*A)
\end{align*}
Let $g\colon [\Spd \breve{E}/L^+_{\Spd \breve{E}}T]\to \Spd \breve{E}$ denote the projection, and let $\cal{IC}_{BL^+T,\Spd \breve{E}}$ denote the essential image of the functor
\[
\DDRep^{\perf}(I_E,\Qlb)\simeq \cal{D}_{\rm{lc}}(\Spd \breve{E},\Qlb)\xrightarrow{g^*}\cal{D}^{\ULA}([\Spd \breve{E}/L^+_{\Spd \breve{E}}T]).
\]
Then by the same argument as Theorem \ref{thm:mainSpdCtorusbody}, we can show full faithfulness and monoidality of $R\Gamma_{\Spd \breve{E}}$ from those of $R\Gamma_{\Spd \breve{E}}|_{\cal{IC}_{BL^+T,\Spd \breve{E}}}$.
Noting that the homomorphism (\ref{eqn:RGammaBL+TKunnethhom}) is compatible with the Weil action on $\Spd C/L^+_{\Spd C}T$, it follows that $R\Gamma_{\Spd \breve{E}}$ has the structure of a morphism in $\Catinf{\Qlb}^{mon,\ B(W_{F}/I_E)}$.

Then by arguing in the same way as Theorem \ref{thm:mainSpdCtorusbody} again, we get an equivalence
\[
\cal{D}^{\ULA}([\Spd \breve{E}/L^+_{\Spd \breve{E}}T])\simeq  \rm{Sym}(\wh{\fk{t}}[-2])\lmod^{\perf}(\DDRep(I_E))
\]
and an equivalence
\[
\cal{D}^{\ULA}(\Hck_{T,\Spd \breve{E}})\simeq  \rm{Sym}(\wh{\fk{t}}[-2])\lmod^{\perf}(\DRep(\wh{T}\times I_E))
\]
in $\Catinf{\Qlb}^{mon,\ B(W_{F}/I_E)}$ by (\ref{eqn:HckTbreveEsplitiing}).

The compatiblity with the non-derived Satake equivalence follows from the same argument as Theorem \ref{thm:mainSpdCtorusbody}.
\end{proof}
\subsection{General case}\label{ssc:generalcaseEbreve}
In this section, we prove the derived Satake equivalence over $\Spd \breve{E}$ for arbitrary connected reductive group $G$ over $F$.
\begin{lemm}
Let $G'$ be a reductive group over $F$ and $T'$ a torus over $F$.
Then the natural functor
\[
\cal{D}^{\ULA}(\Hck_{T',\Spd \breve{E}},\Qlb) \otimes_{\cal{D}_{\rm{lc}}(\Spd \breve{E},\Qlb)} \cal{D}^{\ULA}(\Hck_{G',\Spd \breve{E}},\Qlb)\to \cal{D}^{\ULA}(\Hck_{T'\times G',\Spd \breve{E}},\Qlb)
\]
in $\Catinf{\Qlb}^{mon,\ B(W_{F}/I_E)}$ is an equivalence.
\end{lemm}
\begin{proof}
By Lemma \ref{lemm:DULAequivIF} and Lemma \ref{lemm:HckBIFKunneth}, it suffices to show that the canonical functor
\begin{multline*}
(\cal{D}^{\ULA}(\Hck_{T',\Spd k},\Qlb)\otimes \cal{D}_{\rm{lc}}([\Spd k/I_E],\Qlb)) \\\otimes_{\cal{D}_{\rm{lc}}([\Spd k/I_E],\Qlb)} (\cal{D}^{\ULA}(\Hck_{G',\Spd k},\Qlb)\otimes \cal{D}_{\rm{lc}}([\Spd k/I_E],\Qlb)) )\\
\to \cal{D}^{\ULA}(\Hck_{T'\times G',\Spd k},\Qlb)\otimes \cal{D}_{\rm{lc}}([\Spd k/I_E],\Qlb))
\end{multline*}
is an equivalence.
This follows from Lemma \ref{lemm:exttensorDULA}, replacing $\Spd C$ with $\Spd k$.
\end{proof}
Now by the same argument as Theorem \ref{thm:mainSpdCbody}, we get the derived Satake equivalence over $\Spd \breve{E}$:
\begin{thm}\label{thm:mainSpdbreveEbody}
There is a canonical equivalence
\begin{align}\label{eqn:mainSpdbreveEbody}
\cal{D}^{\ULA}(\Hck_{G,\Spd \breve{E}},\Qlb)\simeq \rm{Sym}(\wh{\fk{g}}[-2])\lmod^{\perf}(\DDRep(\wh{G}\times I_E,\Qlb)).    
\end{align}
in $\Catinf{\Qlb}^{mon,\ B(W_{F}/I_E)}$.

Moreover, there is a canonical equivalence
\[
\cal{S}_{G,\Spd C}\simeq \cal{S}^{der}_{G,\Spd \breve{E}}\circ fr|_{\Rep^{\rm{f.d.}}(\wh{G}\times I_E,\Qlb)}\colon \Rep^{\rm{f.d.}}(\wh{G}\times I_E,\Qlb)\to \cal{D}^{\ULA}(\Hck_{G,\Spd \breve{E}},\Qlb).
\]
in $\Catinf{\Qlb}^{mon,\ B(W_{F}/I_E)}$.
\end{thm}
\section{Derived Satake equivalence over $\Div^1$}\label{sec:derSatoverDiv1}
To prove the derived Satake equivalence over $\Div^1$, we apply the operation $(-)^{W_{F}/I_E}$, defined in \cite[(A.2)]{AG}, to both sides of (\ref{eqn:mainSpdbreveEbody}).
The result of applying $(-)^{W_{F}/I_E}$ to the left hand side is $\cal{O}_{\wh{\fk{g}}^*}^{\shear}\lmod^{fr}(\DDRep(\wh{G}\rtimes W_{F},\Qlb))$ by the standard argument.
Applying it to the right hand side gives $\cal{D}^{\ULA}(\Hck_{G,\Div^1},\Qlb)$ by definition.
Therefore, we get the derived Satake equivalence over $\Div^1$:
\begin{thm}
There is an equivalence of stable monoidal $\infty$-categories
\[
\cal{S}^{der}_{\Div^1}\colon \cal{D}^{\ULA}(\Hck_{G,\Div^1},\Qlb)\simeq \rm{Sym}(\wh{\fk{g}}(-1)[-2])\lmod^{\perf}(\DDRep(\wh{G}\rtimes W_{F})).
\]
\end{thm}
Moreover, let
\[
S_{G,\Div^1}\colon \Rep^{\rm{f.d.}}(\wh{G}\rtimes W_E,\Qlb)\to \cal{D}^{\ULA}(\Hck_{G,\Div^1},\Qlb)
\]
denote the functor induced from the non-derived geometric Satake equivalence over $\Div^1$ in \cite[\S VI]{FS}.
Then by construction, $S_{G,\Div^1}$ is isomorphic to $(S_{G,\Spd \breve{E}})^{W_F/I_E}$, and we have a canonical monoidal equivalence
\[
S_{G,\Div^1}\simeq \cal{S}^{der}_{G,\Div^1}\circ fr|_{\Rep^{\rm{f.p.}}(\wh{G}\rtimes W_{F},\Qlb)}.
\]
\bibliographystyle{test}
\bibliography{Div1SatBib}
\end{document}